\documentclass[10pt]{article}

\usepackage[small]{titlesec}
\usepackage[small,it]{caption}

\usepackage{tocloft}

\usepackage{amsmath}
\usepackage{amsthm}
\usepackage{amssymb}
\usepackage{latexsym}
\usepackage{MnSymbol}

\usepackage{tikz}
\usetikzlibrary{matrix}
\usepackage{tikz-qtree}

\usepackage{color}
\definecolor{lightgray}{rgb}{0.8, 0.8, 0.8}
\definecolor{darkgray}{rgb}{0.7, 0.7, 0.7}
\definecolor{darkblue}{rgb}{0, 0, .4}

\usepackage[bookmarks, backref=page]{hyperref}
\hypersetup{
        colorlinks=true,
        linkcolor=black,
        anchorcolor=black,
        citecolor=black,
        urlcolor=black,
        pdfpagemode=UseThumbs,
        pdftitle={Permutation classes},
        pdfsubject={Combinatorics},
        pdfauthor={Vatter},
}

\renewcommand*{\backref}[1]{}
\renewcommand*{\backrefalt}[4]{%
  \ifcase #1 Not cited.%
  \or    Cited on page~#2.%
  \else   Cited on pages~#2.%
  \fi}

\newcounter{todocounter}

\theoremstyle{plain}
\newtheorem{theorem}{Theorem}[section]
\newtheorem{observation}[theorem]{Observation}
\newtheorem{proposition}[theorem]{Proposition}

\newtheorem{corollary}[theorem]{Corollary}
\newtheorem{conjecture}[theorem]{Conjecture}
\newtheorem{question}[theorem]{Question}

\makeatletter
\newfont{\footsc}{cmcsc10 at 8truept}
\newfont{\footbf}{cmbx10 at 8truept}
\newfont{\footrm}{cmr10 at 10truept}
\renewenvironment{abstract}%
                {
                  \begin{list}{}%
                     {\setlength{\rightmargin}{1in}%
                      \setlength{\leftmargin}{1in}}%
                   \item[]\ignorespaces\begin{small}}%
                 {\end{small}\unskip\end{list}}

\newcommand{\Av}{\operatorname{Av}}
\newcommand{\Sub}{\operatorname{Sub}}
\newcommand{\exop}{\operatorname{ex}}
\newcommand{\Simples}{\operatorname{Si}}
\newcommand{\wt}{\operatorname{wt}}
\newcommand{\gr}{\mathrm{gr}}
\newcommand{\lgr}{\underline{\gr}}
\newcommand{\ugr}{\overline{\gr}}
\newcommand{\C}{\mathcal{C}}
\newcommand{\D}{\mathcal{D}}
\newcommand{\E}{\mathcal{E}}
\newcommand{\G}{\mathcal{G}}
\newcommand{\U}{\mathcal{U}}

\newcommand{\gridded}{\sharp}
\newcommand{\equivfig}{\approx}
\newcommand{\vvg}[1]{\node[fill=lightgray]{#1};}
\newcommand{\fnmatrix}[2]{\text{\begin{footnotesize}$\left(\begin{array}{#1}#2\end{array}\right)$\end{footnotesize}}}
\newcommand{\zpm}{0/\mathord{\pm} 1}
\newcommand{\Grid}{\operatorname{Grid}}
\newcommand{\Geom}{\operatorname{Geom}}
\newcommand{\st}{\::\:}
\newcommand{\verteq}{\rotatebox{90}{$=\,$}}

\titleformat{\section}
	{\large\sc}
	{\thesection.}{1em}{}	

\titleformat{\subsection}
	{\normalsize\sc}
	{\thesubsection.}{1em}{}	

\newpagestyle{main}[\small]{
        \headrule
        \sethead[\usepage][][]
        {\sc Permutation Classes}{}{\usepage}}

\title{{\sc Permutation Classes}\\ {\small To appear in \emph{The Handbook of Enumerative Combinatorics}, published by CRC Press.}}
\author{%
Vincent Vatter\\[-0.25ex]
\small Department of Mathematics\\[-0.5ex]
\small University of Florida\\[-0.5ex]
\small Gainesville, Florida USA\\[-1.5ex]
}

\date{}

\begin{document}
\maketitle

\pagestyle{main}

\begin{abstract}
\end{abstract}

\makeatletter
\@starttoc{toc}
\makeatother

\section{Introduction}

Hints of the study of patterns in permutations date back a century, to Volume I, Section III, Chapter V of MacMahon's 1915 magnum opus \emph{Combinatory Analysis}~\cite{macmahon:combinatory-ana:}. In that work, MacMahon showed that the permutations that can be partitioned into two decreasing subsequences (in other words, the $123$-avoiding permutations) are counted by the Catalan numbers. Twenty years later, Erd\H{o}s and Szekeres~\cite{erdos:a-combinatorial:} proved that every permutation of length at least $(k-1)(\ell-1)+1$ must contain either $12\cdots k$ or $\ell\cdots 21$. Roughly twenty-five years after Erd\H{o}s and Szekeres, Schensted's famous paper~\cite{schensted:longest-increas:} on increasing and decreasing subsequences was published.

Most, however, date the study of permutation classes to 1968, when Knuth published Volume 1 of \emph{The Art of Computer Programming}~\cite{knuth:the-art-of-comp:1}. In Section 2.2.1 of that book, Knuth introduced sorting with stacks and double-ended queues (deques), which leads naturally to the notion of permutation patterns. In particular, Knuth observed that a permutation can be sorted by a stack if and only if it avoids $231$ and showed that these permutations are also counted by the Catalan numbers. He inspired many subsequent papers, including those of Even and Itai~\cite{even:queues-stacks-a} in 1971, Tarjan~\cite{tarjan:sorting-using-n:} in 1972, Pratt~\cite{pratt:computing-permu:} in 1973, Rotem~\cite{rotem:on-a-correspond:} in 1975, and Rogers~\cite{rogers:ascending-seque:} in 1978.

Near the end of his paper, Pratt wrote that
\begin{quote}
From an abstract point of view, the [containment order] on permutations is even more interesting than the networks we were characterizing. This relation seems to be the only partial order on permutations that arises in a simple and natural way, yet it has received essentially no attention to date.
\end{quote}
Pratt's suggestion to study this order in the abstract was taken up a dozen years later by Simion and Schmidt in their seminal 1985 paper ``Restricted permutations''~\cite{simion:restricted-perm:}. The field has continually expanded since then, and is now the topic of the conference  \emph{Permutation Patterns}, held each year since its inauguration (by Albert and Atkinson) at the University of Otago in 2003.

Several overviews of the field have been published, including Kitaev's 494-page compendium \emph{Patterns in Permutations and Words}~\cite{kitaev:patterns-in-per:}, one chapter in B\'ona's undergraduate textbook \emph{A Walk Through Combinatorics}~\cite{bona:a-walk-through-:} and several in his monograph \emph{Combinatorics of Permutations}~\cite{bona:combinatorics-o:}, and Steingr{\'{\i}}msson's survey article~\cite{steingrimsson:some-open-probl:} for the 2013 \emph{British Combinatorial Conference}. In addition, the proceedings of the conference \emph{Permutation Patterns 2007}~\cite{:permutation-pat:} contains surveys by
Albert~\cite{albert:an-introduction:},
Atkinson~\cite{atkinson:permuting-machi:},
B\'ona~\cite{bona:on-three-differ:},
Brignall~\cite{brignall:a-survey-of-sim:},
Kitaev~\cite{kitaev:a-survey-on-par:},
Klazar~\cite{klazar:overview-of-som}, and
Steingr{\'{\i}}msson~\cite{steingrimsson:generalized-per:}
on various aspects of the field.

This survey differs significantly from prior overviews. This is partly because there is a lot of new material to discuss. In particular, Section~\ref{subsec-fox} presents Fox's results on growth rates of principal classes. After that, Section~\ref{sec-structure} is peppered with recent results, while Section~\ref{sec-growth-rates} presents some new results improving on those published. More significantly, this survey differentiates itself from previous summaries of the area by its focus on permutation classes in general.

In order to maintain this focus, a great many beautiful results have been omitted.
Thus, despite the impressive results of Elizalde~\cite{elizalde:the-most-and-th:}, consecutive patterns will not be discussed.
Nor will there be any discussion of mesh patterns, which began as a generalization of the ``generalized'' (now called vincular) patterns introduced by Babson and Steingr{\'{\i}}msson~\cite{babson:generalized-per:} in their classification of Mahonian statistics but have since been shown to be worthy of study on their own via the wonderful Reciprocity Theorem of Br{\"a}nd{\'e}n and Claesson~\cite{branden:mesh-patterns-a:}.
We similarly neglect two questions raised by Wilf in \cite{wilf:the-patterns-of:}: packing densities (which Presutti and Stromquist~\cite{presutti:packing-rates-o:} have shown can be incredibly interesting) and the topology of the poset of permutations (where McNamara and Steingr{\'{\i}}msson~\cite{mcnamara:on-the-topology:} have established some significant results).
Indeed, we even ignore the original application to sorting, despite the deep results of Albert and Bousquet-M\'elou~\cite{albert:permutations-so:} and Pierrot and Rossin~\cite{pierrot:2-stack-pushall:,pierrot:2-stack-sorting:}.
Alas, even this list of omitted topics contains omissions, for which I apologize.

\subsection{Basics}
\label{subsec-basics}

Throughout this survey we think of permutations in one-line notation, so a permutation of length $n$ is simply an ordering of the set $[1,n]=\{1,2,\dots,n\}$ of integers. The permutation $\pi$ \emph{contains} the permutation $\sigma$ of length $k$ if it has a subsequence of length $k$ that is \emph{order isomorphic} to $\sigma$, i.e., that has the same pairwise comparisons as $\sigma$. For example, the subsequence $38514$ is order isomorphic to $25413$, so $25413$ is contained in the permutation $36285714$. Permutation containment is perhaps best seen by drawing the \emph{plot} of a permutation, which is the set of points $\{(i,\pi(i))\}$ as shown on the right of Figure~\ref{fig-perm-contain}. Permutation containment is a partial order on the set of all finite permutations, so if $\sigma$ is contained in $\pi$ we write $\sigma\le\pi$. If $\sigma\not\le\pi$, we say that $\pi$ \emph{avoids} $\sigma$.

\index{permutation containment}
\index{permutation avoidance}

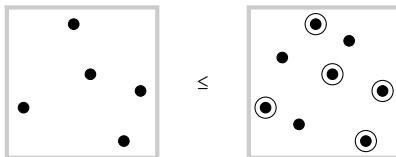
\begin{figure}
\begin{footnotesize}
\begin{center}
	\begin{tikzpicture}[scale=0.2222222222222, baseline=(current bounding box.center)]
		\draw [lightgray, ultra thick, line cap=round] (0,0) rectangle (9,9);
		\draw[fill=black] (1,3) circle (9pt);
		\draw[fill=black] (4,8) circle (9pt);
		\draw[fill=black] (5,5) circle (9pt);
		\draw[fill=black] (7,1) circle (9pt);
		\draw[fill=black] (8,4) circle (9pt);
	\end{tikzpicture}
	\quad
	\begin{tikzpicture}[scale=0.2222222222222, baseline=(current bounding box.center)]
		\node at (0,4) {$\le$};
	\end{tikzpicture}
	\quad
	\begin{tikzpicture}[scale=0.2222222222222, baseline=(current bounding box.center)]
		\draw [lightgray, ultra thick, line cap=round] (0,0) rectangle (9,9);
		\draw[fill=black] (1,3) circle (9pt);
		\draw[fill=black] (2,6) circle (9pt);
		\draw[fill=black] (3,2) circle (9pt);
		\draw[fill=black] (4,8) circle (9pt);
		\draw[fill=black] (5,5) circle (9pt);
		\draw[fill=black] (6,7) circle (9pt);
		\draw[fill=black] (7,1) circle (9pt);
		\draw[fill=black] (8,4) circle (9pt);
		\draw (1,3) circle (18pt);
		\draw (4,8) circle (18pt);
		\draw (5,5) circle (18pt);
		\draw (7,1) circle (18pt);
		\draw (8,4) circle (18pt);
	\end{tikzpicture}
\end{center}
\end{footnotesize}
\caption{The containment order on permutations.}
\label{fig-perm-contain}
\end{figure}

\index{permutation class}

The central objects of study in this survey are \emph{permutation classes}, which are downsets (a.k.a., lower order ideals) of permutations under the containment order. Thus if $\C$ is a class containing the permutation $\pi$ and $\sigma\le\pi$ then $\sigma$ must also lie in $\C$. Given any set $X$ of permutations, one way to obtain a permutation class is to take the \emph{downward closure} of $X$,
$$
\Sub(X)
=
\{\sigma\st\sigma\le\pi\mbox{ for some $\pi\in X$}\}.
$$
There are many other ways to specify a permutation class, for example as the set of permutations sortable by a particular machine, as the set of permutations that can be ``drawn on'' a figure in the plane (considered at the end of this subsection), or by a number of other constructions described in Section~\ref{sec-structure}. However, by far the most common way to define a permutation class is by avoidance:
$$
\Av(B)=\{\pi\st\pi\mbox{ avoids all $\beta\in B$}\}.
$$
If one permutation of $B$ is contained in another then we may remove the larger one without changing the class. Thus we may take $B$ to be an \emph{antichain}, meaning that no element of $B$ contains any others. In the case that $B$ is an antichain we call it the \emph{basis} of this class. The case where $B$ is a singleton has received considerable attention; we call such classes \emph{principal}.

\index{antichain}

The pictorial view shown in Figure~\ref{fig-perm-contain} makes it clear that the containment order has the eight symmetries of the square, which (from the permutation viewpoint) are generated by inverse and reverse. For example, the classes $\Av(132)$, $\Av(213)$, $\Av(231)$, and $\Av(312)$ are all symmetric (isomorphic as partially ordered sets), as are the classes $\Av(123)$ and $\Av(321)$. This shows that there are only $2$ essentially different principal classes avoiding a permutation of length $3$. There are $7$ essentially different principal classes avoiding a permutation of length $4$.

We are frequently interested in the enumeration of permutation classes. Thus letting $\C_n$ denote the set of permutations of length $n$ in the class $\C$, we wish to determine (either exactly or asymptotically) the behavior of the sequence $|\C_0|$, $|\C_1|$, $\dots$ (this sequence is called the \emph{speed} of the class in some contexts). One way of doing this is to explicitly compute the generating function of the class,
$$
\sum_{n\ge 0} |\C_n|x^n=\sum_{\pi\in\C} x^{|\pi|},
$$
where here $|\pi|$ denotes the length of $\pi$. We are often interested in whether this generating function is \emph{rational} (the quotient of two polynomials), \emph{algebraic} (meaning that there is a polynomial $p(x,y)\in\mathbb{Q}[x,y]$ such that $p(x,f(x))=0$), or \emph{$D$-finite} (if its derivatives span a finite dimensional vector space over $\mathbb{Q}(x)$).

In practice it is often quite difficult to compute generating functions of permutation classes, and thus we must content ourselves with the rough asymptotics of $|\C_n|$. To do so we define the \emph{upper} and \emph{lower growth rate} of the class $\C$ by
$$
\ugr(\C)=\limsup_{n\rightarrow\infty} \sqrt[n]{|\C_n|}
\quad\mbox{and}\quad
\lgr(\C)=\liminf_{n\rightarrow\infty} \sqrt[n]{|\C_n|},
$$
respectively. It is not known if these two quantities agree in general.

\index{growth rate}

\begin{conjecture}
\label{conj-gr-exists}
For every permutation class $\C$, $\ugr(\C)=\lgr(\C)$.
\end{conjecture}

If the upper and lower growth rates of a class are equal, we refer to their common value as the \emph{(proper) growth rate} of the class. In order to establish a sufficient condition for the existence of proper growth rates, we need two definitions. Pictorially, the \emph{(direct) sum} of the permutations $\pi$ and $\sigma$, denoted by $\pi\oplus\sigma$, is shown on the left of Figure~\ref{fig-sums}. If $\pi$ has length $k$ and $\sigma$ has length $\ell$, we can also define the sum of $\pi$ and $\sigma$ by
$$
(\pi\oplus\sigma)(i)
=
\left\{\begin{array}{ll}
\pi(i)&\mbox{for $i\in[1,k]$},\\
\sigma(i-k)+k&\mbox{for $i\in[k+1,k+\ell]$}.
\end{array}\right.
$$
The analogous operation depicted on the right of Figure~\ref{fig-sums} is called the \emph{skew sum}.

\index{sum of permutations}
\index{skew sum of permutations}

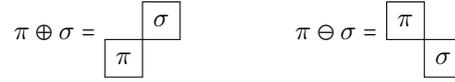
\begin{figure}
\begin{center}
	$\pi\oplus\sigma=$
	\begin{tikzpicture}[scale=0.5, baseline=(current bounding box.center)]
		\draw (0,0) rectangle (1,1);
		\draw (1,1) rectangle (2,2);
		\node at (0.5,0.5) {$\pi$};
		\node at (1.5,1.5) {$\sigma$};
	\end{tikzpicture}
\quad\quad\quad\quad
	$\pi\ominus\sigma=$
	\begin{tikzpicture}[scale=0.5, baseline=(current bounding box.center)]
		\draw (0,1) rectangle (1,2);
		\draw (1,0) rectangle (2,1);
		\node at (0.5,1.5) {$\pi$};
		\node at (1.5,0.5) {$\sigma$};
	\end{tikzpicture}
\end{center}
\caption{The sum and skew sum operations.}
\label{fig-sums}
\end{figure}

The permutation class $\C$ is said to be \emph{sum closed} (respectively, \emph{skew closed}) if $\pi\oplus\sigma\in\C$ (respectively, $\pi\ominus\sigma\in\C$) for every pair of permutations $\sigma,\pi\in\C$. The permutation $\pi$ is further said to be \emph{sum (respectively, skew) decomposable} if it can be expressed as a nontrivial sum (respectively, skew sum) of permutations, and \emph{sum (respectively, skew) indecomposable} otherwise. It is easy to establish that a class is sum (respectively, skew) closed if and only if all of its basis elements are sum (respectively, skew) indecomposable. By observing that a single permutation cannot be both sum and skew decomposable, we obtain the following.

\index{sum closed}

\begin{observation}
\label{obs-principal-sum-closed}
Every principal permutation class is either sum or skew closed.
\end{observation}

The sequence $\{a_n\}$ is said to be \emph{supermultiplicative} if $a_{m+n}\ge a_ma_n$ for all $m$ and $n$. Fekete's Lemma states that if the sequence $\{a_n\}$ is supermultiplicative then $\lim \sqrt[n]{a_n}$ exists and is equal to $\sup \sqrt[n]{a_n}$. A simple application of this lemma gives us the following result.

\begin{proposition}[Arratia~\cite{arratia:on-the-stanley-:}]
\label{prop-arratia-gr}
Every sum closed (or, by symmetry, skew closed) permutation class has a (possibly infinite) growth rate. In particular, this holds for every principal class.
\end{proposition}%
\begin{proof}
Suppose $\C$ is a sum closed permutation class, and thus $\pi\oplus\sigma\in\C_{m+n}$ for all $\pi\in\C_m$ and $\sigma\in\C_n$. Moreover, a given $\tau\in\C_{m+n}$ arises in this way from at most one such pair so $|\C_{m+n}|\ge|\C_m||\C_n|$. This shows that the sequence $\{|\C_n|\}$ is supermultiplicative, and thus $\lim\sqrt[n]{|\C_n|}$ exists by Fekete's Lemma.
\end{proof}

In particular, $\gr(\C)\ge \sqrt[n]{|\C_n|}$ for all sum or skew closed classes $\C$ and all integers $n$, which can be used to establish lower bounds for growth rates of these classes. (We appeal to this fact in the proofs of both Proposition~\ref{prop-valtr} and Theorem~\ref{thm-fox-bound}.)

Continuing our exploration of the sum operation, for every permutation $\pi$ there are unique sum indecomposable permutations $\alpha_1,\dots,\alpha_k$ (called the \emph{sum components} of $\pi$) such that $\pi=\alpha_1\oplus\dots\oplus\alpha_k$. Therefore the permutations in a sum closed class can be viewed as sequences of sum indecomposable permutations. In particular, if a class is sum closed and the generating function for its nonempty sum indecomposable members is $s$, then the generating function for the class is $1/(1-s)$.

\index{sum components}
\index{layered permutation}

A permutation is \emph{layered} if it is the sum of decreasing permutations. Clearly every decreasing permutation is sum indecomposable, so by our previous remarks the generating function for the class of layered permutations is
$$
\frac{1}{1-\frac{x}{1-x}}=\frac{1-x}{1-2x}.
$$
(We could have instead observed that the layered permutations are in bijection with integer compositions.) It is not difficult to show that the basis of the class of layered permutations is $\{231,312\}$.

\index{separable permutation}

A permutation is \emph{separable} if it can be built from the permutation $1$ by repeated sums and skew sums. For example, the permutation $576984132$ is separable:
\begin{eqnarray*}
765984132
&=&32154\ominus 1\ominus 132\\
&=&(321\oplus 21)\ominus 1 \ominus (1\oplus 21)\\
&=&((1\ominus 1\ominus 1)\oplus(1\ominus 1))\ominus 1\ominus (1\oplus(1\ominus 1)).
\end{eqnarray*}
The term separable is due to Bose, Buss, and Lubiw~\cite{bose:pattern-matchin:}, who proved that the separable permutations are $\Av(2413,3142)$, although these permutations first appeared in the much earlier work of Avis and Newborn~\cite{avis:on-pop-stacks-i:}.

Recall that the \emph{little Schr\"oder number} indexed by $n-1$ counts the number of ways to insert parenthesis into a sequence of $n$ symbols in such a way that every pair of parentheses surrounds at least two symbols or parenthesized groups and no parentheses surround the entire sequence (see Stanley~\cite[Exercise 6.39.a]{stanley:enumerative-com:2}), while the \emph{large Schr\"oder numbers} are twice the little Schr\"oder numbers. It follows immediately from our decomposition above that the separable permutations are counted by the large Schr\"oder numbers. For example, the permutation from our example can be encoded by the pair
$$
\ominus, ((\bullet\bullet\bullet)(\bullet\bullet))\bullet(\bullet(\bullet\bullet)),
$$
where $\ominus$ indicates that the outermost division is a skew sum and each pair of parentheses within the expression denotes the opposite type of decomposition as the pair enclosing it. These numbers begin $1$, $2$, $6$, $22$, $90$, $394$, $1806$, $\dots$, and have the algebraic generating function
$$
\frac{3-x-\sqrt{1-6x+x^2}}{2}.
$$
It follows that the growth rate of the separable permutations is approximately $5.83$.


The pictorial perspective presented in Figure~\ref{fig-perm-contain} leads naturally to a geometric treatment of the containment order, in which permutations are objects of a moduli space. From this viewpoint, the fundamental objects are \emph{figures}, which are simply subsets of the plane. Given two figures $\Phi, \Psi\subseteq\mathbb{R}^2$, the figure $\Phi$ is \emph{involved} in the figure $\Psi$, denoted $\Phi\le\Psi$ if there are subsets $A, B\subseteq\mathbb{R}$ and increasing injections $\tau_x\st A\rightarrow\mathbb{R}$ and $\tau_y\st B\rightarrow\mathbb{R}$ such that
$$
\Phi\subseteq A\times B \mbox{ and } \tau(\Phi)\subseteq\Psi,
$$
where $\tau(\Phi)=\{(\tau_x(a), \tau_y(b))\st (a,b)\in\Phi\}$.

The involvement order is a \emph{preorder} on the collection of all figures (it is reflexive and transitive but not necessarily antisymmetric). If $\Phi\le\Psi$ and $\Psi\le\Phi$, then we say that $\Phi$ and $\Psi$ are \emph{equivalent} figures and write $\Phi\equivfig\Psi$. If two figures have only finitely many points, it can be shown that they are equivalent if and only if one can be transformed to the other by stretching and shrinking the axes. Figure~\ref{fig-equivfig} shows two examples of equivalent figures.

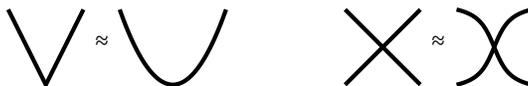
\begin{figure}
\begin{center}
	\begin{tikzpicture}[scale=0.5, baseline=(current bounding box.center)]
		\draw [ultra thick] (-1,2)--(0,0)--(1,2);
	\end{tikzpicture}
	$\equivfig$
	\begin{tikzpicture}[scale=0.5, baseline=(current bounding box.center)]
		\draw[ultra thick, domain=-1.41421356237:1.41421356237,smooth,variable=\x] plot ({\x},{\x*\x});
	\end{tikzpicture}
\quad\quad\quad\quad
	\begin{tikzpicture}[scale=0.5, baseline=(current bounding box.center)]
		\draw [ultra thick] (-1,1)--(1,-1);
		\draw [ultra thick] (-1,-1)--(1,1);
	\end{tikzpicture}
	$\equivfig$
	\begin{tikzpicture}[scale=0.5, baseline=(current bounding box.center)]
		\draw[ultra thick, domain=-1:1,smooth,variable=\x] plot ({\x},{atan(3*\x)/71.5650512});
		\draw[ultra thick, domain=-1:1,smooth,variable=\x] plot ({\x},{-atan(3*\x)/71.5650512});
	\end{tikzpicture}
\end{center}
\caption{Two pairs of equivalent figures.}
\label{fig-equivfig}
\end{figure}

We call a figure \emph{independent} if no two points of the figure lie on a common horizontal or vertical line. It is clear that every finite independent figure is equivalent to the plot of a unique permutation (which we could call its \emph{modulus} if adopting the moduli space perspective). Under this identification, the involvement order on equivalence classes of finite independent figures is isomorphic to the containment order on permutations. Every figure $\Phi\subseteq\mathbb{R}^2$ therefore defines a permutation class,
$$
\Sub(\Phi)
=
\{\pi\st \mbox{$\pi$ is equivalent to a finite independent figure involved in $\Phi$}\},
$$
which we call a \emph{figure class}. If $\pi\in\Sub(\Phi)$ then we say that $\pi$ can be \emph{drawn} on the figure $\Phi$. Two examples of this are shown in Figure~\ref{fig-drawnon}. For example, the class $\Sub(\mathsf{V})$ of permutations that can be drawn on the leftmost figure in Figure~\ref{fig-equivfig} contains all permutations that consist of a decreasing sequence followed by an increasing subsequence. It is not difficult to show that $\Sub(\textsf{V})=\Av(132, 231)$. The class $\Sub(\textsf{X})$ was studied by Elizalde~\cite{elizalde:the-x-class-and:}, who established a bijection between this class and a set of permutations studied by Knuth in Volume 3 of \emph{The Art of Computer Programming}~\cite{knuth:the-art-of-comp:3}. Both of these classes are examples of geometric grid classes, introduced in Section~\ref{subsec-geom-grid}.

\index{figure class}
\index{permutations drawn on an $\textsf{X}$}

\begin{figure}
\begin{center}
\begin{footnotesize}
	\begin{tikzpicture}[scale=0.1, baseline=(current bounding box.center)]
		\draw [thick, line cap=round] (-10,20)--(0,0)--(10,20);
		\useasboundingbox (current bounding box.south west) rectangle (current bounding box.north east);
		\draw[fill=black] (-7,14) circle (20pt) node [left] {$6$};
		\draw[fill=black] (-3,6) circle (20pt) node [left] {$3$};
		\draw[fill=black] (-1,2) circle (20pt) node [left] {$1$};
		\draw[fill=black] (2,4) circle (20pt) node [right] {$2$};
		\draw[fill=black] (4,8) circle (20pt) node [right] {$4$};
		\draw[fill=black] (5.5,11) circle (20pt) node [right] {$5$};
		\draw[fill=black] (8.5,17) circle (20pt) node [right] {$7$};
		\draw[fill=black] (10,20) circle (20pt) node [right] {$8$};
	\end{tikzpicture}
\quad\quad\quad\quad
	\begin{tikzpicture}[scale=0.1, baseline=(current bounding box.center)]
		\draw [thick, line cap=round] (-10,10)--(0,0)--(10,-10);
		\draw [thick, line cap=round] (-10,-10)--(0,0)--(10,10);
		\useasboundingbox (current bounding box.south west) rectangle (current bounding box.north east);
		\draw[fill=black] (-8,-8) circle (20pt) node [below right] {$1$};
		\draw[fill=black] (-6,6) circle (20pt) node [below left] {$8$};
		\draw[fill=black] (-4,4) circle (20pt) node [below left] {$7$};
		\draw[fill=black] (-1,1) circle (20pt) node [below left] {$5$};
		\draw[fill=black] (2,2) circle (20pt) node [above left] {$6$};
		\draw[fill=black] (3,-3) circle (20pt) node [above right] {$4$};
		\draw[fill=black] (5,-5) circle (20pt) node [above right] {$3$};
		\draw[fill=black] (7,-7) circle (20pt) node [above right] {$2$};
	\end{tikzpicture}
\end{footnotesize}
\end{center}
\caption{On the left, the permutation $63124578$ can be drawn on a $\textsf{V}$. On the right, the permutation $18756432$ can be drawn on an $\textsf{X}$.}
\label{fig-drawnon}
\end{figure}
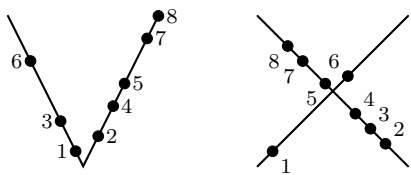

\subsection{Avoiding a permutation of length three}
\label{subsec-len3}

\index{$231$-avoiding permutations}
\index{$321$-avoiding permutations}

Here we briefly consider the two simplest nontrivial principal classes: $\Av(231)$ and $\Av(321)$. As mentioned in the introduction, both of these classes are counted by the Catalan numbers (and thus have growth rates of $4$), so there ought to be nice bijections between them and Dyck paths. Indeed there are, as shown in Figure~\ref{fig-231-321-bij}, and it is almost the same bijection for both classes.

\begin{figure}
\begin{center}
	\begin{tikzpicture}[scale=0.32, baseline=(current bounding box.center)]
		\foreach \i in {0,1,2,3,4,5,6,7,8}{
			\draw[color=lightgray] (0,\i)--(8,\i);
		}
		\foreach \i in {0,1,2,3,4,5,6,7,8}{
			\draw[color=lightgray] (\i,0)--(\i,8);
		}
		\draw[very thick, line cap=round, color=lightgray] (0,0)--(8,8);
		\draw[very thick, line cap=round] (0,0)--(3,0)--(3,2)--(5,2)--(5,3)--(6,3)--(6,6)--(8,6)--(8,8);
		\draw[fill=black] (0.5,5.5) circle (5pt); 
		\draw[fill=black] (1.5,1.5) circle (5pt); 
		\draw[fill=black] (2.5,0.5) circle (5pt);
		\draw[fill=black] (3.5,4.5) circle (5pt); 
		\draw[fill=black] (4.5,2.5) circle (5pt);
		\draw[fill=black] (5.5,3.5) circle (5pt);
		\draw[fill=black] (6.5,7.5) circle (5pt); 
		\draw[fill=black] (7.5,6.5) circle (5pt);
	\end{tikzpicture}
\quad\quad
	\begin{tikzpicture}[scale=0.32, baseline=(current bounding box.center)]
		\foreach \i in {0,1,2,3,4,5,6,7,8}{
			\draw[color=lightgray] (0,\i)--(8,\i);
		}
		\foreach \i in {0,1,2,3,4,5,6,7,8}{
			\draw[color=lightgray] (\i,0)--(\i,8);
		}
		\draw[very thick, line cap=round, color=lightgray] (0,0)--(8,8);
		\draw[very thick, line cap=round] (0,0)--(3,0)--(3,2)--(5,2)--(5,3)--(6,3)--(6,6)--(8,6)--(8,8);
		\draw[fill=black] (0.5,1.5) circle (5pt); 
		\draw[fill=black] (1.5,4.5) circle (5pt); 
		\draw[fill=black] (2.5,0.5) circle (5pt);
		\draw[fill=black] (3.5,5.5) circle (5pt); 
		\draw[fill=black] (4.5,2.5) circle (5pt);
		\draw[fill=black] (5.5,3.5) circle (5pt);
		\draw[fill=black] (6.5,7.5) circle (5pt); 
		\draw[fill=black] (7.5,6.5) circle (5pt);
	\end{tikzpicture}
\end{center}
\caption{The construction of Dyck paths from permutations avoiding $231$ (left) and $321$ (right).}
\label{fig-231-321-bij}
\end{figure}
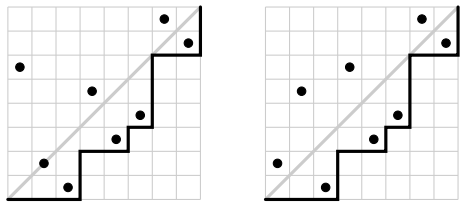

The bijection depicted in Figure~\ref{fig-231-321-bij} is originally due to Knuth~\cite{knuth:the-art-of-comp:1,knuth:the-art-of-comp:3}, though it was Krattenthaler~\cite{krattenthaler:permutations-wi:} who gave the nonrecursive formulation we describe. In this bijection, we plot a permutation and then draw the Dyck path that lies just below its right-to-left minima (an entry is a \emph{right-to-left minimum} if it is smaller than every entry to its right). It is then possible to show that the positions of the non-right-to-left minima can be determined from this Dyck path. There are at least eight more essentially different bijections between these two classes. For the rest of them we refer to Claesson and Kitaev's survey~\cite{claesson:classification-:}.

\index{right-to-left minimum}

Despite the elegance of this bijection, $\Av(231)$ and $\Av(321)$ are \emph{very} different classes. Indeed, Miner and Pak~\cite{miner:the-shape-of-ra:} make a compelling argument that there is no truly ``ultimate'' bijection between these two classes. As discussed in Section~\ref{sec-structure}, $\Av(231)$ has only countably many subclasses, while $\Av(321)$ has uncountably many. Indeed, Albert and Atkinson~\cite{albert:simple-permutat:} have proved that every proper subclass of $\Av(231)$ has a rational generating function, while a simple counting argument shows that $\Av(321)$ contains subclasses whose generating functions aren't even $D$-finite.

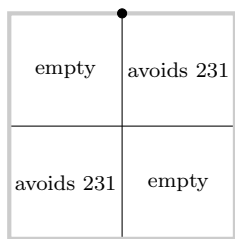
\begin{figure}
\begin{footnotesize}
\begin{center}
	\begin{tikzpicture}[scale=0.3, baseline=(current bounding box.center)]
		\draw (5,0)--(5,10);
		\draw (0,5)--(10,5);
		\draw [lightgray, ultra thick, line cap=round] (0,0) rectangle (10,10);
		\draw[fill=black] (5,10) circle (5.71428571429pt);
		\node at (2.5,2.5) {avoids $231$};
		\node at (2.5,7.5) {empty};
		\node at (7.5,2.5) {empty};
		\node at (7.5,7.5) {avoids $231$};
	\end{tikzpicture}
\end{center}
\end{footnotesize}
\caption{A $231$-avoiding permutation.}
\label{fig-231-decomp}
\end{figure}

Another way to see the contrast between $\Av(231)$ and $\Av(321)$ is to consider the decomposition of a $231$-permutation shown in Figure~\ref{fig-231-decomp}. As this figure shows, all entries to the left of the maximum in a $231$-avoiding permutation must lie below all entries to the right of the maximum. Going in the other direction, every permutation constructed in this manner avoids $231$. Thus if we let $f$ denote the generating function for the $231$-avoiding permutations (including the empty permutation), we see immediately that
$$
f=xf^2+1,
$$
so $f$ is the generating function for the Catalan numbers.

Despite the similarities of Figure~\ref{fig-231-321-bij}, there is no analogue of Figure~\ref{fig-231-decomp} for the $321$-avoiding permutations. Perhaps the best structural description of this class one can give is the following.

\begin{proposition}
\label{prop-321-merge}
The class of permutations that can be expressed as the union of two increasing subsequences is $\Av(321)$.
\end{proposition}
\begin{proof}
Label each entry of $\pi\in\Av(321)$ by the length of the longest decreasing subsequence that ends at that entry. Clearly we have only used the labels $1$ and $2$, and for each value of $i$ the entries labeled by $i$ form an increasing subsequence.
\end{proof}

\begin{figure}
\begin{center}
	\begin{footnotesize}
	\begin{tikzpicture}[scale=0.2, baseline=(current bounding box.center)]
		\draw [thick, line cap=round] (0,0)--(15,15);
		\draw [thick, line cap=round] (10,0)--(25,15);
		\useasboundingbox (current bounding box.south west) rectangle (current bounding box.north east);
		\draw[fill=black] (2,2) circle (10pt) node [above left] {$2$};
		\draw[fill=black] (6,6) circle (10pt) node [above left] {$5$};
		\draw[fill=black] (10,10) circle (10pt) node [above left] {$7$};
		\draw[fill=black] (11,1) circle (10pt) node [below right] {$1$};
		\draw[fill=black] (12,12) circle (10pt) node [above left] {$9$};
		\draw[fill=black] (13,3) circle (10pt) node [below right] {$3$};
		\draw[fill=black] (14,4) circle (10pt) node [below right] {$4$};
		\draw[fill=black] (18,8) circle (10pt) node [below right] {$6$};
		\draw[fill=black] (21,11) circle (10pt) node [below right] {$8$};
	\end{tikzpicture}
	\end{footnotesize}
\quad\quad
	\begin{tikzpicture}[scale=0.2, baseline=(current bounding box.center)]
		\draw [thick, line cap=round] (0,0)--(15,15);
		\draw [thick, line cap=round] (7,0)--(22,15);
		\draw [thick, line cap=round] (12,0)--(27,15);
		\draw [<->] (11,10)--(13.5,7.5);
		\draw [<->] (14.5,6.5)--(16,5);
		\node [above right] at (12.25,8.75) {$d_1$};
		\node [above right] at (15.25,5.75) {$d_2$};
	\end{tikzpicture}
\end{center}
\caption{On the left, a drawing of the permutation $257193468$ on two parallel lines. On the right, three parallel lines.}
\label{fig-321-lines}
\end{figure}
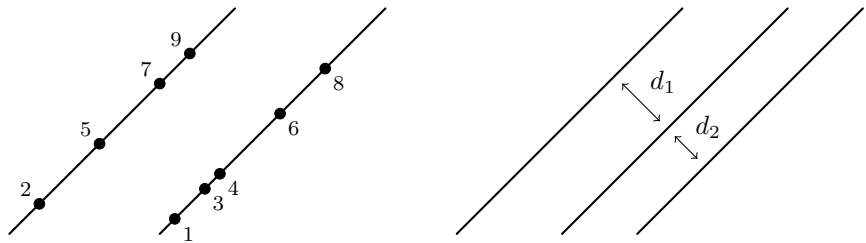

In his thesis, Waton proved the following geometric version of Proposition~\ref{prop-321-merge} (illustrated on the left of Figure~\ref{fig-321-lines}).

\begin{proposition}[Waton~\cite{waton:on-permutation-:}]
\label{prop-321-lines}
The class of permutations that can be drawn on any two parallel lines of positive slope is $\Av(321)$.
\end{proposition}

%

Our proof of Proposition~\ref{prop-321-merge} generalizes easily to show that every $k\cdots 21$-avoiding permutation can be expressed as the union of $k-1$ increasing subsequences. However, Waton showed that Proposition~\ref{prop-321-lines} does not extend in the same way. Consider three parallel lines at distances $d_1$ and $d_2$ from each other as depicted on the right of Figure~\ref{fig-321-lines}. Clearly every permutation that can be drawn on these lines avoids $4321$, but Waton proved that every different choice of the ratio $d_1/d_2$ leads to a different proper subclass of $\Av(4321)$.


%
%
%
%
%
%

\subsection{Wilf-equivalence}
\label{subsec-wilf-equiv}

As we saw in the last subsection, the permutations $231$ and $321$ are equally avoided. Instead of being a coincidence of small numbers, this turns out to be a common phenomenon. The classes $\C$ and $\D$ are said to be \emph{Wilf-equivalent} if they are equinumerous, i.e., if $|\C_n|=|\D_n|$ for every $n$. In this subsection we are primarily concerned with Wilf-equivalence of principal classes, so instead of saying that $\Av(\beta)$ and $\Av(\gamma)$ are Wilf-equivalent we simply say that $\beta$ and $\gamma$ are Wilf-equivalent.

\index{Wilf-equivalence}
\index{full rook placements}

Large swaths of Wilf-equivalences can be explained via a much stronger notion of equivalence, though doing so requires a shift of perspective to \emph{full rook placements (frps)}. These consist of a Ferrers board with a designated set of cells, called \emph{rooks} (here drawn as dots), so that each row and column contains precisely one rook. For example, both objects in Figure~\ref{fig-danger-zone} are frps. (We use French/Cartesian indexing throughout this survey, so for us a Ferrers board is a left-justified array of cells in which the number of cells in each row is at least the number of cells in the row above.)

There is a natural partial order on the set of all frps: given frps $R$ and $S$, we say that $R$ is \emph{contained} in $S$ if $R$ can be obtained from $S$ by deleting rows and columns. Furthermore, this partial order generalizes the containment order on permutations. To make this precise, let us call a frp \emph{square} if the underlying Ferrers board is square and say that the permutation $\pi$ of length $n$ corresponds to the $n\times n$ frp with rooks in the cells $(i,\pi(i))$ for every $i$. When restricted to square frps, the containment order on frps is equivalent to the containment order on the corresponding permutations.

Because of this correspondence between permutations and square frps, we say that a frp \emph{contains} the permutation $\sigma$ if it contains the square frp corresponding to $\sigma$ and otherwise say that it \emph{avoids} $\sigma$. Observe that the entire square containing $\sigma$ must be contained within the frp; for example, the frp below contains $12$ but avoids $21$.
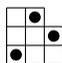
\begin{center}
	\begin{tikzpicture}[scale=0.25, baseline=(current bounding box.center)]
		\foreach \i in {0,1,2}{
			\draw (0,\i)--(3,\i);
		}
		\draw (0,3)--(2,3);
		\foreach \i in {0,1,2}{
			\draw (\i,0)--(\i,3);
		}
		\draw (3,0)--(3,2);
		\draw[fill=black] (0.5,0.5) circle (8pt);
		\draw[fill=black] (1.5,2.5) circle (8pt);
		\draw[fill=black] (2.5,1.5) circle (8pt);
	\end{tikzpicture}
\end{center}
We say that the permutations $\beta$ and $\gamma$ are \emph{shape-Wilf-equivalent} if given any shape $\lambda$, the number of $\beta$-avoiding frps of shape $\lambda$ is the same as the number of $\gamma$-avoiding frps of shape $\lambda$. Shape-Wilf-equivalence implies Wilf-equivalence by considering only square shapes, but it also implies quite a bit more.

\index{shape-Wilf-equivalence}

\begin{proposition}[Babson and West~\cite{babson:the-permutation:}]
\label{prop-swe-sums}
If $\beta$ and $\gamma$ are shape-Wilf-equivalent, then for every permutation $\delta$, $\beta\oplus\delta$ and $\gamma\oplus\delta$ are also shape-Wilf-equivalent.
\end{proposition}
\begin{proof}
Suppose that there is a bijection between $\beta$-avoiding and $\gamma$-avoiding frps of every shape. Now fix a shape $\lambda$. We construct a bijection between $\beta\oplus\delta$-avoiding frps and $\gamma\oplus\delta$-avoiding frps of shape $\lambda$. Let $R$ be a $\beta\oplus\delta$-avoiding frp of shape $\lambda$.

\begin{figure}
$$
\begin{array}{ccc}
	\begin{tikzpicture}[scale=0.4, baseline=(current bounding box.center)]
		\draw [fill=lightgray, color=lightgray] (0,0)--(0,5)--(1,5)--(1,2)--(4,2)--(4,1)--(5,1)--(5,0)--(0,0);
		\foreach \i in {0,1,2,3,4,5}{
			\draw (0,\i)--(7,\i);
		}
		\draw (0,6)--(5,6);
		\draw (0,7)--(3,7);
		\foreach \i in {0,1,2,3}{
			\draw (\i,0)--(\i,7);
		}
		\foreach \i in {4,5}{
			\draw (\i,0)--(\i,6);
		}
		\foreach \i in {6,7}{
			\draw (\i,0)--(\i,5);
		}
		\draw[fill=black] (0.5,4.5) circle (4pt);
		\draw[fill=black] (1.5,5.5) circle (4pt);
		\draw[fill=black] (2.5,6.5) circle (4pt);
		\draw[fill=black] (3.5,0.5) circle (4pt);
		\draw[fill=black] (4.5,2.5) circle (4pt);
		\draw[fill=black] (5.5,1.5) circle (4pt);
		\draw[fill=black] (6.5,3.5) circle (4pt);
	\end{tikzpicture}
	&
	\begin{tikzpicture}[scale=0.4, baseline=(current bounding box.center)]
		\draw[->] (0,0)--(1,0);
	\end{tikzpicture}
	&
	\begin{tikzpicture}[scale=0.4, baseline=(current bounding box.center)]
		\draw [fill=lightgray, color=lightgray] (0,0)--(0,2)--(1,2)--(1,1)--(2,1)--(2,0)--(0,0);
		\draw (0,0)--(2,0);
		\draw (0,1)--(2,1);
		\draw (0,2)--(1,2);
		\draw (0,0)--(0,2);
		\draw (1,0)--(1,2);
		\draw (2,0)--(2,1);
		\draw[fill=black] (0.5,1.5) circle (4pt);
		\draw[fill=black] (1.5,0.5) circle (4pt);
	\end{tikzpicture}
\end{array}
$$
\caption{On the left, a frp with its danger zone for the permutation $12$ shaded. On the right, the frp that the danger zone contracts to after removing rookless rows and columns.}
\label{fig-danger-zone}
\end{figure}
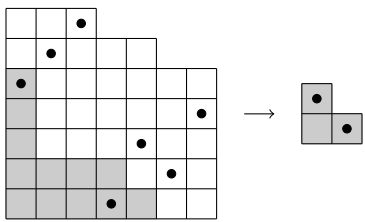

We call a cell of $R$ \emph{dangerous} if there is a copy of $\delta$ completely contained in the region above and to the right of the cell. The entire set of dangerous cells is called the \emph{danger zone} (see Figure~\ref{fig-danger-zone}). The danger zone forms a (possibly empty) Ferrers board nestled in the bottom-left corner of $R$. Ignoring the rookless rows and columns of the danger zone, we thus obtain a $\beta$-avoiding frp. We may then use the bijection between $\beta$-avoiding and $\gamma$-avoiding frps of that shape to produce a $\gamma\oplus\delta$-avoiding frp of shape $\lambda$, as desired.
\end{proof}

Only two shape-Wilf-equivalence results are known. We sketch bijective proofs for both.

\begin{theorem}[Backelin, West, and Xin~\cite{backelin:wilf-equivalenc:}]
\label{thm-inc-dec-wilf-equiv}
For every value of $k$, the permutations $k\cdots 21$ and $12\cdots k$ are shape-Wilf-equivalent.
\end{theorem}

By Proposition~\ref{prop-swe-sums} this implies that the permutations $k\cdots 21\oplus\beta$ and $12\cdots k\oplus\beta$ are Wilf-equivalent for every permutation $\beta$. Special cases of Theorem~\ref{thm-inc-dec-wilf-equiv} had been established before the general case above; West proved the $k=2$ case in his thesis~\cite{west:permutations-wi:}, while the $k=3$ case was shown by Babson and West~\cite{babson:the-permutation:}. Krattenthaler~\cite{krattenthaler:growth-diagrams:} was the first to give a truly nice bijection between $k\cdots 21$-avoiding and $12\cdots k$-avoiding frps, using the growth diagrams of Fomin~\cite{fomin:the-generalized:}.

\index{growth diagram}

\begin{figure}
\begin{scriptsize}
\begin{center}
	\begin{tikzpicture}[scale=0.69, baseline=(current bounding box.center)]
		\foreach \i in {0,1,2,3,4,5}{
			\draw (0,\i)--(7,\i);
		}
		\draw (0,6)--(5,6);
		\draw (0,7)--(3,7);
		\foreach \i in {0,1,2,3}{
			\draw (\i,0)--(\i,7);
		}
		\foreach \i in {4,5}{
			\draw (\i,0)--(\i,6);
		}
		\foreach \i in {6,7}{
			\draw (\i,0)--(\i,5);
		}
		\draw[fill=black] (0.5,0.5) circle (2.89855072464pt);
		\draw[fill=black] (1.5,6.5) circle (2.89855072464pt);
		\draw[fill=black] (2.5,2.5) circle (2.89855072464pt);
		\draw[fill=black] (3.5,4.5) circle (2.89855072464pt);
		\draw[fill=black] (4.5,5.5) circle (2.89855072464pt);
		\draw[fill=black] (5.5,1.5) circle (2.89855072464pt);
		\draw[fill=black] (6.5,3.5) circle (2.89855072464pt);
		\foreach \i in {0,1,2,3,4,5,6,7}{
			\node[above right] at (\i,0) {$0$};
		}
		\node[above right] at (0,1) {$0$};
		\foreach \i in {1,2,3,4,5,6,7}{
			\node[above right] at (\i,1) {$1$};
		}
		\node[above right] at (0,2) {$0$};
		\foreach \i in {1,2,3,4,5}{
			\node[above right] at (\i,2) {$1$};
		}
		\foreach \i in {6,7}{
			\node[above right] at (\i,2) {$2$};
		}
		\node[above right] at (0,3) {$0$};
		\foreach \i in {1,2}{
			\node[above right] at (\i,3) {$1$};
		}
		\foreach \i in {3,4,5}{
			\node[above right] at (\i,3) {$2$};
		}
		\foreach \i in {6,7}{
			\node[above right] at (\i,3) {$21$};
		}
		\node[above right] at (0,4) {$0$};
		\foreach \i in {1,2}{
			\node[above right] at (\i,4) {$1$};
		}
		\foreach \i in {3,4,5}{
			\node[above right] at (\i,4) {$2$};
		}
		\node[above right] at (6,4) {$21$};
		\node[above right] at (7,4) {$31$};
		\node[above right] at (0,5) {$0$};
		\foreach \i in {1,2}{
			\node[above right] at (\i,5) {$1$};
		}
		\node[above right] at (3,5) {$2$};
		\node[above right] at (4,5) {$3$};
		\node[above right] at (5,5) {$3$};
		\node[above right] at (6,5) {$31$};
		\node[above right] at (7,5) {$32$};
		\node[above right] at (0,6) {$0$};
		\node[above right] at (1,6) {$1$};
		\node[above right] at (2,6) {$1$};
		\node[above right] at (3,6) {$2$};
		\node[above right] at (4,6) {$3$};
		\node[above right] at (5,6) {$4$};
		\node[above right] at (0,7) {$0$};
		\node[above right] at (1,7) {$1$};
		\node[above right] at (2,7) {$2$};
		\node[above right] at (3,7) {$21$};
	\end{tikzpicture}
	\quad
	\begin{tikzpicture}[scale=0.5, baseline=(current bounding box.center)]
		\draw[<->] (0,0)--(1,0);
	\end{tikzpicture}
	\quad
	\begin{tikzpicture}[scale=0.69, baseline=(current bounding box.center)]
		\foreach \i in {0,1,2,3,4,5}{
			\draw (0,\i)--(7,\i);
		}
		\draw (0,6)--(5,6);
		\draw (0,7)--(3,7);
		\foreach \i in {0,1,2,3}{
			\draw (\i,0)--(\i,7);
		}
		\foreach \i in {4,5}{
			\draw (\i,0)--(\i,6);
		}
		\foreach \i in {6,7}{
			\draw (\i,0)--(\i,5);
		}
		\draw[fill=black] (0.5,5.5) circle (2.89855072464pt);
		\draw[fill=black] (1.5,2.5) circle (2.89855072464pt);
		\draw[fill=black] (2.5,6.5) circle (2.89855072464pt);
		\draw[fill=black] (3.5,1.5) circle (2.89855072464pt);
		\draw[fill=black] (4.5,0.5) circle (2.89855072464pt);
		\draw[fill=black] (5.5,4.5) circle (2.89855072464pt);
		\draw[fill=black] (6.5,3.5) circle (2.89855072464pt);
		\node[above right] at (7,0) {$0$};
		\node[above right] at (7,1) {$1$};
		\node[above right] at (7,2) {$11$};
		\node[above right] at (7,3) {$21$};
		\node[above right] at (7,4) {$211$};
		\node[above right] at (7,5) {$221$};
		\node[above right] at (6,5) {$211$};
		\node[above right] at (5,5) {$111$};
		\node[above right] at (5,6) {$1111$};
		\node[above right] at (4,6) {$111$};
		\node[above right] at (3,6) {$11$};
		\node[above right] at (3,7) {$21$};
		\node[above right] at (2,7) {$11$};
		\node[above right] at (1,7) {$1$};
		\node[above right] at (0,7) {$0$};
	\end{tikzpicture}
\end{center}
\end{scriptsize}
\caption{Krattenthaler's bijection, shown here converting a $321$-avoiding frp to a $123$-avoiding frp.}
\label{fig-krattenthaler-bij}
\end{figure}
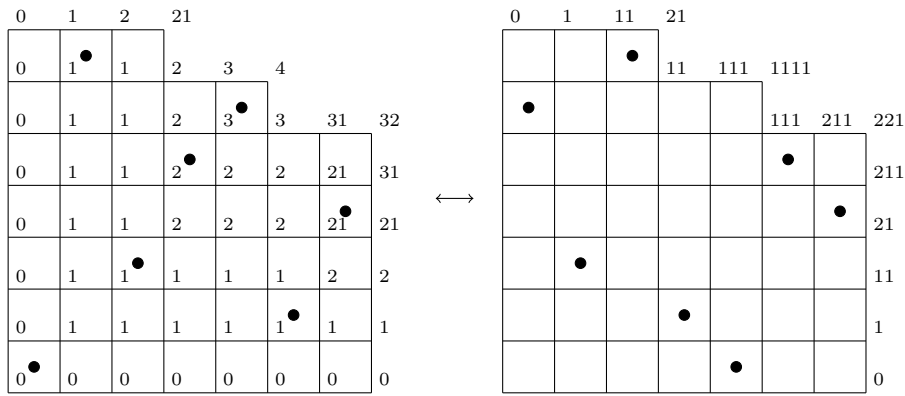

Consider the labeled frp shown on the left of Figure~\ref{fig-krattenthaler-bij}. In this diagram every corner is labeled by an integer partition (the labels lie slightly above and to the right of the corners). This partition (written in shortened notation) represents the shape of the resulting standard Young tableaux when the partial permutation lying below and to the left of that corner is passed through the Robinson--Schensted correspondence. As one can imagine, there are a variety of restrictions on such growth diagrams. For example, the labels along the northwest-southeast boundary form an \emph{oscillating tableaux}, meaning that each partition differs from the previous partition by precisely one box.

\index{oscillating tableaux}

We will not concern ourselves with the other requirements here, except to call labeled tableaux formed in this way \emph{valid}. One of the facts that Krattenthaler established is that given the oscillating tableaux along the border of a valid growth diagram, one can reconstruct all of the interior corner labels, and from these, the position of the rooks in the corresponding frp. Thus we need only concern ourselves with the labels of this border. It follows from Fomin's work that by conjugating the partitions labeling the border of a valid growth diagram one obtains another valid growth diagram.

Now recall that if a frp avoids $k\cdots 21$, its shape under Robinson--Schensted has fewer than $k$ parts (Schensted's Theorem~\cite{schensted:longest-increas:}). Equivalently, all parts along the border of its valid growth diagram will have fewer than $k$ parts. As Krattenthaler observed, this combination of Fomin's growth diagrams and Schensted's Theorem shows that if we conjugate each of those partitions, we obtain the border of a valid growth diagram in which no part is $k$ or larger. Therefore we have produced a unique frp avoiding $12\cdots k$, giving a bijection between $k\dots 21$-avoiding and $12\cdots k$-avoiding frps. Krattenthaler's bijection not only proves Theorem~\ref{thm-inc-dec-wilf-equiv} but also its analogue for involutions, first established by Bousquet-M\'elou and Steingr{\'{\i}}msson~\cite{bousquet-melou:decreasing-subs:}.

We move on to the second of two known examples of shape-Wilf-equivalence.

\begin{theorem}[Stankova and West~\cite{stankova:a-new-class-of-:}]
\label{thm-231-312-wilf-equiv}
The permutations $231$ and $312$ are shape-Wilf-equivalent.
\end{theorem}

Motivated by Fomin's growth diagrams, Bloom and Saracino~\cite{bloom:a-simple-biject:} gave a beautiful bijective proof of Theorem~\ref{thm-231-312-wilf-equiv}. Consider a $231$-avoiding frp, as shown in the upper-left diagram of Figure~\ref{fig-bloom-saracino}. This time, instead of partitions, we label each corner of the northwest-southeast border by the length of the longest increasing sequence in the partial permutation lying below and to the left of that corner. We then read these labels to construct a labeled Dyck path, as shown on the right of this figure.

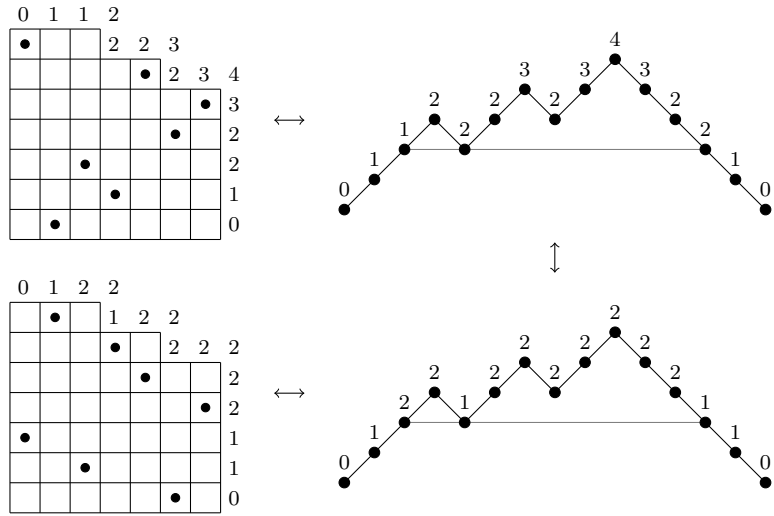
\begin{figure}
\begin{footnotesize}
$$
\begin{array}{ccc}
	\begin{tikzpicture}[scale=0.4, baseline=(current bounding box.center)]
		\foreach \i in {0,1,2,3,4,5}{
			\draw (0,\i)--(7,\i);
		}
		\draw (0,6)--(5,6);
		\draw (0,7)--(3,7);
		\foreach \i in {0,1,2,3}{
			\draw (\i,0)--(\i,7);
		}
		\foreach \i in {4,5}{
			\draw (\i,0)--(\i,6);
		}
		\foreach \i in {6,7}{
			\draw (\i,0)--(\i,5);
		}
		\draw[fill=black] (0.5,6.5) circle (4pt);
		\draw[fill=black] (1.5,0.5) circle (4pt);
		\draw[fill=black] (2.5,2.5) circle (4pt);
		\draw[fill=black] (3.5,1.5) circle (4pt);
		\draw[fill=black] (4.5,5.5) circle (4pt);
		\draw[fill=black] (5.5,3.5) circle (4pt);
		\draw[fill=black] (6.5,4.5) circle (4pt);
		\node[above right] at (0,7) {$0$};
		\node[above right] at (1,7) {$1$};
		\node[above right] at (2,7) {$1$};
		\node[above right] at (3,7) {$2$};
		\node[above right] at (3,6) {$2$};
		\node[above right] at (4,6) {$2$};
		\node[above right] at (5,6) {$3$};
		\node[above right] at (5,5) {$2$};
		\node[above right] at (6,5) {$3$};
		\node[above right] at (7,5) {$4$};
		\node[above right] at (7,4) {$3$};
		\node[above right] at (7,3) {$2$};
		\node[above right] at (7,2) {$2$};
		\node[above right] at (7,1) {$1$};
		\node[above right] at (7,0) {$0$};
	\end{tikzpicture}
	&
	\begin{tikzpicture}[scale=0.4, baseline=(current bounding box.center)]
		\draw[<->] (0,0)--(1,0);
	\end{tikzpicture}
	&
	\begin{tikzpicture}[scale=0.4, baseline=(current bounding box.center)]
		\draw[color=gray] (2,2)--(12,2);
		\draw (0,0)--(3,3)--(4,2)--(6,4)--(7,3)--(9,5)--(14,0);
		\draw[fill=black] (0,0) circle (5pt);
		\draw[fill=black] (1,1) circle (5pt);
		\draw[fill=black] (2,2) circle (5pt);
		\draw[fill=black] (3,3) circle (5pt);
		\draw[fill=black] (4,2) circle (5pt);
		\draw[fill=black] (5,3) circle (5pt);
		\draw[fill=black] (6,4) circle (5pt);
		\draw[fill=black] (7,3) circle (5pt);
		\draw[fill=black] (8,4) circle (5pt);
		\draw[fill=black] (9,5) circle (5pt);
		\draw[fill=black] (10,4) circle (5pt);
		\draw[fill=black] (11,3) circle (5pt);
		\draw[fill=black] (12,2) circle (5pt);
		\draw[fill=black] (13,1) circle (5pt);
		\draw[fill=black] (14,0) circle (5pt);
		\node[above=2pt] at (0,0) {$0$};
		\node[above=2pt] at (1,1) {$1$};
		\node[above=2pt] at (2,2) {$1$};
		\node[above=2pt] at (3,3) {$2$};
		\node[above=2pt] at (4,2) {$2$};
		\node[above=2pt] at (5,3) {$2$};
		\node[above=2pt] at (6,4) {$3$};
		\node[above=2pt] at (7,3) {$2$};
		\node[above=2pt] at (8,4) {$3$};
		\node[above=2pt] at (9,5) {$4$};
		\node[above=2pt] at (10,4) {$3$};
		\node[above=2pt] at (11,3){$2$};
		\node[above=2pt] at (12,2) {$2$};
		\node[above=2pt] at (13,1) {$1$};
		\node[above=2pt] at (14,0) {$0$};
	\end{tikzpicture}
\\
&&
	\begin{tikzpicture}[scale=0.4, baseline=(current bounding box.center)]
		\draw[<->] (0,0)--(0,1);
	\end{tikzpicture}
\\
	\begin{tikzpicture}[scale=0.4, baseline=(current bounding box.center)]
		\foreach \i in {0,1,2,3,4,5}{
			\draw (0,\i)--(7,\i);
		}
		\draw (0,6)--(5,6);
		\draw (0,7)--(3,7);
		\foreach \i in {0,1,2,3}{
			\draw (\i,0)--(\i,7);
		}
		\foreach \i in {4,5}{
			\draw (\i,0)--(\i,6);
		}
		\foreach \i in {6,7}{
			\draw (\i,0)--(\i,5);
		}
		\draw[fill=black] (0.5,2.5) circle (4pt);
		\draw[fill=black] (1.5,6.5) circle (4pt);
		\draw[fill=black] (2.5,1.5) circle (4pt);
		\draw[fill=black] (3.5,5.5) circle (4pt);
		\draw[fill=black] (4.5,4.5) circle (4pt);
		\draw[fill=black] (5.5,0.5) circle (4pt);
		\draw[fill=black] (6.5,3.5) circle (4pt);
		\node[above right] at (0,7) {$0$};
		\node[above right] at (1,7) {$1$};
		\node[above right] at (2,7) {$2$};
		\node[above right] at (3,7) {$2$};
		\node[above right] at (3,6) {$1$};
		\node[above right] at (4,6) {$2$};
		\node[above right] at (5,6) {$2$};
		\node[above right] at (5,5) {$2$};
		\node[above right] at (6,5) {$2$};
		\node[above right] at (7,5) {$2$};
		\node[above right] at (7,4) {$2$};
		\node[above right] at (7,3) {$2$};
		\node[above right] at (7,2) {$1$};
		\node[above right] at (7,1) {$1$};
		\node[above right] at (7,0) {$0$};
	\end{tikzpicture}
	&
	\begin{tikzpicture}[scale=0.4, baseline=(current bounding box.center)]
		\draw[<->] (0,0)--(1,0);
	\end{tikzpicture}
	&
	\begin{tikzpicture}[scale=0.4, baseline=(current bounding box.center)]
		\draw[color=gray] (2,2)--(12,2);
		\draw (0,0)--(3,3)--(4,2)--(6,4)--(7,3)--(9,5)--(14,0);
		\draw[fill=black] (0,0) circle (5pt);
		\draw[fill=black] (1,1) circle (5pt);
		\draw[fill=black] (2,2) circle (5pt);
		\draw[fill=black] (3,3) circle (5pt);
		\draw[fill=black] (4,2) circle (5pt);
		\draw[fill=black] (5,3) circle (5pt);
		\draw[fill=black] (6,4) circle (5pt);
		\draw[fill=black] (7,3) circle (5pt);
		\draw[fill=black] (8,4) circle (5pt);
		\draw[fill=black] (9,5) circle (5pt);
		\draw[fill=black] (10,4) circle (5pt);
		\draw[fill=black] (11,3) circle (5pt);
		\draw[fill=black] (12,2) circle (5pt);
		\draw[fill=black] (13,1) circle (5pt);
		\draw[fill=black] (14,0) circle (5pt);
		\node[above=2pt] at (0,0) {$0$};
		\node[above=2pt] at (1,1) {$1$};
		\node[above=2pt] at (2,2) {$2$};
		\node[above=2pt] at (3,3) {$2$};
		\node[above=2pt] at (4,2) {$1$};
		\node[above=2pt] at (5,3) {$2$};
		\node[above=2pt] at (6,4) {$2$};
		\node[above=2pt] at (7,3) {$2$};
		\node[above=2pt] at (8,4) {$2$};
		\node[above=2pt] at (9,5) {$2$};
		\node[above=2pt] at (10,4) {$2$};
		\node[above=2pt] at (11,3){$2$};
		\node[above=2pt] at (12,2) {$1$};
		\node[above=2pt] at (13,1) {$1$};
		\node[above=2pt] at (14,0) {$0$};
	\end{tikzpicture}
\end{array}
$$
\end{footnotesize}
\caption{Bloom and Saracino's bijection between $231$-avoiding (upper-left) and $312$-avoiding frps (bottom-left).}
\label{fig-bloom-saracino}
\end{figure}

\index{Dyck path}

We need a final term. A \emph{weak tunnel} in a Dyck path is a horizontal segment between two vertices of the path that stays weakly below the path (the paths shown in Figure~\ref{fig-bloom-saracino} each have a weak tunnel indicated in gray). Bloom and Saracino proved that these labeled Dyck paths satisfy three rules:
\begin{itemize}
\item \emph{Monotone property:} Labels increase by at most $1$ after an up step and decrease by at most $1$ after a down step.
\item \emph{Zero property:} The $0$ labels are precisely those on the $x$-axis.
\item \emph{${231}$-avoiding tunnel property:} Given two vertices at the same height connected by a weak tunnel, the label of the leftmost vertex is at most the label of the rightmost vertex. (This was called the ``diagonal property'' in \cite{bloom:pattern-avoidan:}.)
\end{itemize}
They also proved that there is a bijection between labeled Dyck paths and $231$-avoiding frps satisfying these three properties. Suppose we perform the same operation on $312$-avoiding frps, as shown on the bottom of Figure~\ref{fig-bloom-saracino}. Bloom and Saracino showed that the resulting labeled Dyck paths have all of the same properties except for one; the $231$-avoiding tunnel property becomes the
\begin{itemize}
\item \emph{${312}$-avoiding tunnel property:} Given two vertices at the same height connected by a weak tunnel, the label of the rightmost vertex is at most the label of the leftmost vertex.
\end{itemize}
To prove Theorem~\ref{thm-231-312-wilf-equiv}, we merely need to construct a bijection between labeled Dyck paths satisfying the monotone, zero, and $231$-avoiding tunnel properties and those that satisfy the monotone, zero, and $312$-avoiding tunnel properties. This (as is illustrated on the right of Figure~\ref{fig-bloom-saracino}) is done by fixing the labels on the $x$-axis at zero, and then replacing the label $\ell$ in the labeled Dyck path of a $231$-avoiding frp by $h-\ell+1$, where $h$ is the height of the corresponding vertex.

As no other shape-Wilf-equivalences are known, we must pose the following question.

\begin{question}
Do Theorems~\ref{thm-inc-dec-wilf-equiv} and \ref{thm-231-312-wilf-equiv} (and their implications) constitute all of the shape-Wilf-equivalences?
\end{question}

To finish mapping the known world of Wilf-equivalence of principal classes, we need one more result.

\begin{theorem}[Stankova~\cite{stankova:forbidden-subse:}]
\label{thm-1342-2413-wilf-equiv}
The permutations $1342$ and $2413$ are Wilf-equivalent, but not shape-Wilf-equivalent.
\end{theorem}

Bloom~\cite{bloom:a-refinement-of:} was the first to describe a nice bijection between these classes (more precisely, he constructed a bijection between $1423$-avoiding permutations and $2413$-avoiding permutations). By showing that his bijection preserves certain permutation statistics, he was able to establish a conjecture of Dokos, Dwyer, Johnson, Sagan, and Selsor~\cite{dokos:permutation-pat:}.

\begin{figure}
\begin{center}
	\begin{tikzpicture}[scale=0.2, baseline=(current bounding box.center)]
		\draw [lightgray, ultra thick, line cap=round] (0,0) rectangle (13,13);
		\draw[fill=black] (1,2) circle (10pt);
		\draw[fill=black] (2,11) circle (10pt);
		\draw[fill=black] (3,4) circle (10pt);
		\draw[fill=black] (4,1) circle (10pt);
		\draw[fill=black] (5,6) circle (10pt);
		\draw[fill=black] (6,3) circle (10pt);
		\draw[fill=black] (7,8) circle (10pt);
		\draw[fill=black] (8,5) circle (10pt);
		\draw[fill=black] (9,10) circle (10pt);
		\draw[fill=black] (10,7) circle (10pt);
		\draw[fill=black] (11,12) circle (10pt);
		\draw[fill=black] (12,9) circle (10pt);
	\end{tikzpicture}
\end{center}
\caption{A typical basis element for the class of permutations sortable by two increasing stacks in series.}
\label{fig-inc-inc-basis}
\end{figure}
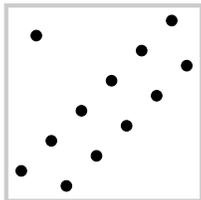

Finally, we remark that classes do not need to have the same number of basis elements to be Wilf-equivalent. For example, Atkinson, Murphy, and Ru\v{s}kuc~\cite{atkinson:sorting-with-tw:} showed that the class $\Av(1342)$ is Wilf-equivalent to the infinitely based class
$$
\Av(\{2\ (2k-1)\ 4\ 1\ 6\ 3\ 8\ 5\ \cdots\ (2k)\ (2k-3) \st k\ge 2\}),
$$
which consists of those permutations that can be sorted by two increasing stacks in series (for a visualization of these basis elements, see Figure~\ref{fig-inc-inc-basis}; a proof of this result using frps is given in Bloom and Vatter~\cite{bloom:two-vignettes-o:}). Burstein and Pantone~\cite{burstein:two-examples-of:} later showed that $\Av(1234)$ and $\Av(1324,3416725)$ are Wilf-equivalent, along with other \emph{unbalanced Wilf-equivalences}.

\index{unbalanced Wilf-equivalence}

\subsection{Avoiding a longer permutation}
\label{subsec-len4}

In the study of principal classes, there is a steep increase in difficulty when we increase the length of the basis element from $3$ to $4$. We begin this section with a routine bound, but end it with the easily-stated but notorious problem of computing the growth rate of $\Av(1324)$.

\begin{proposition}
\label{prop-bound-mono}
The growth rate of $\Av(k\cdots 21)$ is at most $(k-1)^2$.
\end{proposition}
\begin{proof}
Let $\pi$ be a $k\cdots21$-avoiding permutation of length $n$. By the obvious generalization of Proposition~\ref{prop-321-merge} mentioned in Section~\ref{subsec-len3}, we can partition the entries of $\pi$ into $k-1$ increasing subsequences. Label these subsequences $1$ to $k-1$ and extend this labeling to their entries. We can easily reconstruct $\pi$ if we know the positions and values of the entries of each of these subsequences. We can record this information in two lists, one reading left to right, the other reading bottom to top. Thus there are at most $(k-1)^{2n}$ permutations of length $n$ that avoid $k\cdots 21$, as desired.
\end{proof}

There does not seem to be an elementary way to establish a matching lower bound for Proposition~\ref{prop-bound-mono}. B\'ona~\cite[Corollary 5.8]{bona:the-limit-of-a-:} proved that for every permutation $\beta$,
$$
\sqrt{\gr(\Av(\beta\ominus 21))}=1+\sqrt{\gr(\Av(\beta\ominus 1))},
$$
which implies that $\gr(\Av(k\cdots 21))=(k-1)^2$, but his proof is quite involved. The first derivation of $\gr(\Av(k\cdots21))$ follows from a much more general result of Regev.

\begin{theorem}[Regev~\cite{regev:asymptotic-valu:}]
\label{thm-regev}
The growth rate of $\Av(k\cdots 21)$ is $(k-1)^2$.
\end{theorem}

This immediately implies the following, which will be useful later when we bound the growth rate of principal classes avoiding layered permutations (Theorem~\ref{thm-layered-gr}).

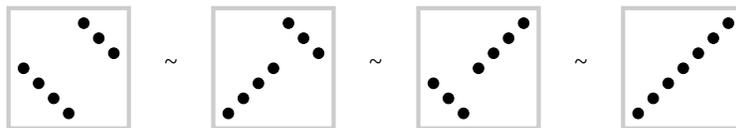
\begin{figure}
$$
	\begin{tikzpicture}[scale=0.2, baseline=(current bounding box.center)]
		\draw [lightgray, ultra thick, line cap=round] (0,0) rectangle (8,8);
		\draw[fill=black] (1,4) circle (10pt);
		\draw[fill=black] (2,3) circle (10pt);
		\draw[fill=black] (3,2) circle (10pt);
		\draw[fill=black] (4,1) circle (10pt);
		\draw[fill=black] (5,7) circle (10pt);
		\draw[fill=black] (6,6) circle (10pt);
		\draw[fill=black] (7,5) circle (10pt);
	\end{tikzpicture}
\quad\sim\quad
	\begin{tikzpicture}[scale=0.2, baseline=(current bounding box.center)]
		\draw [lightgray, ultra thick, line cap=round] (0,0) rectangle (8,8);
		\draw[fill=black] (1,1) circle (10pt);
		\draw[fill=black] (2,2) circle (10pt);
		\draw[fill=black] (3,3) circle (10pt);
		\draw[fill=black] (4,4) circle (10pt);
		\draw[fill=black] (5,7) circle (10pt);
		\draw[fill=black] (6,6) circle (10pt);
		\draw[fill=black] (7,5) circle (10pt);
	\end{tikzpicture}
\quad\sim\quad
	\begin{tikzpicture}[scale=0.2, baseline=(current bounding box.center)]
		\draw [lightgray, ultra thick, line cap=round] (0,0) rectangle (8,8);
		\draw[fill=black] (1,3) circle (10pt);
		\draw[fill=black] (2,2) circle (10pt);
		\draw[fill=black] (3,1) circle (10pt);
		\draw[fill=black] (4,4) circle (10pt);
		\draw[fill=black] (5,5) circle (10pt);
		\draw[fill=black] (6,6) circle (10pt);
		\draw[fill=black] (7,7) circle (10pt);
	\end{tikzpicture}
\quad\sim\quad
	\begin{tikzpicture}[scale=0.2, baseline=(current bounding box.center)]
		\draw [lightgray, ultra thick, line cap=round] (0,0) rectangle (8,8);
		\draw[fill=black] (1,1) circle (10pt);
		\draw[fill=black] (2,2) circle (10pt);
		\draw[fill=black] (3,3) circle (10pt);
		\draw[fill=black] (4,4) circle (10pt);
		\draw[fill=black] (5,5) circle (10pt);
		\draw[fill=black] (6,6) circle (10pt);
		\draw[fill=black] (7,7) circle (10pt);
	\end{tikzpicture}
$$
\caption{Demonstrating the Wilf-equivalence of a two-layer permutation and a monotone permutation of the same length. The first and last equivalences follow from Theorem~\ref{thm-inc-dec-wilf-equiv}, while the central equivalence is an application of symmetry.}
\label{fig-two-layer}
\end{figure}

\begin{corollary}
\label{cor-two-layer}
For any layered permutation $\beta$ of length $k$ and consisting of at most two layers, the growth rate of $\Av(\beta)$ is $(k-1)^2$.
\end{corollary}
\begin{proof}
The single layer case is Regev's Theorem~\ref{thm-regev}, while the two-layer case follows by Theorem~\ref{thm-inc-dec-wilf-equiv} and symmetry, as shown in Figure~\ref{fig-two-layer}.
\end{proof}

Gessel~\cite{gessel:symmetric-funct:} later gave an explicit formula for the generating functions of the classes $\Av(k\cdots 21)$ in terms of determinants. Stanley~\cite{stanley:recent-progress:} expressed the significance of Gessel's result by writing that
\begin{quote}
Gessel's theorem reduces the theorems of Baik, Deift, and Johansson to ``just'' analysis, viz., the Riemann-Hilbert problem in the theory of integrable systems, followed by the method of steepest descent to analyze the asymptotic behavior of integrable systems.
\end{quote}
(The quote refers to Baik, Deift, and Johansson's proof~\cite{baik:on-the-distribu:} that, after appropriate rescaling, the distribution of the longest increasing subsequence statistic on permutations of length $n$ converges to the Tracy-Widom distribution.)

\index{$4321$-avoiding permutations}

Gessel's result was later rederived by Bousquet-M\'elou~\cite{bousquet-melou:counting-permut:} using the kernel method. In the case $k=4$, she expressed the enumeration as an explicit sum,
$$
|\Av_n(4321)|
=
\frac{1}{(n+1)^2(n+2)}\sum_{i=0}^n {2i\choose i}{n+1\choose i+1}{n+2\choose i+2}.
$$
(Bousquet-M\'elou had given an earlier derivation of this formula in \cite{bousquet-melou:four-classes-of:}.) The generating function in the $k=4$ case is known to be $D$-finite but nonalgebraic, and no ``nice'' formulas are known for any larger values of $k$.

The Wilf-equivalences of Section~\ref{subsec-wilf-equiv} together with the trivial symmetries show that (from an enumerative perspective) there are only two more cases of principal classes avoiding a permutation of length $4$: $\Av(1342)$ and $\Av(1324)$. B\'ona was the first to enumerate $\Av(1342)$, by constructing a bijection between the skew indecomposable $1342$-avoiding permutations and $\beta(0,1)$-trees, which had been previously counted by Tutte~\cite{tutte:a-census-of-pla:}. (Because the class $\Av(1342)$ is skew closed, its generating function can be computed easily from the generating function for its skew indecomposables, as mentioned in Section~\ref{subsec-basics}.)

\index{$1342$-avoiding permutations}


\begin{theorem}[B\'ona~\cite{bona:exact-enumerati:}]
\label{thm-1342}
The generating function for $\Av(1342)$ is 
$$
\frac{32x}{1+20x-8x^2-(1-8x)^{3/2}},
$$
and thus the growth rate of this class is $8$.
\end{theorem}

Bloom and Elizalde~\cite{bloom:pattern-avoidan:} have recently shown how to derive this generating function using the techniques of Bloom and Saracino~\cite{bloom:a-simple-biject:} discussed in Section~\ref{subsec-wilf-equiv}, which we briefly sketch.

First, instead of counting $\Av(1342)$ they count the symmetric class $\Av(3124)$. A frp is \emph{board minimal} if its rooks do not lie in any smaller Ferrers board, or equivalently, if it has a rook in each of its upper-right corners. There is an obvious bijection, which we denote by $\chi$, between permutations and board minimal frps. It is also easy to see that were a frp from $\chi(\Av(3124))$ to contain $312$, there would be a rook in an upper-right corner above and to the right of this copy of $312$, a contradiction. Therefore these frps avoid $312$.

Next label the border of these frps as in Section~\ref{subsec-wilf-equiv} and convert them to Dyck paths. The corresponding Dyck paths must satisfy the monotone, zero, and $312$-avoiding tunnel properties, and in addition, because they arise from board minimal frps, they satisfy the \emph{peak property}: the labels rounding any peak (an up step followed by a down step) are $\ell$, $\ell+1$, $\ell$ for some value of $\ell$.


Thus we are reduced to counting the family of labeled Dyck paths satisfying the monotone, zero, $312$-avoiding tunnel, and peak properties. Bloom and Elizalde were able to find the generating function for these by introducing a catalytic variable and applying Tutte's quadratic method (which is described in Flajolet and Sedgewick~\cite[VII. 8.2]{flajolet:analytic-combin:}).

%
%
%

\index{$1324$-avoiding permutations}

This leaves a single case, that of $\Av(1324)$. As reported in \cite{elder:problems-and-co:}, at the conference \emph{Permutation Patterns 2005}, Zeilberger made the colorful claim that
\begin{quote}
Not even God knows $|\Av_{1000}(1324)|$.
\end{quote}
Steingr{\'{\i}}msson is less pessimistic, writing in \cite{steingrimsson:some-open-probl:} that
\begin{quote}
I'm not sure how good Zeilberger's God is at math, but I believe that some humans will find this number in the not so distant future.
\end{quote}
B\'ona currently holds the record for the lowest upper bound on the growth rate of $\Av(1324)$.

\begin{theorem}[B\'ona~\cite{bona:a-new-record-fo:}]
\label{thm-1324-upper-bound}
The growth rate of $\Av(1324)$ is at most $13.74$.
\end{theorem}

\begin{figure}
\begin{center}
\begin{footnotesize}
\begin{tikzpicture}
  \matrix (A) [matrix of math nodes, row sep=-0.05mm, column sep=-0.05mm] {
    \vvg{1}\\
    \vvg{1}& 1\\ 
    \vvg{1}& \vvg{2}& 2& 1\\
    \vvg{1}& \vvg{2}& \vvg{5}& 6& 5& 3& 1\\
    \vvg{1}& \vvg{2}& \vvg{5}& \vvg{10}& 16& 20& 20&  15&   9&   4&   1\\
    \vvg{1}& \vvg{2}& \vvg{5}& \vvg{10}& \vvg{20}& 32& 51&  67&  79&  80&  68&  49&   29& \dots\\
    \vvg{1}& \vvg{2}& \vvg{5}& \vvg{10}& \vvg{20}& \vvg{36}& 61&  96& 148& 208& 268& 321&  351& \dots\\
    \vvg{1}& \vvg{2}& \vvg{5}& \vvg{10}& \vvg{20}& \vvg{36}& \vvg{65}& 106& 171& 262& 397& 568&  784& \dots\\
    \vvg{1}& \vvg{2}& \vvg{5}& \vvg{10}& \vvg{20}& \vvg{36}& \vvg{65}& \vvg{110}& 181& 286& 443& 664&  985& \dots\\
  };
\end{tikzpicture}
\end{footnotesize}
\end{center}
\caption{The number of $1324$-avoiding permutations, grouped into columns by their number of inversions.}
\label{fig-claesson-chart}
\end{figure}

B\'ona established Theorem~\ref{thm-1324-upper-bound} by refining the approach of Claesson, Jel{\'{\i}}nek, and Steingr{\'{\i}}msson~\cite{claesson:upper-bounds-fo:} presented in Section~\ref{subsec-merge}. Claesson, Jel{\'{\i}}nek, and Steingr{\'{\i}}msson made a further conjecture which, if true, would give a better upper bound. Consider the table shown in Figure~\ref{fig-claesson-chart}. In this table, the entry in column $k+1$ of row $n$ is the number of $1324$-avoiding permutations of length $n$ with precisely $k$ inversions. As the gray cells are meant to indicate, it appears that the entries along a column are weakly increasing and eventually constant. They proved that the columns are eventually constant, but could not show that they are weakly increasing. If this conjecture could be proved, it would give the improved bound of $e^{\pi\sqrt{2/3}}\le 13.01$ on the growth rate of $\Av(1324)$. Indeed this might be part of a much more general phenomenon; empirical evidence suggests that the analogous table for $\beta$-avoiding permutations has weakly increasing columns for all $\beta\neq 12\cdots k$.


Arratia~\cite{arratia:on-the-stanley-:} had conjectured that $\gr(\Av(\beta))\le (k-1)^2$ for all permutations $\beta$ of length $k$, and B\'ona~\cite{bona:the-limit-of-a-:} had gone even further, conjecturing that equality held if and only if $\beta$ was layered. We know now (by Fox's results presented in Section~\ref{subsec-fox}) that these conjectures are far from the truth, but the first known counterexample was the class $\Av(1324)$. Using quite a bit of computation and the insertion encoding (introduced by Albert, Linton, and Ru\v{s}kuc~\cite{albert:the-insertion-e:}), Albert, Elder, Rechnitzer, Westcott, and Zabrocki~\cite{albert:on-the-stanley-:} showed that
$$
\gr(\Av(1324))\ge 9.47. 
$$
Extrapolating from the amount their lower bounds improved as they increased the accuracy of their approximation, they guessed that the true value lies between $11$ and $12$. Madras and Liu~\cite{madras:random-pattern-:} later used Markov chain Monte Carlo methods to estimate that this growth rate lies between $10.71$ and $11.83$.

At present, the best estimate seems to be due to Conway and Guttmann~\cite{conway:on-the-growth-r:}. They extended the approach of Johansson and Nakamura~\cite{johansson:using-functiona:} (who had computed the first $31$ terms of the enumeration) to compute the first $36$ terms of the enumeration of this class and then applied the methodology laid out in Guttmann~\cite{guttmann:analysis-of-ser:} to approximate that
$$
|\Av_n(1324)|\sim C\mu^n\nu^{\sqrt{n}} n^g,
$$
where $C\approx 9.5$, $\nu\approx 0.04$, $g\approx -1.1$. and $\mu$---the growth rate---is approximately $11.60$.

Bevan~\cite{bevan:a-large-set-of-:} has recently established that $\gr(\Av(1324))>9.81$ by constructing a large subset of this class with a particularly regular structure. Sketching his proof would require quite a detour, but we attempt to convey a hint of his novel approach.

Given a permutation $\pi$, define the poset $P_\pi$ on the points $\{(i,\pi(i))\}$ in which $(i,\pi(i))\le (j,\pi(j))$ if both $i<j$ and $\pi(i)<\pi(j)$, i.e., if this pair of entries forms a noninversion in $\pi$. The drawing on the left of Figure~\ref{fig-1324-hasse} shows the Hasse diagram of the poset of a $1324$-avoiding permutation. Because this permutation avoids $1324$, its Hasse diagram does not contain a diamond:
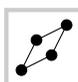
\begin{center}
	\begin{tikzpicture}[scale=0.2, baseline=(current bounding box.center)]
		\draw [lightgray, ultra thick, line cap=round] (0,0) rectangle (5,5);
		\draw (1,1)--(2,3)--(4,4);
		\draw (1,1)--(3,2)--(4,4);
		\draw[black, fill=black] (1,1) circle (10pt);
		\draw[black, fill=black] (2,3) circle (10pt);
		\draw[black, fill=black] (3,2) circle (10pt);
		\draw[black, fill=black] (4,4) circle (10pt);
	\end{tikzpicture}
\end{center}

\begin{figure}
\begin{center}
	\begin{tikzpicture}[scale=0.2, baseline=(current bounding box.center)]
		\draw [lightgray, ultra thick, line cap=round] (0,0) rectangle (22,22);
		\draw (6,9)--(7,18);
		\draw (6,9)--(13,17);
		\draw (6,9)--(17,13);
		\draw (6,9)--(18,12);
		\draw (6,9)--(19,10);
		\draw (11,4)--(13,17);
		\draw (11,4)--(16,6);
		\draw (11,4)--(21,5);
		\draw (1,15)--(2,21);
		\draw (1,15)--(3,16);
		\draw (3,16)--(7,18);
		\draw (3,16)--(13,17);
		\draw (7,18)--(8,20);
		\draw (7,18)--(12,19);
		\draw (11,4)--(12,19);
		\draw (14,14)--(6,9);
		\draw (14,14)--(11,4);
		\draw (6,9)--(4,8);
		\draw (6,9)--(5,7);
		\draw (11,4)--(9,2); 
		\draw (15,1)--(16,6);
		\draw (15,1)--(21,5);
		\draw (16,6)--(17,13);
		\draw (16,6)--(18,12);
		\draw (16,6)--(19,10);
		\draw (19,10)--(20,11);
		\draw[black, fill=black] (1,15) circle (10pt);
		\draw[black, fill=black] (2,21) circle (10pt);
		\draw[black, fill=black] (3,16) circle (10pt);
		\draw[black, fill=black] (4,8) circle (10pt);
		\draw[black, fill=black] (5,7) circle (10pt);
		\draw[black, fill=black] (6,9) circle (10pt);
		\draw[black, fill=black] (7,18) circle (10pt);
		\draw[black, fill=black] (8,20) circle (10pt);
		\draw[black, fill=black] (12,19) circle (10pt);
		\draw[black, fill=black] (9,2) circle (10pt);
		\draw[black, fill=black] (10,3) circle (10pt);
		\draw[black, fill=black] (11,4) circle (10pt);
		\draw[black, fill=black] (13,17) circle (10pt);
		\draw[black, fill=black] (14,14) circle (10pt);
		\draw[black, fill=black] (15,1) circle (10pt);
		\draw[black, fill=black] (16,6) circle (10pt);
		\draw[black, fill=black] (17,13) circle (10pt);
		\draw[black, fill=black] (18,12) circle (10pt);
		\draw[black, fill=black] (19,10) circle (10pt);
		\draw[black, fill=black] (20,11) circle (10pt);
		\draw[black, fill=black] (21,5) circle (10pt);
	\end{tikzpicture}
\quad\quad
	\begin{footnotesize}
	\begin{tikzpicture}[scale=0.2, baseline=(current bounding box.center)]
		\draw[color=lightgray, fill=lightgray, rotate around={-45:(5.1,8.1)}] (5.1,8.1) ellipse (45pt and 55pt);
		\draw[color=lightgray, fill=lightgray, rotate around={-45:(10,3)}] (10,3) ellipse (25pt and 60pt);
		\draw [lightgray, ultra thick, line cap=round] (0,0) rectangle (22,22);
		\draw [lightgray, ultra thick, line cap=round] (0,14.5)--(22,14.5);
		\draw [lightgray, ultra thick, line cap=round] (14.5,0)--(14.5,22);
		\draw (1,15)--(2,21);
		\draw (1,15)--(3,16);
		\draw (3,16)--(7,18);
		\draw (3,16)--(13,17);
		\draw (7,18)--(8,20);
		\draw (7,18)--(12,19);
		\draw (14,14)--(6,9);
		\draw (14,14)--(11,4);
		\draw (6,9)--(4,8);
		\draw (6,9)--(5,7);
		\draw (11,4)--(9,2); 
		\draw (15,1)--(16,6);
		\draw (15,1)--(21,5);
		\draw (16,6)--(17,13);
		\draw (16,6)--(18,12);
		\draw (16,6)--(19,10);
		\draw (19,10)--(20,11);
		\draw[black, fill=black] (1,15) circle (10pt);
		\draw[black, fill=black] (2,21) circle (10pt);
		\draw[black, fill=black] (3,16) circle (10pt);
		\draw[black, fill=black] (4,8) circle (10pt);
		\draw[black, fill=black] (5,7) circle (10pt);
		\draw[black, fill=black] (6,9) circle (10pt);
		\draw[black, fill=black] (7,18) circle (10pt);
		\draw[black, fill=black] (8,20) circle (10pt);
		\draw[black, fill=black] (12,19) circle (10pt);
		\draw[black, fill=black] (9,2) circle (10pt);
		\draw[black, fill=black] (10,3) circle (10pt);
		\draw[black, fill=black] (11,4) circle (10pt);
		\draw[black, fill=black] (13,17) circle (10pt);
		\draw[black, fill=black] (14,14) circle (10pt);
		\draw[black, fill=black] (15,1) circle (10pt);
		\draw[black, fill=black] (16,6) circle (10pt);
		\draw[black, fill=black] (17,13) circle (10pt);
		\draw[black, fill=black] (18,12) circle (10pt);
		\draw[black, fill=black] (19,10) circle (10pt);
		\draw[black, fill=black] (20,11) circle (10pt);
		\draw[black, fill=black] (21,5) circle (10pt);
		\node [above right] at (14,14) {$\pi(i)$};
		\node [right] at (15,1) {$\pi(j)$};
	\end{tikzpicture}
	\end{footnotesize}
\end{center}
\caption{The drawing on the left shows the Hasse diagram of the poset of a $1324$-avoiding permutation. On the right, this Hasse diagram has been divided into three trees.}
\label{fig-1324-hasse}
\end{figure}
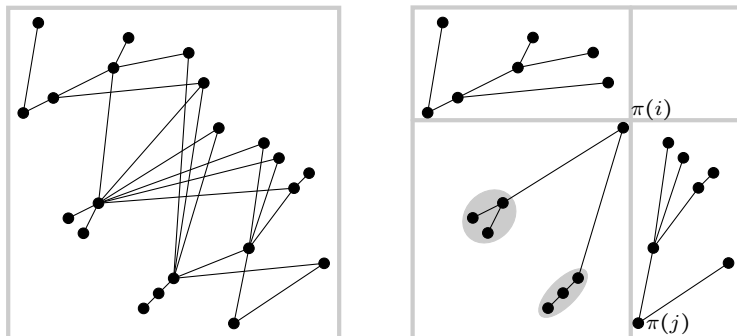

Therefore if we restrict the Hasse diagram to $\pi(1)$ and all entries lying above and to its right we see a rooted tree, as drawn on the right of Figure~\ref{fig-1324-hasse}. Next let $\pi(i)$ denote the greatest point lying below $\pi(1)$. If we restrict the Hasse diagram to $\pi(i)$ and all entries lying below and to its left, we see another rooted tree. Finally, for this example, if we let $\pi(j)$ denote the leftmost entry to the right of $\pi(i)$ and restrict the Hasse diagram to this entry and all entries above and to its right, we see yet another rooted tree. Let us call rooted trees like the first and third \emph{up trees} and trees like the second  \emph{down trees}.

Bevan's construction consists of alternating up trees and down trees, while interleaving their vertices. However, this alone is not sufficient to avoid $1324$. For example, the following interleaving of a down tree and an up tree contains $1324$.
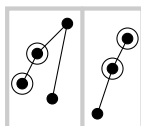
\begin{center}
	\begin{tikzpicture}[scale=0.2, baseline=(current bounding box.center)]
		\draw [lightgray, ultra thick, line cap=round] (0,0) rectangle (9,8);
		\draw [lightgray, ultra thick, line cap=round] (5,0)--(5,8);
		\draw (4,7)--(2,5);
		\draw (2,5)--(1,3);
		\draw (4,7)--(3,2);
		\draw (6,1)--(7,4)--(8,6);
		\draw[fill=black] (1,3) circle (10pt);
		\draw[fill=black] (2,5) circle (10pt);
		\draw[fill=black] (3,2) circle (10pt);
		\draw[fill=black] (4,7) circle (10pt);
		\draw[fill=black] (6,1) circle (10pt);
		\draw[fill=black] (7,4) circle (10pt);
		\draw[fill=black] (8,6) circle (10pt);
		\draw (1,3) circle (20pt);
		\draw (2,5) circle (20pt);
		\draw (7,4) circle (20pt);
		\draw (8,6) circle (20pt);
	\end{tikzpicture}
\end{center}
Thus we must be careful in performing this interleaving. In constructing the permutation shown in Figure~\ref{fig-1324-hasse}, we avoided creating a copy of $1324$ by interleaving the vertices of the up trees with the principal subtrees of the down tree, as indicated by the shading in Figure~\ref{fig-1324-hasse} (a \emph{principal subtree} of a rooted tree is a component of the forest obtained after removing the root).

\section{Growth rates of principal classes}
\label{sec-principal}

In the 1980s, Stanley and Wilf independently conjectured that every proper permutation class (that is, every class except the class of all permutations) has a finite upper growth rate. Obviously, it suffices to prove this for principal classes. In this context, the growth rate of $\Av(\beta)$ is often (though not in this survey) called the \emph{Stanley--Wilf limit} of $\beta$.

\index{Stanley--Wilf limit}

Two noteworthy partial results were proved in the 20th century.  First B\'ona~\cite{bona:the-solution-of:} proved the Stanley--Wilf Conjecture for $\Av(\beta)$ when $\beta$ is a layered permutation (we prove a stronger result in Theorem~\ref{thm-layered-gr}). Then Alon and Friedgut~\cite{alon:on-the-number-o:} used a result of Klazar~\cite{klazar:a-general-upper:} about Davenport-Schinzel sequences to prove that for every $\beta$ there is a constant $c_\beta$ such that $|\Av_n(\beta)|<c_\beta^{n\gamma(n)}$ where $\gamma$ grows even more slowly than the inverse Ackermann function (the inverse Ackermann function itself is considered to be less than $5$ for all ``reasonable'' values of $n$).  This result caused Wilf's belief to waver; in 2002 he~\cite{wilf:the-patterns-of:} wrote that

\begin{quote}
This conjecture had been considered a ``sure thing,'' but the results of Alon and Friedgut seem to make it somewhat less certain because a similar bound, involving the Ackermann function, in the Davenport-Schinzel theory turns out to be best possible.
\end{quote}

In 2004, Marcus and Tardos~\cite{marcus:excluded-permut:} presented an elegant proof of the Stanley--Wilf Conjecture. About their proof, Zeilberger~\cite{zeilberger:opinion-58:-i-s:} wrote
\begin{quote}
Once I thought that proofs of long-standing conjectures had to be difficult. After all didn't brilliant people make them, and didn't these people, and many other brilliant people, unsuccessfully try to prove them? Isn't that a meta-proof that the proof, if it exists, should be long, difficult, and ugly? And if not, shouldn't they, at least, contain entirely new ideas and/or techniques?
\end{quote}

We present Marcus and Tardos' proof in Section~\ref{subsec-mt}. Of course, the proof of the Stanley--Wilf Conjecture only raises new questions. For example, what can be said about the relationship between the growth rate of $\Av(\beta)$ and the length of $\beta$? Many conjectures (some of which were mentioned in Section~\ref{subsec-len4}) have been made about this relationship, and most have been proven false. The first truly general result in this direction is due to Valtr, and appeared in a paper of Kaiser and Klazar~\cite{kaiser:on-growth-rates:}. We present a specialization.

\begin{proposition}
\label{prop-valtr}
For all sufficiently large $k$ and all permutations $\beta$ of length $k$, $\gr(\Av(\beta))>k^2/27$.
\end{proposition}
\begin{proof}
Set $n=\lfloor k^2/9\rfloor$ (the integer part of $k^2/9$) and let $\beta$ be any pattern of length $k$. Using linearity of expectation and Stirling's Formula, we see that the probability that a permutation of length $n$ chosen uniformly at random contains $\beta$ is at most
$$
\frac{1}{k!}{n\choose k}
<
\frac{n^k}{(k!)^2}
\le
\frac{k^{2k}}{9^k\left(k/e\right)^{2k}}
<
\left(\frac{e^2}{9}\right)^k.
$$
This quantity tends to $0$ as $k\rightarrow\infty$ so for sufficiently large $k$ at least half of the permutations of length $n$ avoid $\beta$. Now Proposition~\ref{prop-arratia-gr} and Stirling's Formula show that
$$
\gr(\Av(\beta))
\ge
\sqrt[n]{|\Av_n(\beta)|}
\ge
\sqrt[n]{n!/2}
\ge
\frac{\sqrt[n]{\lceil k^2/9\rceil!}}{\sqrt[n]{n/2}}
\ge
\frac{k^2}{9e\sqrt[n]{n/2}}
>
k^2/27
$$
for large enough $k$, as desired.
\end{proof}

Therefore growth rates of principal class grow at least quadratically in terms of the length of the avoided pattern, but the upper bound conjectured by Arratia is false (as we saw in Section~\ref{subsec-len4}). The fact that Arratia's Conjecture failed for a layered permutation only lent more credence to the long-standing conjecture that among all permutations $\beta$ of length $k$, the permutation that is easiest to avoid (thereby giving the greatest growth rate) is layered. In his survey for the 2013 \emph{British Combinatorial Conference}, Steingr{\'{\i}}msson~\cite{steingrimsson:some-open-probl:} wrote that
\begin{quote}
Although the evidence is strong in support of the conjecture that the most easily avoided pattern of any given length is a layered pattern, there is currently no general conjecture that fits all the known data about the particular layered patterns with the most avoiders. However, there are some ideas about what form the most avoided layered patterns ought to have, and specific conjectures that have not been shown to be false.
\end{quote}
The oldest of the specific conjectures Steingr{\'{\i}}msson mentions appeared in Burstein's 1998 thesis~\cite{burstein:enumeration-of-:}. B\'ona~\cite{bona:new-records-in-:} made a competing conjecture in 2007 that the most easily avoided permutation of length $k$ was $1\oplus 21\oplus 21\oplus\cdots\oplus 21$ for odd $k$ and $1\oplus 21\oplus 21\oplus\cdots\oplus 21\oplus 1$ for even $k$.

After these conjectures were made, Claesson, Jel{\'{\i}}nek, and Steingr{\'{\i}}msson~\cite{claesson:upper-bounds-fo:} established that for every layered permutation $\beta$ of length $k$ the growth rate of $\Av(\beta)$ is less than $4k^2$ (we present their argument in Section~\ref{subsec-merge}). B\'ona~\cite{bona:on-the-best-upp:} refined this approach to show that for his conjectured easiest-to-avoid permutations, the corresponding growth rates were at most $2.25k^2$. Because of this, the function
$$
g(k)=\max\{\gr(\Av(\beta))\st |\beta|=k\}
$$
was widely believed to grow quadratically. Thus it was quite a shock when Fox~\cite{fox:stanley-wilf-li:} showed in 2014 that $g(k)$ grows faster than any polynomial. We prove a specialization of Fox's Theorem in Section~\ref{subsec-fox}.

Before moving on, because this section concerns estimates we briefly define the pieces of big-oh notation we use. Let $f(n)$ and $g(n)$ be positive functions. These functions are \emph{asymptotic}, written $f\sim g$, if $f(n)/g(n)\rightarrow 1$ as $n\rightarrow\infty$. We write $f=O(g)$ if there is a constant $C$ such that $f(n)\le Cg(n)$ for all sufficiently large $n$. We similarly write $f=\Omega(g)$ if there is a constant $c$ such that $f(n)\ge cg(n)$ for all sufficiently large $n$. Finally, we write $f=\Theta(g)$ if both $f=O(g)$ and $f=\Omega(g)$.

\subsection{Matrices and the interval minor order}
\label{subsec-int-minor}

In the context of principal classes, the most general results on growth rates have come from expanding the vista to the context of (zero/one) matrices. The study of patterns within matrices dates back to at least 1951, when Zarankiewicz~\cite{zarankiewicz:problem-p101:} posed what is now known as the Zarankiewicz Problem. While his problem is often phrased in terms of bipartite graphs, what he actually asked was
\begin{quote}
Soit $R_n$, o\`u $n>3$, un r\'eseau plan form\'e de $n^2$ points rang\'es en $n$ lignes et $n$ colonnes. Trouver le plus petit nombre naturel $k_2(n)$ tel que tout sous-ensemble de $R_n$ form\'e de $k_2(n)$ points contienne $4$ points situ\'es simultan\'ement dans $2$ lignes et dans $2$ colonnes de ce r\'eseau.
\end{quote}
Translated, the Zarankiewicz Problem asks for the minimum number of ones such that any way those ones are arranged in an $n\times n$ matrix there will be a $2\times 2$ submatrix of all ones.

Let us introduce some notation to discuss this and more general problems. The \emph{weight} of the matrix $M$ is the number of ones it contains, and we denote this quantity by $\wt(M)$. We also say that the matrix $P$ is a \emph{submatrix} of $M$ if $P$ can be obtained from $M$ by deleting rows, columns, and ones (in this last case, changing them to zeros). A generalization of Zarankiewicz's problem is to determine the asymptotics of the function
$$
\exop(n;P)=\max\{\wt(M) : \mbox{$M$ is an $n\times n$ matrix avoiding $P$ as a submatrix}\},
$$
and the quantity Zarankiewicz asked about was $\exop(n;J_2)+1$, where we denote by $J_k$ the $k\times k$ all-one matrix.

Shortly after Zarankiewicz posed his problem, K\H{o}v\'ari, S\'os, and Tur\'an~\cite{kovari:on-a-problem-of:} showed that the answer is asymptotic to $n^{3/2}$. We include a short proof below. The upper bound is from the original proof while the lower bound was presented by Reiman~\cite{reiman:uber-ein-proble:} in 1958.

\begin{theorem}
\label{thm-j2}
The function $\exop(n;J_2)$ is asymptotic to $n^{3/2}$.
\end{theorem}
\begin{proof}
We begin with the upper bound. Take $M$ to be a $J_2$-avoiding $n\times n$ matrix of weight $\exop(n;J_2)$. For each $i$, let $d_i$ denote the number of ones in row $i$. Call two ones in the same row a \emph{couple}, so row $i$ contains ${d_i\choose 2}$ couples. Clearly $M$ can have at most ${n\choose 2}$ total couples because the same two columns cannot participate as a couple in two different rows so
$$
{n\choose 2}\ge \sum_{i=1}^n {d_i\choose 2},
$$
Next we use the Cauchy-Schwarz inequality to conclude that $\sum d_i^2\ge \left(\sum d_i\right)^2/n$. By noting that $\wt(M)=\sum d_i$, we obtain
$$
n^2-n
\ge
\sum_{i=1}^n d_i^2 - \sum_{i=1}^n d_i
\ge
\frac{\left(\wt(M)\right)^2}{n}-\wt(M),
$$
showing that $\exop(n;J_2)$ does not grow asymptotically faster than $n^{3/2}$.

For the lower bound, let $q$ be a prime power and suppose that $n=q^2+q+1$. It is known that there is a finite projective plane of order $q$, with $n$ points and $n$ lines. Label the points $p_1,\dots,p_n$ and the lines $\ell_1,\dots,\ell_n$ and define the incidence matrix $M$ by setting its $(i,j)$ entry equal to one if and only if the point $p_i$ lies on the line $\ell_j$. Clearly $M$ does not contain $J_2$, because any two points determine a unique line. Moreover, a basic counting argument shows that in a projective plane of order $q$ every point lies on $q+1$ lines. Thus we have
$$
\wt(M)=(q+1)(q^2+q+1)\approx n^{3/2}.
$$
The fact that $\exop(n;J_2)\sim n^{3/2}$ follows from this and basic results about the distribution of primes.
\end{proof}

As K\H{o}v\'ari, S\'os, and Tur\'an~\cite{kovari:on-a-problem-of:} noted in their original paper, the upper bound in Theorem~\ref{thm-j2} can be generalized to show that $\exop(n;J_k)=O(n^{2-1/k})$ for all $k$ and they conjectured that this is the true order of magnitude, i.e., that $\exop(n;J_k)=\Theta(n^{2-1/k})$. While this conjecture is widely believed to be true, it remains open. The best general lower bound shows that $\exop(n;J_k)=\Omega(n^{2-2/(k+1)})$.

Almost 40 years passed before the function $\exop(n;P)$ was investigated for matrices other than $J_k$, in two papers published around the same time and, oddly enough, concerning the same forbidden pattern (up to symmetry). Bienstock and Gy{\H{o}}ri~\cite{bienstock:an-extremal-pro:} established that
$$
\exop\left(n; \fnmatrix{ccc}{1&1&\\1&&1}\right)=O(n\log n).
$$
(It is common practice in this field to suppress zeros.) Along with giving an independent proof of this result, F\"uredi~\cite{furedi:the-maximum-num:} also presented a construction (set $M(i,j)=1$ if and only if $j-i$ is a power of $2$) showing that this is the correct rate of growth. Both papers were motivated by problems in discrete geometry: Bienstock and Gy{\H{o}}ri were investigating the complexity of an algorithm for computing obstacle-avoiding rectilinear paths in the plane, while F\"uredi used this extremal function to bound the number of unit distances between the vertices of a convex polygon.

The first systematic investigation of these extremal functions was performed by F\"uredi and Hajnal~\cite{furedi:davenport-schin:} in 1992. They found the extremal functions (up to constant factors) for several matrices. Some of these extremal functions are quite exotic; for example, using the work of Hart and Sharir~\cite{hart:nonlinearity-of:} F\"uredi and Hajnal showed that
$$
\exop\left(n; \fnmatrix{cccc}{1&&1&\\&1&&1}\right)=\Theta(n\alpha(n)),
$$
where $\alpha(n)$ is the inverse Ackermann function. At the end of their article, they asked
\begin{quote}
is it true that the complexity of all permutation configurations are linear?
\end{quote}
The statement that $\exop(n;M_\beta)$ is linear in $n$ for every permutation $\beta$ became known as the F\"uredi--Hajnal Conjecture; here $M_\beta$ denotes the \emph{permutation matrix} of $\beta$, which is defined by $M_\beta(i,j)=1$ if and only if $\beta(i)=j$. Indeed, at the time Marcus and Tardos proved the F\"uredi--Hajnal Conjecture, they did not know about the Stanley--Wilf Conjecture, nor that Klazar had shown that the F\"uredi--Hajnal Conjecture implied the Stanley--Wilf Conjecture (we prove this in the next subsection).

Despite the rich history of these submatrix problems, later advances---especially those of Fox~\cite{fox:stanley-wilf-li:}---show that the permutation containment order is much more closely related to a different order on matrices, the \emph{interval minor} order. In retrospect, this order has featured in the proofs of all of the main results presented in this section, before it had ever been formally defined. Such a formal definition requires a few preliminaries. The \emph{contraction} of two adjacent rows of a matrix is obtained by replacing those two rows by a single row that has a one in a column if either of the original rows had a one in that column. We define the contraction of a column analogously. We further say that $P$ is a \emph{contraction} of $M$ if $P$ can be obtained from $M$ by a sequence of contractions of adjacent rows or columns.

\index{interval minor}

To give an alternative definition of contractions, given an $n\times n$ matrix $M$ and intervals $X,Y\subseteq [1,n]$, we denote by $M(X\times Y)$ the \emph{block} of $M$ consisting of those entries $M(i,j)$ with $(i,j)\in X\times Y=\{(x,y)\st x\in X, y\in Y\}$. The contraction $P$ of $M$ can be defined by a pair of sequences $1=c_1\le\cdots\le c_{t+1}=n+1$ (the column divisions) and $1=r_1\le\cdots\le r_{u+1}=n+1$ (the row divisions) where we set
$$
P(i,j)
=
\left\{\begin{array}{ll}
1&\mbox{if the block $M([c_i,c_{i+1})\times [r_j,r_{j+1}))$ contains a one and}\\
0&\mbox{otherwise.}
\end{array}\right.
$$
(We continue to use Cartesian coordinates so that the permutation matrix of $\pi$ looks like its plot.)

Finally, the matrix $P$ is an \emph{interval minor} of $M$ if $P$ is a submatrix of a contraction of $M$. In other words, the interval minor order allows for the contraction of adjacent rows or columns, and the deletion of ones. As we will see, at the coarsest level, questions about growth rates of principal permutation classes are intimately connected with the following question:
\begin{quote}
How much can an $n\times n$ matrix weigh if does not contain $J_k$ as an interval minor?
\end{quote}
We present two answers, in Sections~\ref{subsec-mt} (an upper bound) and \ref{subsec-fox} (a lower bound). The connection between this question and growth rates is explored in Sections~\ref{subsec-klazar} and \ref{subsec-cibulka}. We begin with two elementary observations.

\begin{observation}
\label{obs-minor-implies-sub}
If the matrix $M$ contains the permutation matrix $M_\beta$ as an interval minor then it also contains $M_\beta$ as a submatrix.
\end{observation}

Due to this fact, when we say that a matrix contains a permutation matrix we needn't specify which order we are considering. Indeed, permutation matrices are essentially characterized by the condition in Observation~\ref{obs-minor-implies-sub}---it is an easy exercise to show that if $P$ is such that every matrix containing $P$ as an interval minor also contains $P$ as a submatrix then, modulo all-zero rows and columns, $P$ is a permutation matrix.

Observation~\ref{obs-minor-implies-sub} implies that if a matrix contains $J_k$ as an interval minor then it contains \emph{every} $k\times k$ permutation matrix. Thus to get an upper bound on $\gr(\Av(\beta))$ we can consider the set of matrices that avoid $J_{|\beta|}$ as an interval minor. In particular, the F\"uredi--Hajnal Conjecture will follow if we can establish that, for every $k$,  there is a constant $c_k$ such that $\wt(M)\le c_k n$ for every $n\times n$ matrix that does not contain $J_k$ as an interval minor.

\begin{figure}
\begin{footnotesize}
\begin{center}
	\begin{tikzpicture}[scale=0.3, baseline=(current bounding box.center)]
		\foreach \i in {0.5, 3.5, 6.5, 9.5}{
			\draw (\i,0.5)--(\i,9.5);
			\draw (0.5,\i)--(9.5,\i);
		}
		\node at (1,1) {$1$};
		\node at (4,2) {$1$};
		\node at (7,3) {$1$};
		\node at (2,4) {$1$};
		\node at (5,5) {$1$};
		\node at (8,6) {$1$};
		\node at (3,7) {$1$};
		\node at (6,8) {$1$};
		\node at (9,9) {$1$};
	\end{tikzpicture}
	\quad
	\begin{tikzpicture}[scale=0.3, baseline=(current bounding box.center)]
		\node at (0,4.5) {$\longrightarrow$};
	\end{tikzpicture}
	\quad
	\begin{tikzpicture}[scale=0.3, baseline=(current bounding box.center)]
		\foreach \i in {0.5, 3.5, 6.5, 9.5}{
			\draw (\i,0.5)--(\i,9.5);
			\draw (0.5,\i)--(9.5,\i);
		}
		\foreach \i in {2,5,8}{
			\foreach \j in {2,5,8}{
				\node at (\i,\j) {$1$};
			}
		}
	\end{tikzpicture}
\end{center}
\end{footnotesize}
\caption{A $9\times 9$ permutation matrix containing $J_3$ as an interval minor.}
\label{fig-perm-J3}
\end{figure}
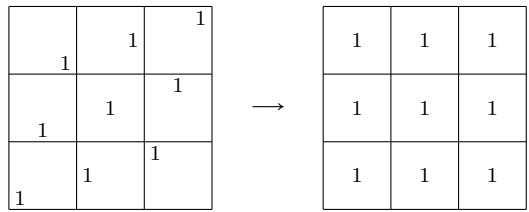


The link between avoiding $J_k$ and lower bounds on growth rates comes via the following fact, whose truth is demonstrated by Figure~\ref{fig-perm-J3}.

\begin{observation}
\label{obs-Jk-minor}
For every nonnegative integer $k$, there is a permutation $\beta$ of length $k^2$ whose permutation matrix contains $J_k$ as an interval minor.
\end{observation}

Therefore if many permutation matrices avoid $J_k$ as an interval minor, they also avoid $M_\beta$ for some permutation $\beta$ of length $k^2$.

\subsection{The number of light matrices}
\label{subsec-klazar}

The main result of this subsection provides the link between the extremal F\"uredi--Hajnal Conjecture and the enumerative Stanley--Wilf Conjecture. We state it in a slightly different way than it was originally presented. Let us define the \emph{density} of the matrix $M$ as
$$
\delta(M)
=
\max\left\{\frac{\wt(P)}{m} \st \mbox{$P$ is an $m\times m$ interval minor of $M$}\right\}.
$$
Note that this definition measures the maximum average number of ones per row over all square interval minors of $M$. In particular, the density of a matrix can be greater than $1$; for example, the permutation matrix shown on the left of Figure~\ref{fig-perm-J3} has density at least $3$, because it contains $J_3$ as a minor.

We seek to bound the number of matrices of small density.

\begin{theorem}[Klazar~\cite{klazar:the-furedi-hajn:}]
\label{thm-klazar}
Let $d\ge 0$ be a constant. There are fewer than $15^{2dn}$ matrices of size $n\times n$ and density at most $d$.
\end{theorem}%

\begin{proof}
We prove the theorem by induction on $n$. As the base case $n=1$ is trivial, we may assume that $n\ge 2$ and that the claim holds for all lesser values of $n$.

Let $M$ be an $n\times n$ matrix of density at most $d$ and denote by $M^{/2}$ the interval minor of $M$ formed by using the row and column divisions $\{1,3,5,\dots,n+1\}$, without deleting any ones. Thus $M^{/2}$ is formed by contracting $2\times 2$ blocks of $M$ (or possibly smaller blocks on the top and right), as shown in Figure~\ref{fig-perm-klazar}.

\begin{figure}
\begin{footnotesize}
\begin{center}
	\begin{tikzpicture}[scale=0.3, baseline=(current bounding box.center)]
		\foreach \i in {0.5, 2.5, 4.5, 6.5, 8.5, 9.5}{
			\draw (\i,0.5)--(\i,9.5);
			\draw (0.5,\i)--(9.5,\i);
		}
		\node at (1,2) {$1$};
		\node at (1,6) {$1$};
		\node at (2,1) {$1$};
		\node at (2,2) {$1$};
		\node at (3,3) {$1$};
		\node at (4,4) {$1$};
		\node at (4,9) {$1$};
		\node at (5,7) {$1$};
		\node at (5,8) {$1$};
		\node at (6,5) {$1$};
		\node at (6,7) {$1$};
		\node at (7,2) {$1$};
		\node at (7,5) {$1$};
		\node at (7,6) {$1$};
		\node at (7,7) {$1$};
		\node at (8,5) {$1$};
		\node at (8,6) {$1$};
		\node at (8,7) {$1$};
		\node at (9,3) {$1$};
		\node at (9,4) {$1$};
		\node at (9,5) {$1$};
	\end{tikzpicture}
	\quad
	\begin{tikzpicture}[scale=0.3, baseline=(current bounding box.center)]
		\node at (0,4.5) {$\longrightarrow$};
	\end{tikzpicture}
	\quad
	\begin{tikzpicture}[scale=0.3, baseline=(current bounding box.center)]
		\foreach \i in {0.5, 2.5, 4.5, 6.5, 8.5, 9.5}{
			\draw (\i,0.5)--(\i,9.5);
			\draw (0.5,\i)--(9.5,\i);
		}
		\node at (1.5,1.5) {$1$};
		\node at (1.5,5.5) {$1$};
		\node at (3.5,3.5) {$1$};
		\node at (3.5,9) {$1$};
		\node at (5.5,5.5) {$1$};
		\node at (5.5,7.5) {$1$};
		\node at (7.5,1.5) {$1$};
		\node at (7.5,5.5) {$1$};
		\node at (7.5,7.5) {$1$};
		\node at (9,3.5) {$1$};
		\node at (9,5.5) {$1$};
	\end{tikzpicture}
\end{center}
\end{footnotesize}
\caption{The interval minor $M^{/2}$ formed in the proof of Klazar's Theorem~\ref{thm-klazar}.}
\label{fig-perm-klazar}
\end{figure}
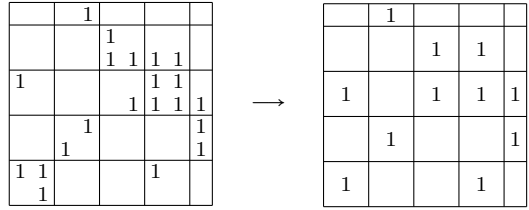

By the definition of density, $\wt(M^{/2})\le d\lceil n/2\rceil$. Now we ask how many $n\times n$ matrices could condense to a given $\lceil n/2\rceil\times \lceil n/2\rceil$ matrix $M^{/2}$. Each nonzero entry of $M^{/2}$ can come from at most $15$ possible blocks of $M$ (all blocks except for the all-zero block), while the zero entries of $M^{/2}$ can only come from an all-zero block of $M$. Since $M^{/2}$ has at most $d\lceil n/2\rceil$ nonzero entries, the number of $n\times n$ matrices that could condense to it is at most $15^{d\lceil n/2\rceil}$.

Therefore, the number of $n\times n$ matrices of density at most $d$ is at most $15^{d\lceil n/2\rceil}$ times the number of possibilities for $M^{/2}$, which by induction is less than $15^{2d\lceil n/2\rceil}$. In other words, the number of $n\times n$ matrices of density at most $d$ is less than
$$
15^{d\lceil n/2\rceil}\cdot 15^{2d\lceil n/2\rceil}
=
15^{3d\lceil n/2\rceil}
<
15^{2dn}
$$
for $n\ge 2$, completing the proof of the theorem.
\end{proof}

This result shows that the F\"uredi--Hajnal Conjecture implies the Stanley--Wilf Conjecture as follows: if the F\"uredi--Hajnal Conjecture is true, then for every permutation $\beta$ there is a constant $c_\beta$ such that $\exop(n; M_\beta)\le c_\beta n$ for all $n$. Therefore every $M_\beta$-avoiding matrix has density at most $c_\beta$. Klazar's Theorem~\ref{thm-klazar} shows that there are at most $(15^{2c_\beta})^n$ such matrices, so there are only exponentially many such permutation matrices, so $\gr(\Av(\beta))$ is finite.

As we have presented it, the bound in Klazar's Theorem~\ref{thm-klazar} is essentially best possible. To see this, let $M$ be any $n\times n$ matrix with $dn$ ones and density $d$. Then $M$ itself contains $2^{dn}=(2^d)^n$ submatrices of size $n\times n$ and density at most $d$. Thus to improve on the link between bounds in the F\"uredi--Hajnal Conjecture and those in the Stanley--Wilf Conjecture, one must restrict to permutation matrices. Via some quite delicate bounds, Cibulka~\cite{cibulka:on-constants-in:} showed that the number of $n\times n$ permutation matrices of density at most $d$ is at most $(2.88d^2)^n$. Fox~\cite{fox:stanley-wilf-li:} gives a much simpler proof that the $n$th root of this number is $O(d^2)$.


\subsection{Matrices avoiding $J_k$ are light}
\label{subsec-mt}

We now present Marcus and Tardos' proof of the F\"uredi--Hajnal Conjecture. The general approach is again to consider a specific contraction.

\begin{theorem}[Marcus and Tardos~\cite{marcus:excluded-permut:}]
\label{thm-marcus-tardos}
If the $n\times n$ matrix $M$ does not contain $J_k$ as an interval minor then
$$
\delta(M)\le 2k^4{k^2\choose k}.
$$
\end{theorem}
\begin{proof}
Let $f_k(n)$ denote the maximum weight of an $n\times n$ matrix that avoids $J_k$ as an interval minor. We apply induction on $n$, noting that the $n=1$ case is trivial. Let $M$ be an $n\times n$ matrix of weight $f_k(n)$ not containing $J_k$ as an interval minor. We form the interval minor $M^{/q}$ by choosing our row and column divisions to be $\{1,q+1,2q+1,\dots n+1\}$; here $q$ is a parameter to be determined at the end of the proof. Thus we are contracting $M$ into $q\times q$ blocks (possibly including smaller blocks on the top and right). Of course, $M^{/q}$ also does not contain $J_k$ as an interval minor.

The cleverest part of the proof is the following definition. A $q\times q$ block of $M$ is called \emph{tall} if it contains ones in at least $k$ different rows and/or \emph{wide} if it contains ones in at least $k$ different columns. We begin by bounding the number of tall and/or wide blocks of $M$.

Suppose that more than $(k-1){q\choose k}$ tall blocks of $M$ correspond to entries in the same row of $M^{/q}$. We contract each such block into a single column. This gives an interval minor of $M$ containing $q$ rows and $(k-1){q\choose k}$ columns in which every column contains at least $k$ ones. By deleting extraneous ones as needed, we may assume that each column has precisely $k$ ones. However, that leaves us with only ${q\choose k}$ choices of columns for this interval minor, and thus (under our assumption about the number of tall blocks in this row) at least one column must occur at least $k$ times, showing that $M$ contains $J_k$ as an interval minor, a contradiction.

By symmetry, at most $(k-1){q\choose k}$ entries of each column of $M^{/q}$ correspond to wide blocks in $M$. As $M^{/q}$ has $\lceil n/q\rceil$ rows and columns, we have the bound
$$
\mbox{$\#$ of tall and/or wide entries of $M^{/q}$}\le 2\left\lceil\frac{n}{q}\right\rceil(k-1){q\choose k},
$$
which is linear in $n$. For these entries of $M^{/q}$ we use the trivial bound that they correspond to blocks with at most $q^2$ nonzero entries. The remaining entries of $M^{/q}$, corresponding to blocks that are neither tall nor wide, can have at most $(k-1)^2$ nonzero entries. Putting these bounds together, we obtain by induction that
\[
    \wt(M)
    \le
    2q^2\left\lceil\frac{n}{q}\right\rceil(k-1){q\choose k}
    +
    (k-1)^2f_k\left(\left\lceil \frac{n}{q}\right\rceil\right).
\]
In order to get a sense of the solution to this recurrence, we iterate it approximately:
\begin{eqnarray*}
f_k(n)
&\lesssim&
2qn(k-1){q\choose k}+(k-1)^2f_k\left(\frac{n}{q}\right),\\
&\lesssim&
2qn(k-1){q\choose k}+(k-1)^2
\left(2n(k-1){q\choose k}+(k-1)^2f_k\left(\frac{n}{q^2}\right	)\right),\\
&\lesssim&
2qn(k-1){q\choose k}
+
2n(k-1)^3{q\choose k}
+
2\frac{n}{q}(k-1)^5{q\choose k}
+\cdots.
\end{eqnarray*}
We can further approximate this quantity by viewing it as a geometric series with ratio $(k-1)^2/q$:
$$
f_k(n)
\lesssim
2qn(k-1){q\choose k}\sum_{i=0}^\infty\left(\frac{(k-1)^2}{q}\right)^i.
$$
This series has a finite sum so long as we choose $q>(k-1)^2$, in which case we get a linear upper bound on $\wt(M)$, completing the proof. In particular, setting $q=k^2$, we see that
$$
\sum_{i=0}^\infty \left(\frac{(k-1)^2}{k^2}\right)^i=\frac{k^2}{2k-1}<k,
$$
and the bound becomes
$$
f_k(n)
\lesssim
2k^4{k^2\choose k}n.
$$
While we presented only estimates above in order to motivate the choice of $q$, with this bound on $f_k(n)$ now discovered it would only take a simple inductive argument to give a rigorous proof.
\end{proof}

We have proved quite a bit more than the Stanley--Wilf Conjecture. Every proper downset of matrices in the interval minor order must avoid some matrix $J_k$. Thus every proper downset of matrices in this order has bounded density by Marcus and Tardos' Theorem~\ref{thm-marcus-tardos} and thus grows at most exponentially by Klazar's Theorem~\ref{thm-klazar}. Proper permutation classes are merely subsets of these downsets.

\subsection{Dense matrices contain many permutations}
\label{subsec-cibulka}

Having established upper bounds on growth rates, we now turn our attention to lower bounds. In this subsection we prove a simplified version of a result of Cibulka~\cite{cibulka:on-constants-in:}, showing that dense matrices contain many permutation matrices. In the following subsection, we use this result to show that growth rates of principal classes can be very large. Our first result in this direction, below, gives a lower bound for how many permutation matrices a dense matrix with a certain amount of regularity must contain.

\begin{proposition}[Cibulka~\cite{cibulka:on-constants-in:}]
\label{prop-cibulka-lemma3}
Suppose that the matrix $M$ has $n$ columns and $m\ge n$ rows, and that each column contains at least $n$ ones. Then $M$ contains at least $n!/{m\choose n}$ permutation matrices of size $n\times n$.
\end{proposition}
\begin{proof}
Proceeding from left to right, we can find at least $n!$ copies of $n\times n$ permutation matrices in $M$ by choosing a nonzero entry in a new row from each column. Of course, some permutation matrices may be contained in $M$ many times. However, $M$ cannot contain more than ${m\choose n}$ copies of a given $n\times n$ permutation matrix because that is the number of ways to select the rows from which the nonzero entries are then chosen.
\end{proof}

The bound we actually use follows from Proposition~\ref{prop-cibulka-lemma3} via Stirling's Formula:
$$
\frac{n!}{{m\choose n}}
\ge
\frac{n!}{m^n/n!}
=
\frac{(n!)^2}{m^n}
\ge
\left(\frac{n^2}{e^2m}\right)^n.
$$

The main result of this subsection establishes the link between the weight of a matrix and the number of permutation matrices contained in it. Note that the conclusion of this theorem is only interesting for large densities $\delta(M)$.

\begin{theorem}[based on Cibulka~{\cite[Theorem 6]{cibulka:on-constants-in:}}]
\label{thm-cibulka}
For every $n\times n$ matrix $M$, there is some value of $m\ge 1$ such that $M$ contains $\gamma^m$ or more $m\times m$ permutation matrices, where $\gamma=\delta(M)^{1/9}/16$.
\end{theorem}
\begin{proof}
The argument resembles that used to prove Marcus and Tardos' Theorem~\ref{thm-marcus-tardos}, though with an additional clever idea and different parameters. In particular, instead of being given $k$ (formerly the size of the forbidden interval minor), here we choose $k$ to be the smallest perfect square such that $k\ge e^4\gamma^2$. Clearly we can satisfy this requirement for some $k$ at most $2e^4\gamma^2$.

Define $f_\gamma(n)$ to be the greatest possible weight of an $n\times n$ matrix that does not contain $\gamma^m$ or more $m\times m$ permutation matrices for any value of $m$. The theorem is equivalent to the fact that
\begin{equation}
\label{eqn-cibulka}
f_\gamma(n)<(16\gamma)^9 n,
\end{equation}
which we establish by induction on $n$. If $n\le (16\gamma)^9$, then \eqref{eqn-cibulka} holds trivially. Thus we may assume that $n>(16\gamma)^9$ and take $M$ to be an $n\times n$ matrix of weight $f_\gamma(n)$ that does not contain $\gamma^m$ or more $m\times m$ permutation matrices for any value of $m$.

Again we begin by forming the interval minor $M^{/q}$ by choosing our row and column divisions to be $\{1,q+1,2q+1,\dots n+1\}$, although this time we set $q=4k^2$ instead of $k^2$. We have by induction that $\wt(M^{/q})<(16\gamma)^9\lceil n/q\rceil$. The definitions of tall and wide are unchanged: a $q\times q$ block of $M$ is tall (respectively, wide) if it contains ones in at least $k$ different rows (respectively, columns).

The next step is to bound the number of tall/wide blocks of $M$, but this time we obtain a different bound because of our assumption that $M$ contains few permutation matrices. We bound the number of tall blocks that correspond to a single row of $M^{/q}$, since the bound for wide blocks in a column will follow by symmetry.

Given a set of tall blocks corresponding to entries in the same row of $M^{/q}$, we group them into contiguous sets of $k$ tall blocks apiece. Our bound comes from considering two different types of groups. First, if the ones in all of the blocks of a group of tall blocks lie in only $k^{3/2}$ rows, we call the group \emph{concentrated}. In this case, we form a minor by removing the all-zero rows of the block (this can be viewed as a contraction) and then contracting each block of the group to a single column. In doing so we obtain a matrix with $k$ columns and at most $k^{3/2}$ rows in which every column contains at least $k$ ones (because the original blocks were tall). Applying Proposition~\ref{prop-cibulka-lemma3} to this interval minor shows that it (and thus also $M$) contains at least
$$
\frac{k!}{{k^{3/2}\choose k}}
\ge
\left(\frac{k^2}{e^2k^{3/2}}\right)^k
=
\left(\frac{\sqrt{k}}{e^2}\right)^k
\ge
\gamma^k
$$
$k\times k$ permutation matrices. Thus $M$ cannot contain a concentrated group of tall blocks.

It remains to consider groups of tall blocks that are not concentrated. If we have a nonconcentrated group of tall blocks then these blocks can all be contracted into a single column with at least $k^{3/2}$ ones. Thus if $M$ contains $k^{3/2}$ nonconcentrated groups then by contracting each group into a single column we obtain a matrix with $k^{3/2}$ columns and $q=2k^2$ rows in which each column contains at least $k^{3/2}$ ones. Applying Proposition~\ref{prop-cibulka-lemma3} to this interval minor shows that it contains at least
$$
\frac{(k^{3/2})!}{{2k^2\choose k^{3/2}}}
\ge
\left(\frac{k^3}{2e^2k^2}\right)^{k^{3/2}}
=
\left(\frac{k}{2e^2}\right)^{k^{3/2}}
\ge
\left(\frac{e^2\gamma^2}{2}\right)^{k^{3/2}}
>
\gamma^{k^{3/2}}
$$
$k^{3/2}\times k^{3/2}$ permutation matrices. Therefore $M$ cannot contain $k^{3/2}$ nonconcentrated groups within the same row of blocks.

We now count the ones in $M$. A given row of $M^{/q}$ must correspond to fewer than $k^{5/2}$ tall blocks because when arranged in groups of $k$ tall blocks apiece, no group can be concentrated, and $M$ cannot contain $k^{3/2}$ nonconcentrated groups. Using the trivial bound that a $q\times q$ block can contain at most $q^2$ ones, we see that the combined weight of an entire row of tall blocks is at most $q^2k^{5/2}$. As there are $\lceil n/q\rceil$ such rows of blocks and also $\lceil n/q\rceil$ columns of blocks that are subject to analogous constraints, the total weight of all tall and/or wide blocks in $M$ is less than $2\lceil n/q\rceil q^2k^{5/2}$.

Next we must bound the weight of the blocks that are neither tall nor wide. Each such block has weight less than $k^2$. Moreover, the number of such blocks is at most $\wt(M^{/q})<(16\gamma)^9\lceil n/q\rceil$. Combining our bounds (and remembering that $q=4k^2$ while $k\le 2e^4\gamma^2$) shows that
\begin{eqnarray*}
\wt(M)
&<&
2\left\lceil \frac{n}{q}\right\rceil q^2k^{5/2} + (16\gamma)^9\left\lceil\frac{n}{q}\right\rceil k^2,\\
&<&
\left(4qk^{5/2} + 2\frac{(16\gamma)^9 k^2}{q}\right)n,\\
&=&
\left(16k^{9/2} + \frac{(16\gamma)^9}{2}\right)n.
\end{eqnarray*}
It can be now be checked that
$$
16k^{9/2}
\le
16\cdot 2^{9/2}e^{18}\gamma^9
<
\frac{(16\gamma)^9}{2},
$$
establishing \eqref{eqn-cibulka} and completing the proof.
\end{proof}

The clever new idea in the proof of Theorem~\ref{thm-cibulka} is that of concentration, which allows one to handle tall/wide blocks in either of two different ways. In Cibulka's original proof he also refined the notion of tall/wide blocks, distinguishing between tall, very tall, and ultratall blocks. With these additional considerations, he was able to lower the $9$ in the statement of the theorem to $4.5$.


%
%
%
%
%
%
%
%
%
%
%
%
%
%
%
%
%
%
%
%
%
%
%
%
%
%
%
%
%
%
%
%
%
%
%
%
%
%
%
%

\subsection{Dense matrices avoiding $J_k$}
\label{subsec-fox}

Our final result of this section shows that the formerly widely held belief that $\gr(\Av(\beta))$ grows quadratically in $|\beta|$ is false. This result, proved by Fox~\cite{fox:stanley-wilf-li:}, is established by means of an elegant construction.

Throughout this section we are concerned with a matrix $M$ of size $n\times n$ where $n$ is a power of $2$, say $n=2^r$. We divide the interval $[1,2^r]$ into a number of \emph{dyadic intervals}. The interval $[1,2^r]$ is itself dyadic. Its two halves, $[1,2^{r-1}]$ and $[2^{r-1}+1,2^r]$ are also dyadic. In turn, the halves of each of those intervals are dyadic, and this process continues until we reach the singletons, which are all dyadic. More formally, a dyadic interval is any interval of the form $[a2^b+1,(a+1)2^b]$ for nonnegative integers $a$ and $b$. A \emph{dyadic rectangle} is then a product of two dyadic intervals (in the same sense as the products we used to define contraction).

The number of dyadic intervals in $[1,2^r]$ is
$$
\sum_{b=0}^r 2^b=2^{r+1}-1=2n-1,
$$
and every element of $[1,n]$ lies in precisely $r+1$ such dyadic intervals (one of each size). The final observation we need to make about dyadic intervals is their most important property, at least from the viewpoint of our upcoming construction. Suppose that $I_1$ and $I_2$ are subintervals of $[1,n]$, and for $i\in\{1,2\}$ let $D_i$ denote the smallest dyadic interval containing $I_i$. If $D_1=D_2$, then both $I_1$ and $I_2$ must contain elements of both halves of this dyadic interval. Therefore we can conclude that the smallest dyadic intervals containing two \emph{disjoint} subintervals of $[1,n]$ are \emph{distinct}.

We are now ready to describe the construction of $M$. Order the dyadic subintervals of $[1,n]$ as $D_1,\dots,D_{2n-1}$ arbitrarily and let $A$ be a $(2n-1)\times (2n-1)$ \emph{auxiliary matrix}. We build $M$ from $A$ by setting
$$
M(i,j)
=
\left\{
\begin{array}{cl}
1&\mbox{if $A(k,\ell)=1$ for every pair $k,\ell$ with $(i,j)\in D_k\times D_\ell$,}\\
0&\mbox{otherwise.}
\end{array}
\right.
$$
The following result illustrates the usefulness of this construction.

\begin{proposition}[Fox~\cite{fox:stanley-wilf-li:}]
\label{prop-aux-matrix}
Suppose that $A$ and $M$ are related as above. If $A$ avoids $J_k$ as a submatrix, then $M$ avoids $J_k$ as an interval minor.
\end{proposition}
\begin{proof}
Suppose that $M$ contains $J_k$ as an interval minor. Thus there are row and column divisions $1=r_1\le\cdots\le r_{k+1}=n+1$ and $1=c_1\le\cdots\le c_{k+1}=n+1$ such that for every $i$ and $j$, the block $M([c_i,c_{i+1})\times [r_j,r_{j+1}))$ contains a nonzero entry.

Now choose indices $\overline{c}_1,\dots,\overline{c}_k$ such that for every $i$, the smallest dyadic interval containing $[c_i,c_{i+1})$ is $D_{\overline{c}_i}$ and similarly choose indices $\overline{r}_1,\dots,\overline{r}_k$ such that for every $j$, the smallest dyadic interval containing $[r_j,r_{j+1})$ is $D_{\overline{r}_j}$. Our observation above shows that because the intervals we are considering are disjoint, these are sets of distinct indices. Thus it follows by the construction of $M$ that $A(\overline{c}_i,\overline{r}_j)=1$ for all indices $i$ and $j$, i.e., that $A$ contains $J_k$ as a submatrix, completing the proof.
\end{proof}

As one might expect for an extremal result, the proof of Fox's Theorem is probabilistic. In order to find a dense matrix $M$ avoiding $J_k$ as an interval minor we choose the entries of the auxiliary matrix $A$ uniformly at random and then (essentially) construct $M$ as above. While this part of the proof is conceptually straight-forward, the reader may find the requisite inequalities a bit laborious.

\begin{theorem}[Fox~\cite{fox:stanley-wilf-li:}]
\label{thm-fox-bound}
For all sufficiently large positive integers $k$ such that $\sqrt{k}/8$ is an integer, there exists a matrix of density at least $2^{\sqrt{k}/16}$ that avoids $J_k$ as an interval minor.
\end{theorem}
\begin{proof}
Suppose that $r=\sqrt{k}/8$ is an integer and set $n=2^r$ so that every integer in the interval $[1,n]$ lies in precisely $r+1$ dyadic intervals. The proof uses a parameter $q$, which is a probability that will be determined near the end and will satisfy $1/2>q\ge 4r/k$.

We begin by building the auxiliary matrix $A$ of size $(2n-1)\times (2n-1)$. Choose each entry of $A$ uniformly at random so that it is $1$ with probability $1-q$. Let $X$ denote the random variable counting the number of copies of $J_k$ occurring as a submatrix in $A$. Clearly
$$
\mathbb{E}[X]={2n-1\choose k}^2(1-q)^{k^2}.
$$
The standard bound $1-x\le e^{-x}$ shows that $(1-q)^{k^2}\le e^{-qk^2}$. For the binomial coefficient, we have
$$
{2n-1\choose k}<\frac{(2n)^k}{k!}<\frac{(2n)^k}{2^k}=n^{k},
$$
where the second inequality follows because $k\ge 4$. Combining these bounds yields
$$
\mathbb{E}[X]<n^{2k}e^{-qk^2}.
$$
The right-hand side of this inequality is maximized when $q$ is as small as possible. As
$$
q\ge \frac{4r}{k}=\frac{4\log_2 n}{k}>\frac{4\ln n}{k},
$$
we see that
$$
\mathbb{E}[X]<n^{2k}e^{-4k\ln n}=n^{-2k}.
$$
In particular,
$$
\operatorname{Pr}[\mbox{$A$ contains $J_k$ as a submatrix}]
\le
\mathbb{E}[X]
<
n^{-2k},
$$
so almost all such matrices avoid $J_k$ as a submatrix.

Now construct $M$ as in Proposition~\ref{prop-aux-matrix}. The probability that an entry of $M$ is equal to $1$ is thus $(1-q)^{(r+1)^2}$ because every entry lies in $(r+1)^2$ dyadic rectangles. It is worth remarking that the entries of $M$ are highly correlated (when viewed as $n^2$ random variables). In particular, the probability that $M$ is the zero matrix is at least $q$, because one of the entries of $A$ corresponds to the dyadic rectangle $[1,n]\times [1,n]$, and thus must be $1$ if $M$ is to have any nonzero entries. Of course, this correlation does not prevent us from appealing to linearity of expectation, which shows that
$$
\mathbb{E}[\wt(M)]=n^2(1-q)^{(r+1)^2}.
$$

However, we must guarantee that $M$ avoids $J_k$ as an interval minor, and this construction does not provide such a guarantee (though it makes it exceedingly likely). Therefore we take $M$ as above if it avoids $J_k$ as an interval minor (i.e., if $A$ avoids $J_k$ as a submatrix), and otherwise set $M$ equal to the zero matrix. From our previous computation of $\mathbb{E}[X]$, we see that
\begin{eqnarray*}
\mathbb{E}[\wt(M)]
&\ge&
n^2(1-q)^{(r+1)^2} - n^2\cdot\operatorname{Pr}[\mbox{$A$ contains $J_k$ as a submatrix}],\\
&>&
n^2(1-q)^{(r+1)^2} - n^{2-2k},\\
&>&
n^2(1-q)^{(r+1)^2} - 1.
\end{eqnarray*}
Thus there is an $n\times n$ matrix $M$ avoiding $J_k$ as an interval minor with weight at least this expectation.

To finish the proof we need to choose $q$ and then manipulate the inequalities to show that the density of $M$ is at least $2^{\sqrt{k}/16}$, which is equivalent to showing that $\wt(M)\ge n^{3/2}$. For the rest of the proof we assume that $k\ge 48^2=2308$ and set $q=1/\sqrt{k}$. As promised at the beginning of the proof,
$$
\frac{1}{2}
>
q
=
\frac{1}{\sqrt{k}}
>
\frac{4(\sqrt{k}/8)}{k}
=
\frac{4r}{k}.
$$
We now claim that for this value of $q$ and the particular matrix $M$ we have selected,
\begin{equation}
\label{eqn-fox-1}
\wt(M)
\ge
n^2(1-q)^{(r+1)^2}-1>n^2 2^{-3qr^2}-1>n^{3/2}.
\end{equation}
Of these three inequalities, the first follows from our choice of $M$ and the second requires the most work. Canceling the $1$ and the $n^2$, we want to show that $(1-q)^{(r+1)^2}>n^{-3qr}$. Taking logarithms of both sides, this is equivalent to
\begin{equation}
\label{eqn-fox-2}
(r+1)^2\log(1-q)>-(3\log 2) qr^2
\end{equation}
For $q<1/2$, we have the bound $\log(1-q)>-3q/2$ so the left-hand side of \eqref{eqn-fox-2} can be bounded by
$$
(r+1)^2\log(1-q)
>
-\frac{3}{2}(r+1)^2q.
$$
For $r\ge 6$ (which we have because $k\ge 48^2$),
$$
-\frac{3}{2}(r+1)^2>-(3\log 2)r^2\ \left(\approx -2.08 r^2\right),
$$
verifying \eqref{eqn-fox-2} and thus the second inequality of \eqref{eqn-fox-1}.

All that remains is to show the final inequality of \eqref{eqn-fox-1}. Recall that $n=2^r$, so
$$
n^2 2^{-3qr^2}-1
=
n^{2-3qr}-1
=
n^{2-3/8}-1.
$$
This quantity is greater than $n^{3/2}=2^{\sqrt{k}/16}n$ for $n\ge 4$, completing the proof.
\end{proof}

It remains only to connect this result to growth rates of principal classes, which we do via our previous definition of
$$
g(k)=\max\{\gr(\Av(\beta))\st |\beta|=k\}.
$$
Choose $k$ large enough so that Theorem~\ref{thm-fox-bound} holds. Recall by Observation~\ref{obs-Jk-minor} that there are permutations of length $k^2$ that contain a $J_k$ minor. Let $\beta$ be such a permutation and set $\ell=k^2$. Fox's Theorem shows that there is a matrix $M$ of density at least $2^{\ell^{1/4}/16}$ which avoids $J_k$ as a minor and thus also avoids the permutation matrix of $\beta$. Cibulka's Theorem~\ref{thm-cibulka} then shows that there is a value of $m$ such that $M$ contains $\gamma^m$ or more $m\times m$ permutation matrices, where
$$
\gamma
=
\frac{\delta(M)^{1/9}}{16}
=
\frac{2^{\ell^{1/4}/144}}{16}
=
2^{{\ell^{1/4}/144}-4}.
$$
Every permutation matrix contained in $M$ also avoids the permutation matrix of $\beta$. From our observation in Section~\ref{subsec-basics} about supermultiplicativity it follows that $\gr(\Av(\beta))\ge\gamma$, showing that
$$
g(k)=2^{\Omega(k^{1/4})}.
$$

Thus $g(k)$ grows faster than every polynomial. (There is no commonly agreed upon term for functions that grow like $2^{n^\alpha}$ for $0<\alpha<1$. Because such functions grow much more slowly than $2^n$, they are often called \emph{subexponential} when found as upper bounds for algorithmic problems. However, this term does not accurately convey how quickly such functions do grow, so terms such as \emph{stretched exponential} and \emph{mildly exponential} are also used.)



%
%
%
%
%
%
%
%
%
%
%
%
%
%
%
%
%
%
%
%
%
%

\section{Notions of structure}
\label{sec-structure}

We saw in the previous section that proper permutation classes have finite upper growth rates. However, none of those techniques are of any use if we want to know the precise growth rate of a class. Indeed, even if we only want to estimate growth rates, it is hard to imagine those techniques providing much insight. Thus it seems that to know more about growth rates of permutation classes we must take a detailed look at their structure.

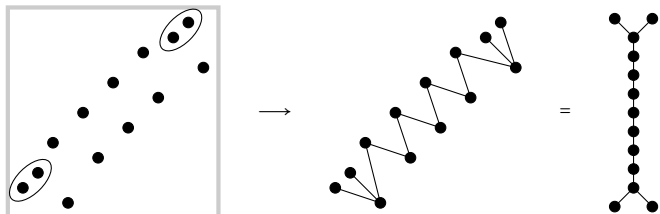
\begin{figure}
\begin{footnotesize}
\begin{center}
	\begin{tikzpicture}[scale=0.2, baseline=(current bounding box.center)]
		\draw [lightgray, ultra thick, line cap=round] (0,0) rectangle (14,14);
		\draw[fill=black] (1,2) circle (10pt);
		\draw[fill=black] (2,3) circle (10pt);
		\draw[fill=black] (3,5) circle (10pt);
		\draw[fill=black] (4,1) circle (10pt);
		\draw[fill=black] (5,7) circle (10pt);
		\draw[fill=black] (6,4) circle (10pt);
		\draw[fill=black] (7,9) circle (10pt);
		\draw[fill=black] (8,6) circle (10pt);
		\draw[fill=black] (9,11) circle (10pt);
		\draw[fill=black] (10,8) circle (10pt);
		\draw[fill=black] (11,12) circle (10pt);
		\draw[fill=black] (12,13) circle (10pt);
		\draw[fill=black] (13,10) circle (10pt);
		\draw[rotate around={-45:(1.5,2.5)}] (1.5,2.5) ellipse (25pt and 50pt);
		\draw[rotate around={-45:(11.5,12.5)}] (11.5,12.5) ellipse (25pt and 50pt);
	\end{tikzpicture}
	\quad
	\begin{tikzpicture}[scale=0.2, baseline=(current bounding box.center)]
		\node at (0,6.5) {$\longrightarrow$};
	\end{tikzpicture}
	\quad
	\begin{tikzpicture}[scale=0.2, baseline=(current bounding box.center)]
		\draw[fill=black] (1,2) circle (10pt);
		\draw[fill=black] (2,3) circle (10pt);
		\draw[fill=black] (3,5) circle (10pt);
		\draw[fill=black] (4,1) circle (10pt);
		\draw[fill=black] (5,7) circle (10pt);
		\draw[fill=black] (6,4) circle (10pt);
		\draw[fill=black] (7,9) circle (10pt);
		\draw[fill=black] (8,6) circle (10pt);
		\draw[fill=black] (9,11) circle (10pt);
		\draw[fill=black] (10,8) circle (10pt);
		\draw[fill=black] (11,12) circle (10pt);
		\draw[fill=black] (12,13) circle (10pt);
		\draw[fill=black] (13,10) circle (10pt);
		\draw (1,2)--(4,1);
		\draw (2,3)--(4,1);
		\draw (3,5)--(4,1);
		\draw (3,5)--(6,4);
		\draw (5,7)--(6,4);
		\draw (5,7)--(8,6);
		\draw (7,9)--(8,6);
		\draw (7,9)--(10,8);
		\draw (9,11)--(10,8);
		\draw (9,11)--(13,10);
		\draw (11,12)--(13,10);
		\draw (12,13)--(13,10);
	\end{tikzpicture}
	\quad
	\begin{tikzpicture}[scale=0.2, baseline=(current bounding box.center)]
		\node at (0,6.5) {$=$};
	\end{tikzpicture}
	\quad
	\begin{tikzpicture}[scale=0.125, baseline=(current bounding box.center)]
		\draw[fill=black] (-2,1) circle (16pt);
		\draw[fill=black] (2,1) circle (16pt);
		\foreach \i in {3,5,7,9,11,13,15,17,19}{
			\draw[fill=black] (0,\i) circle (16pt);
		}
		\draw[fill=black] (-2,21) circle (16pt);
		\draw[fill=black] (2,21) circle (16pt);
		\draw (-2,1)--(0,3);
		\draw (2,1)--(0,3);
		\draw (0,3)--(0,19);
		\draw (0,19)--(-2,21);
		\draw (0,19)--(2,21);
	\end{tikzpicture}
\end{center}
\end{footnotesize}
\caption{A member of the infinite antichain commonly denoted $U$.}
\label{fig-infinite-antichain}
\end{figure}

As we will see, there are a great many aspects of ``structure'' for permutation classes that have so far resisted unification. For a first aspect of structure, consider the permutation shown on the left of Figure~\ref{fig-infinite-antichain}. In the middle and right of this figure, two drawings of the \emph{permutation graph}, $G_\pi$, of this permutation are presented. This is the graph on the vertices $\{(i,\pi(i))\}$ in which $(i,\pi(i))$ and $(j,\pi(j))$ are adjacent if they form an inversion, i.e., $i<j$ and $\pi(i)>\pi(j)$. The drawing on the right of this figure shows that the permutation graph of this permutation is a \emph{split-end path}, that is, it is constructed from a path by adding four vertices, two adjacent to each leaf. Clearly we could modify this construction to build an infinite set of permutations whose permutation graphs are all split-end paths.

\index{infinite antichain}

It is easy to see that if $\sigma\le\pi$ then $G_\sigma$ is an induced subgraph of $G_\pi$, though the converse need not hold (for example, $G_\pi$ and $G_{\pi^{-1}}$ are isomorphic, but of course unless $\pi$ is an involution, it is not contained in $\pi^{-1}$). Therefore, since no split-end path is an induced subgraph of another, it follows that the infinite set of permutations built this way is an infinite antichain (recall that an antichain is a set of pairwise incomparable permutations). The same approach shows that the family of permutations alluded to in Figure~\ref{fig-inc-inc-basis} also forms infinite antichain; the graph of the permutation from that figure is shown below.
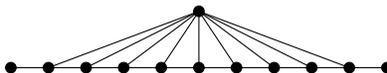
\begin{center}
	\begin{tikzpicture}[scale=0.25, baseline=(current bounding box.center)]
		\draw (3,0)--(23,0);
		\foreach \i in {5,7,9,11,13,15,17,19,21}{
			\draw (\i,0)--(13,3);
		}
		\foreach \i in {3,5,7,9,11,13,15,17,19,21,23}{
			\draw[fill=black] (\i,0) circle (8pt);
		}
		\draw[fill=black] (13,3) circle (8pt);
	\end{tikzpicture}
\end{center}

\index{well-quasi-order}
\index{well-partial-order}

A \emph{well-quasi-order (wqo)} is a \emph{quasi-order} (a binary and transitive, but not necessarily antisymmetric, binary relation) that contains neither an infinite strictly descending sequence nor an infinite antichain. Because containment is a partial order on permutations, we use the term \emph{well-partially-ordered (wpo)} to describe this property instead. Of course, permutation classes cannot contain infinite strictly decreasing sequences, so wpo is synonymous with the absence of infinite antichains in this context. Well-quasi/partial-orders have been studied quite extensively in a variety of contexts. Cherlin~\cite{cherlin:forbidden-subst:} surveys these investigations from a particularly general perspective.

The infinite antichain we constructed above from split-end paths lies in the class $\Av(321)$. Atkinson, Murphy, and Ru\v{s}kuc~\cite{atkinson:partially-well-:} showed that the only wpo principal classes are $\Av(1)$, $\Av(21)$, $\Av(231)$, and their symmetries. This supports our statement in Section~\ref{subsec-len3} that despite their Wilf-equivalence, the classes $\Av(231)$ and $\Av(321)$ are very different.

The following property of wpo classes can be tremendously useful.

\index{descending chain condition}

\begin{proposition}
\label{prop-wpo-subclasses-dcc}
The subclasses of a wpo permutation class satisfy the \emph{descending chain condition}, i.e., if $\C$ is a wpo class, there does not exist an infinite sequence $\C=\C^0\supsetneq \C^1\supsetneq \C^2\supsetneq\cdots$ of permutation classes.
\end{proposition}
\begin{proof}
Suppose to the contrary that the wpo class $\C$ were to contain an infinite strictly decreasing sequence of subclasses $\C=\C^0\supsetneq \C^1\supsetneq \C^2\supsetneq\cdots$. For each $i\ge 1$, choose $\beta_i\in\C^{i-1}\setminus\C^i$.  The set of minimal elements of $\{\beta_1,\beta_2\ldots\}$ is an antichain and therefore finite, so there is an integer $m$ such that $\{\beta_1,\beta_2\ldots,\beta_m\}$ contains these minimal elements. In particular, $\beta_{m+1}\ge\beta_i$ for some $1\le i\le m$. However, we chose $\beta_{m+1}\in\C^m\setminus\C^{m+1}$, and because $\beta_{m+1}$ contains $\beta_i$, it does not lie in $\C^i$ and thus cannot lie in $\C^m$, a contradiction.
\end{proof}

In their 1996 paper, Noonan and Zeilberger~\cite{noonan:the-enumeration:} conjectured that every finitely based permutation class has a $D$-finite generating function. Clearly the finite basis hypothesis is necessary---because there are infinite antichains of permutations, there are uncountably many permutation classes with distinct generating functions, but only countably many $D$-finite generating functions with rational coefficients. The status of the Noonan-Zeilberger Conjecture is still unresolved, though Zeilberger no longer believes it to be true (as witnessed by his quote about $\Av(1324)$ presented in Section~\ref{subsec-len4}). Moreover, the work of Conway and Guttman~\cite{conway:on-the-growth-r:} suggests that $\Av(1324)$ has a non-$D$-finite generating function, giving us a concrete potential counterexample.

One might also ask about the other direction. If a class has a ``nice'' generating function, does it have a good deal of structure? Indeed, at the same conference, Zeilberger asked for necessary and sufficient conditions for a class to have a rational/algebraic/$D$-finite/... generating function~\cite{elder:problems-and-co:}. However, the work of Albert, Brignall, and Vatter~\cite{albert:large-infinite-:} shows that this question is almost certainly intractable.

To briefly sketch this argument, let $\C$ be a proper permutation class, so Marcus and Tardos' Theorem~\ref{thm-marcus-tardos} shows that there is some constant $\gamma$ such that $|\C_n|<\gamma^n$ for all $n$. By modifying the construction of the antichain in Figure~\ref{fig-infinite-antichain}, it can be shown that for every finite $\gamma$, there is an infinite antichain $A$ containing at least $\delta^n$ permutations of each sufficiently long length $n$ for some constant $\delta>\gamma$. Moreover, $A$ can be constructed in such a way that its downward closure, $\Sub(A)$, has a rational generating function. Therefore if we start with the class $\C\cup \Sub(A)$ and remove $|\C_n|$ elements of length $n$ from $A$ for every sufficiently long length $n$, we obtain a permutation class containing $\C$ that has (up to the addition of a polynomial) the same rational generating function as $\Sub(A)$, proving the following.

\begin{theorem}[Albert, Brignall, and Vatter~\cite{albert:large-infinite-:}]
\label{thm-contained-rational}
Every permutation class except for the class of all permutations is contained in a class with a rational generating function.
\end{theorem}

\index{perfect graphs}

In light of Theorem~\ref{thm-contained-rational}, it is clearly hopeless to attempt to establish a structural characterization of classes with rational (or algebraic, or $D$-finite, ...) generating functions. A similar issue arises in the definition of perfect graphs. In that context, we know that the chromatic number of a graph, $\chi(G)$, is at least its clique number, $\omega(G)$ and we would like to ask which graphs achieve equality, but we're faced with the problem that given any graph, its union with a sufficiently large complete graph will satisfy $\chi=\omega$. Thus we say that a graph is \emph{perfect} if $\chi=\omega$ for the graph and \emph{all of its induced subgraphs}. In the permutation class context, we parallel this by saying that a permutation class is \emph{strongly rational (respectively, strongly algebraic)} if it and all of its subclasses have rational (respectively, algebraic) generating functions. Our counting argument from before now yields the following implication.

\index{strongly rational permutation class}
\index{strongly algebraic permutation class}

\begin{proposition}
\label{prop-strong-wpo}
Every strongly algebraic permutation class is wpo.
\end{proposition}

I have conjectured that this is actually the characterization of strongly algebraic permutation classes, though this conjecture has never before appeared in print.

\begin{conjecture}
\label{conj-wpo-algebraic}
A permutation class is strongly algebraic if and only if it is wpo.
\end{conjecture}

As of yet, very little has been established about strongly algebraic classes, but there are some results on strongly rational classes. Using the substitution decomposition (the topic of Section~\ref{subsec-subst-decomp}), Albert and Atkinson~\cite{albert:simple-permutat:} proved that every proper subclass of $\Av(231)$ has a rational generating function. More generally, Albert, Atkinson, and Vatter~\cite{albert:subclasses-of-t:} showed that in a strongly rational class, the sum indecomposable permutations also have a rational generating function. Other properties of strongly rational classes are discussed in Section~\ref{subsec-geom-grid}.

In the rest of this section we present several more notions of structure and study their interactions. Most of the tools developed here are applied in the final section, where we ask about the set of all growth rates of permutation classes.

\subsection{Merging and splitting}
\label{subsec-merge}

\index{merge of two permutation classes}

Here we consider a very coarse notion of structure. Given any two permutation classes $\C$ and $\D$, their \emph{merge}, $\C\odot\D$, consists of those permutations whose entries can be partitioned into two subsequences, one order isomorphic to a permutation in $\C$ and the order isomorphic to a permutation in $\D$ (usually thought of as coloring the entries of the permutation red or blue). For example, Proposition~\ref{prop-321-merge} shows that
$$
\Av(321)=\Av(21)\odot\Av(21),
$$
and this generalizes to any class of the form $\Av(k\cdots 21)$. What if we change one of the $\Av(21)$ classes to $\Av(12)$? Then we obtain the class of permutations that can be expressed as the union of an increasing subsequence and a decreasing subsequence. This class was introduced by Stankova~\cite{stankova:forbidden-subse:} who named it the class of \emph{skew-merged} permutations and computed its basis:
$$
\Av(21)\odot\Av(12)=\Av(2143, 3412).
$$
Atkinson~\cite{atkinson:permutations-wh:} was the first to compute the (algebraic) generating function of this class.

\index{skew-merged permutations}

It should be noted that the merge operation preserves very few properties. The example of $\Av(321)$ shows that it does not preserve wpo or rational generating functions, while the example of $\Av(4321)$ shows that it does not preserve algebraic generating functions. Moreover, it is relatively easy to construct examples of finitely based classes whose merge is not finitely based, so merge does not preserve finite bases (though K\'ezdy, Snevily, and Wang~\cite{kezdy:partitioning-pe:} proved that the merge of $\Av(12\cdots k)$ with $\Av(\ell\cdots 21)$ always has a finite basis). However, we can get a bound on the growth rates of merges by the following result (which seems to have first appeared, though implicitly, in Albert~\cite{albert:on-the-length-o:}).

\begin{proposition}
\label{prop-merge-gr}
For any two permutation classes $\C$ and $\D$,
$$
\ugr(\C\odot\D)\le\left(\sqrt{\ugr(\C)}+\sqrt{\ugr(\D)}\right)^2.
$$
\end{proposition}
\begin{proof}
Given a permutation of length $k$ from $\C$ and another permutation of length $n-k$ from $\D$, there are ${n\choose k}^2$ ways to merge them to form a permutation of length $n$: choose $k$ positions to be occupied by the permutation from $\C$, then choose the $k$ values for this subsequence, and then the permutation is determined. Therefore
\begin{equation}
\label{eqn-merge-ugr}
|(\C\odot\D)_n|
\le
\sum_{k=0}^n {n\choose k}^2|\C_k||\D_{n-k}|
\le
\left(\sum_{k=0}^n {n\choose k}\sqrt{|\C_k||\D_{n-k}|}\right)^2.
\end{equation}
To avoid introducing epsilon, we sketch the proof from this point, though it is not difficult to make it formal. Suppose that $|\C_n|\approx\gamma^n$ and $|\D_n|\approx\delta^n$. Then \eqref{eqn-merge-ugr} becomes
$$
|(\C\odot\D)_n|
\lesssim
\left(\sum_{k=0}^n {n\choose k}\sqrt{\gamma^k\delta^{n-k}}\right)^2
=
\left(\sqrt{\gamma}+\sqrt{\delta}\right)^{2n}.
$$
Taking $n$th roots gives the desired inequality.
\end{proof}

\index{splittable permutation class}

Which classes can we obtain via merges? Let us say that the class $\C$ is \emph{splittable} if $\C$ is contained in $\D\odot\E$ for proper subclasses $\D,\E\subsetneq\C$. Thus $\Av(k\dots 21)$ is splittable for every $k$, as is the class of skew-merged permutations. But clearly the class $\Av(21)$ of increasing permutations is not splittable, as is (with a bit more thought) the class of layered permutations. In its fullest generality, this remains an open (and seemingly quite difficult) question.

\begin{question}
\label{ques-splittable}
Which permutation classes are splittable?
\end{question}

Question~\ref{ques-splittable} remains open even for principal classes. The case of $\Av(\beta)$ where $\beta$ is layered has been considered several times in the literature. B\'ona~\cite{bona:new-records-in-:} presented the first splittability result for such classes, though not in this language. His result was then generalized by Claesson, Jel{\'{\i}}nek, and Steingr{\'{\i}}msson~\cite{claesson:upper-bounds-fo:}, which was in turn generalized by Jel{\'{\i}}nek and Valtr to the following.


\begin{proposition}[Jel{\'{\i}}nek and Valtr~\cite{jelinek:splittings-and-:}]
\label{prop-split-1}
For all nonempty permutations $\alpha$, $\beta$, and $\gamma$, we have
$$
\Av(\alpha\oplus\beta\oplus\gamma)
\subseteq
\Av(\alpha\oplus\beta) \odot \Av(\beta\oplus\gamma).
$$
In particular, every principal class whose basis element is the sum of three (or more) nonempty permutations is splittable.
\end{proposition}
\begin{proof}
Take $\pi\in\Av(\alpha\oplus\beta\oplus\gamma)$. We seek to color the entries of $\pi$ red and blue so that the red entries avoid $\alpha\oplus\beta$ and the blue entries avoid $\beta\oplus\gamma$. We proceed from left to right, coloring the entry $\pi(i)$ red unless
\begin{enumerate}
\item[(B1)] doing so would create a red copy of $\alpha\oplus\beta$, or
\item[(B2)] there is already a blue entry below and to the left of $\pi(i)$.
\end{enumerate}
If either of these conditions hold, we color $\pi(i)$ blue.

The red entries of $\pi$ avoid $\alpha\oplus\beta$ by definition, so it suffices to show that the blue entries avoid $\beta\oplus\gamma$. Suppose otherwise, and let $\pi(k)$ denote the bottommost entry of a blue copy of $\beta\oplus\gamma$ in $\pi$. If $\pi(k)$ were colored blue because of rule (B1), then $\pi$ would contain a copy of $\alpha\oplus\beta$ ending at $\pi(k)$, and thus this together with the blue copy of $\gamma$ shows that $\pi$ would contain $\alpha\oplus\beta\oplus\gamma$. As this leads to a contradiction, $\pi(k)$ must have been colored blue because of rule (B2).

\begin{figure}
\begin{footnotesize}
\begin{center}
	\begin{tikzpicture}[scale=0.35, baseline=(current bounding box.center)]
		\draw (0,1) rectangle (3,3);
		\node at (1.5,2) {red $\alpha$};
		\draw (3,3) rectangle (7,6);
		\node at (5,4) {mixed $\beta$};
		\draw (2,5) rectangle (11,10);
		\node at (6.5,7.5) {blue $\beta$};
		\draw (11,10) rectangle (14,12);
		\node at (12.5,11) {blue $\gamma$};
		\draw[fill=black] (2,7.5) circle (5.71428571429pt) node[left] {blue $\pi(i)$};
		\draw[fill=black] (7,3) circle (5.71428571429pt) node[below] {blue $\pi(j)$};
		\draw[fill=black] (9.5,5) circle (5.71428571429pt) node[below] {blue $\pi(k)$};
	\end{tikzpicture}
\end{center}
\end{footnotesize}
\caption{The final contradiction in the proof of Proposition~\ref{prop-split-1}.}
\label{fig-jel-valtr}
\end{figure}
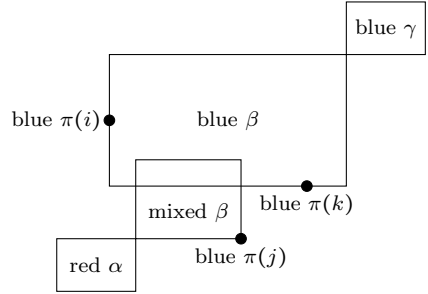

Let $\pi(j)$ denote the leftmost blue entry below and to the left of $\pi(k)$. By this choice of $\pi(j)$, we know that it was colored blue because of rule (B1), and thus it forms the rightmost entry of an otherwise-red copy of $\alpha\oplus\beta$ in $\pi$. We now seek to establish the contradiction shown in Figure~\ref{fig-jel-valtr}. Choose $\pi(i)$ to be the leftmost entry in the blue copy of $\beta\oplus\gamma$ which $\pi(k)$ lies in. Clearly if the red copy of $\alpha$ were to lie completely to the left of $\pi(i)$ then $\pi$ would contain $\alpha\oplus\beta\oplus\gamma$, so some entries of this copy of $\alpha$ must lie to the right of $\pi(i)$. However, this means that all of the entries of the mixed copy of $\beta$ must lie to the right of $\pi(i)$, and thus because $\pi(i)$ is blue, none of the red entries of this copy of $\beta$ may lie above $\pi(i)$. Finally, the blue entry $\pi(j)$ must also lie below $\pi(i)$ because it lies below $\pi(k)$. However, this implies that the mixed copy of $\beta$ must lie entirely below and to the left of the blue copy of $\gamma$, contradicting our assumption that $\pi$ avoids $\alpha\oplus\beta\oplus\gamma$.
\end{proof}

While the work of Jel{\'{\i}}nek and Valtr~\cite{jelinek:splittings-and-:} focuses on the abstract notions of splittability, the earlier work of Claesson, Jel{\'{\i}}nek, and Steingr{\'{\i}}msson~\cite{claesson:upper-bounds-fo:} was concerned with the applications of Propositions~\ref{prop-merge-gr} and (their version of) \ref{prop-split-1} to growth rates of principal classes with layered basis elements. In particular, we see that $\Av(1324)\subseteq\Av(132)\odot\Av(213)$, and thus the growth rate of $\Av(1324)$ is at most $16$. B\'ona~\cite{bona:a-new-record-fo:} has since refined this approach to give an upper bound of $13.74$.


Our next result was promised in the beginning of Section~\ref{sec-principal}.

\begin{theorem}[Claesson, Jel{\'{\i}}nek, and Steingr{\'{\i}}msson~\cite{claesson:upper-bounds-fo:}]
\label{thm-layered-gr}
For every layered permutation $\beta$ of length $k$, the growth rate of $\Av(\beta)$ is less than $4k^2$.
\end{theorem}
\begin{proof}
Given a positive integer $\ell$, we denote by $\delta_\ell$ the permutation $\ell\cdots 21$. We prove the stronger inequality that for a sequence $\ell_1,\dots,\ell_m$ of positive integers summing to $k$,
$$
\gr(\Av(\delta_{\ell_1}\oplus\cdots\oplus\delta_{\ell_m}))
\le
(2k-\ell_1-\ell_m-m+1)^2
<
4k^2.
$$
The proof is by induction on $m$. The $m=1$, $2$ cases follow from Corollary~\ref{cor-two-layer}, so we may assume that $m\ge 3$. By Propositions~\ref{prop-merge-gr} and \ref{prop-split-1},
\begin{eqnarray*}
\sqrt{\gr(\Av(\delta_{\ell_1}\oplus\cdots\oplus\delta_{\ell_m}))}
&<&
\sqrt{\gr(\Av(\delta_{\ell_1}\oplus\delta_{\ell_2}))}+\sqrt{\gr(\Av(\delta_{\ell_2}\oplus\cdots\oplus\delta_{\ell_m}))},\\
&\le&
(\ell_1+\ell_2-1)+(2(k-\ell_1)-\ell_2-\ell_m-(m-1)+1),\\
&=&
2k-\ell_1-\ell_m-m-1,
\end{eqnarray*}
as desired.
\end{proof}

We conclude this subsection by returning to a (very) special case of the splittability question. What about principal classes for which the basis element is the sum of just two permutations, i.e., those of the form $\Av(\alpha\oplus\beta)$? Here we can apply Proposition~\ref{prop-split-1} to see that
$$
\Av(\alpha\oplus\beta)
\subseteq
\Av(\alpha\oplus 1\oplus\beta)
\subseteq
\Av(\alpha\oplus 1) \odot \Av(1\oplus\beta).
$$
Thus so long as neither $\alpha\oplus 1$ nor $1\oplus\beta$ is equal to $\alpha\oplus\beta$, such classes are splittable. By symmetry, this leaves only those classes of the form $\Av(1\oplus\beta)$. We have already remarked that $\Av(12)$ is not splittable. Jel{\'{\i}}nek and Valtr~\cite{jelinek:splittings-and-:} showed that $\Av(132)$ is also not splittable. However, via an intricate argument, they were able to prove that every class of the form $\Av(1\oplus\beta)$ with $|\beta|\ge 3$ is splittable, which is the final piece needed to obtain the following result.

\begin{theorem}[Jel{\'{\i}}nek and Valtr~\cite{jelinek:splittings-and-:}]
\label{thm-split-2}
For all sum (or skew) decomposable permutations $\beta$ of length at least four, the class $\Av(\beta)$ is splittable.
\end{theorem}

\subsection{The substitution decomposition}
\label{subsec-subst-decomp}

\index{substitution decomposition}

While the merge construction can provide somewhat reasonable upper bounds on growth rates on splittable classes, it sheds little light on the exact enumeration problem. In this subsection, we investigate the substitution decomposition, which has proved very useful for computing generating functions of classes, especially when combined with the techniques of Section~\ref{subsec-geom-grid}. It should be noted that, as with most of the structural notions discussed in this section, the substitution decomposition has been studied for a wide variety of combinatorial objects. For a slightly outdated survey, we refer to M\"ohring and Radermacher~\cite{mohring:substitution-de:}. This concept dates back to a 1953 talk of Fra{\"{\i}}ss{\'e}~\cite{fraisse:on-a-decomposit:}, although its first significant application was in Gallai's 1967 paper~\cite{gallai:transitiv-orien:} (see \cite{gallai:a-translation-o:} for a translation).



\index{interval}
\index{inflation}
\index{simple permutation}

An \emph{interval} in the permutation $\pi$ is a set of contiguous indices $I=[a,b]$ such that the set of values $\pi(I)=\{\pi(i) : i\in I\}$ is also contiguous (four intervals are indicated in the permutation on the left of Figure~\ref{fig-subst-tree}).  Given a permutation $\sigma$ of length $m$ and nonempty permutations $\alpha_1,\dots,\alpha_m$, the \emph{inflation} of $\sigma$ by $\alpha_1,\dots,\alpha_m$,  denoted $\sigma[\alpha_1,\dots,\alpha_m]$, is the permutation of length $|\alpha_1|+\cdots+|\alpha_m|$ obtained by replacing each entry $\sigma(i)$ by an interval that is order isomorphic to $\alpha_i$ in such a way that the intervals themselves are order isomorphic to $\sigma$.  For example, the permutation shown in Figure~\ref{fig-subst-tree} is
\[
2413[1,132,321,12]=4\ 798\ 321\ 56. 
\]
Every permutation of length $n\ge 1$ has \emph{trivial} intervals of lengths $0$, $1$, and $n$; all other intervals are termed \emph{proper}. We further say that the empty permutation and the permutation $1$ are \emph{trivial}. A nontrivial permutation is \emph{simple} if it has no proper intervals.  The shortest simple permutations are thus $12$ and $21$, there are no simple permutations of length three, and the simple permutations of length four are $2413$ and $3142$. Simple permutations and inflations are linked by the following result. Its proof follows in a straight-forward manner once one establishes the fact that the intersection of two intervals is itself an interval.

\begin{proposition}[Albert and Atkinson~\cite{albert:simple-permutat:}]
\label{prop-simple-decomp-unique}
Every nontrivial permutation $\pi$ is an inflation of a unique simple permutation $\sigma$.  Moreover, if $\pi=\sigma[\alpha_1,\dots,\alpha_m]$ for a simple permutation $\sigma$ of length $m\ge 4$, then each $\alpha_i$ is unique. If $\pi$ is sum decomposable, then there is a unique sequence of sum indecomposable permutations $\alpha_1,\dots,\alpha_m$ such that $\pi=\alpha_1\oplus\cdots\oplus\alpha_m$. The same holds, mutatis mutandis, with sum replaced by skew sum.
\end{proposition}
%

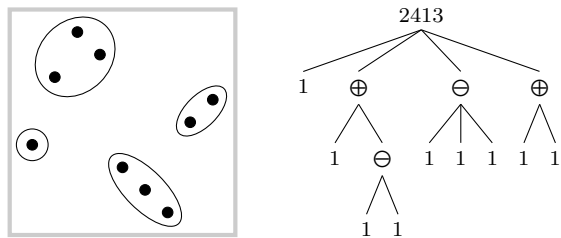
\begin{figure}
\begin{footnotesize}
\begin{center}
$$
\begin{array}{ccc}
	\begin{tikzpicture}[scale=0.3, baseline=(current bounding box.center)]
		\draw [lightgray, ultra thick, line cap=round] (0,0) rectangle (10,10);
		\draw[fill=black] (1,4) circle (6.6666666pt);
		\draw[fill=black] (2,7) circle (6.6666666pt);
		\draw[fill=black] (3,9) circle (6.6666666pt);
		\draw[fill=black] (4,8) circle (6.6666666pt);
		\draw[fill=black] (5,3) circle (6.6666666pt);
		\draw[fill=black] (6,2) circle (6.6666666pt);
		\draw[fill=black] (7,1) circle (6.6666666pt);
		\draw[fill=black] (8,5) circle (6.6666666pt);
		\draw[fill=black] (9,6) circle (6.6666666pt);
		\draw (1,4) circle (20pt);
		\draw[rotate around={-45:(2.9,7.9)}] (2.9,7.9) ellipse (45pt and 55pt);
		\draw[rotate around={45:(6,2)}] (6,2) ellipse (25pt and 60pt);
		\draw[rotate around={-45:(8.5,5.5)}] (8.5,5.5) ellipse (20pt and 40pt);
	\end{tikzpicture}
&&
	\begin{tikzpicture}[baseline=(current bounding box.center)]
	\tikzset{level distance=27pt}
	\Tree	[.$2413$	$1$
				[.$\bigoplus$ $1$ [.$\bigominus$ $1$ $1$ ] ]
				[.$\bigominus$ $1$ $1$ $1$ ]
				[.$\bigoplus$ $1$ $1$ ]
		]
	\end{tikzpicture}
\end{array}
$$
\end{center}
\end{footnotesize}
\caption{The plot of the permutation $479832156$ and its substitution decomposition tree.}
\label{fig-subst-tree}
\end{figure}

\index{substitution decomposition tree}
\index{substitution depth}

By recursively decomposing the permutation $\pi$ and its intervals as suggested by Proposition~\ref{prop-simple-decomp-unique}, we obtain a rooted tree called the \emph{substitution decomposition tree} of $\pi$ (an example is shown on the right of Figure~\ref{fig-subst-tree}). The \emph{substitution depth} of $\pi$ is the height of its substitution decomposition tree, so for example, the substitution depth of the permutation from Figure~\ref{fig-subst-tree} is $3$, while the substitution depth of every simple or monotone permutation is $1$, and the substitution depth of every nonmonotone layered permutation is at most $2$.

The main result of Albert and Atkinson's formative paper on the substitution decomposition is the following.

\begin{theorem}[Albert and Atkinson~\cite{albert:simple-permutat:}]
\label{thm-fin-simples-alg}
Every permutation class with only finitely many simple permutations is strongly algebraic (and thus in particular, wpo). Moreover, every permutation class with only finitely many simple permutations and bounded substitution depth is strongly rational.
\end{theorem}

For example, Theorem~\ref{thm-fin-simples-alg} shows that the class of separable permutations is strongly algebraic, because the only simple permutations in this class are $\{1,12,21\}$. However, there are (relatively speaking) few permutation classes with only finitely many simple permutations, so it may appear that the substitution decomposition is not as applicable as might be hoped. We remedy this when we present a generalization in Section~\ref{subsec-geom-grid} (Theorem~\ref{thm-ggc-inflations}).

A different generalization is due to Brignall, Huczynska, and Vatter~\cite{brignall:simple-permutat:}. They gave a proof of Theorem~\ref{thm-fin-simples-alg} using ``query-complete sets of properties'' showing that the conclusion holds not just for the permutation class itself but for a great many subsets of it. For example, if a permutation class contains only finitely many simple permutations then the generating function for the alternating (up-down) permutations, or the even permutations, or the involutions in the class is also algebraic.

\index{substitution closure}

Given two classes $\D$ and $\U$, the \emph{inflation} of $\D$ by $\U$ is defined as
\[
\D[\U]=\{\sigma[\alpha_1,\dots,\alpha_m]\st\mbox{$\sigma\in\D_m$ and $\alpha_1,\dots,\alpha_m\in\U$}\}.
\]
The class $\C$ is said to be \emph{substitution closed} if $\C[\C]\subseteq\C$. The \emph{substitution closure} $\langle\C\rangle$ of a class $\C$ can be defined in any number of ways, for example as the smallest substitution closed class containing $\C$, or as the largest permutation class with the same set of simple permutations as $\C$, or as
$$
\langle\C\rangle=\C\cup\C[\C]\cup\C[\C[\C]]\cup\cdots.
$$
We leave it to the reader to convince themselves that these are all equivalent definitions and then to establish the following.

\begin{proposition}
\label{prop-subst-closed}
A class is substitution closed if and only if all of its basis elements are simple permutations.
\end{proposition}

Atkinson, Ru\v{s}kuc, and Smith~\cite{atkinson:substitution-cl:} investigated bases of substitution closures of principal classes. In a great many cases these bases are infinite, but they were able to compute several bases in the finite cases. For example, they showed that
$$
\langle\Av(321)\rangle
=
\Av(25314, 35142, 41352, 42513, 362514, 531642)
$$
by bounding the length of potential basis elements of this class and then conducting an exhaustive search by computer.

\index{deflatable permutation class}

If $\C\subseteq\langle\D\rangle$ for a proper subclass $\D\subsetneq\C$ then we say that $\C$ is \emph{deflatable}. We then have the following analogue of the splittability question.

\begin{question}
\label{ques-deflatable}
Which permutation classes are deflatable?
\end{question}

An equivalent definition is that the class $\C$ is deflatable if and only if one can find a permutation $\pi\in\C$ that is not contained in any simple permutation of $\C$. For example, the class of separable permutations is clearly deflatable, as it is equal to $\langle\{1,12,21\}\rangle$. It follows that $\Av(231)$, as a subclass of the separable permutations, is also deflatable. Neither of these classes is splittable, so there are classes that can be decomposed via inflations that can't be decomposed via the merge operation. In the other direction, $\Av(321)$ is trivially splittable, but it can be shown that it is not deflatable.

While using the substitution decomposition to enumerate deflatable classes has become a common technique (several applications of this approach are cited in Section~\ref{subsec-geom-grid}), Question~\ref{ques-deflatable} has only recently attracted attention. Interestingly, the most comprehensive result on deflatability to date, quoted below, shows that most of the classes known to be splittable by Theorem~\ref{thm-split-2} are \emph{not} deflatable. Thus deflatability and splittability are in some sense good complements to each other.

\begin{theorem}[Albert, Atkinson, Homberger, and Pantone~\cite{albert:deflatability-o:}]
\label{thm-deflatable}
For all choices of nonempty permutations $\alpha$, $\beta$, and $\gamma$, the class $\Av(\alpha\oplus\beta\oplus\gamma)$ is not deflatable. Moreover, for $|\alpha|,|\beta|\ge 2$, the class $\Av(\alpha\oplus\beta)$ is not deflatable.
\end{theorem}

They also showed that all classes of the form $\Av(1\oplus\beta)$ with $|\beta|\le 4$ are not deflatable, but the class $\Av(1\oplus 23541)$ is deflatable.

Atkinson posed a much different conjecture about the simple permutations in a class at the conference \emph{Permutation Patterns 2007}. Let $\Simples_n(\C)$ denote the set of all simple permutations in $\C$ of length $n$.

\begin{conjecture}[Atkinson]
\label{conj-mike-dense}
The ratio $|\Simples_n(\C)|/|\C_n|$ tends to $0$ as $n\rightarrow\infty$ for every proper permutation class $\C$.
\end{conjecture}

What happens if we let $\C$ be the (nonproper) class of all permutations in Atkinson's Conjecture? It is relatively easy to show that the probability that a permutation of length $n$ contains an interval of size $3$ or greater tends to $0$ as $n\rightarrow\infty$, so (asymptotically) the only obstructions to simplicity are intervals of size $2$. Let $X$ denote the number of intervals of size $2$ in a permutation of length $n$ chosen uniformly at random. It is routine to compute that $\mathbb{E}[X]=2$. Moreover, in the 1940s, Kaplansky~\cite{kaplansky:the-asymptotic-:} and Wolfowitz~\cite{wolfowitz:note-on-runs-of:} showed (independently) that $X$ is asymptotically Poisson distributed, and thus $\Pr[X=0]=1/e^2$ (a more modern proof using what is known as the Chen-Stein Method appears in Corteel, Louchard, and Pemantle~\cite{corteel:common-interval:}, and it is not difficult to establish this result via Brun's Sieve if one is so inclined).

Therefore the number of simple permutations of length $n$ is asymptotic to $n!/e^2$. Albert, Atkinson, and Klazar~\cite{albert:the-enumeration:} gave refined asymptotics for this quantity, showing that it is
$$
\frac{n!}{e^2}\left(1-\frac{4}{n}+\frac{2}{n(n-1)}+O\left(\frac{1}{n^3}\right)\right).
$$
They also established that the sequence counting simple permutations of length $n$ is not $D$-finite.

Our next result can be used to verify that a given permutation class contains only finitely many simple permutations.

\begin{theorem}[Schmerl and Trotter~\cite{schmerl:critically-inde:}]\label{thm-schmerl-trotter}
Every simple permutation of length $n$ contains a simple permutation of length $n-1$ or $n-2$.
\end{theorem}

It should be noted that Schmerl and Trotter proved their theorem in the more general context of binary, irreflexive relational structures, of which permutations are a special case. A streamlined proof of Schmerl and Trotter's theorem for the special case of permutations is given in Brignall and Vatter~\cite{brignall:a-simple-proof-:}.

\index{alternation}
\index{parallel alternation}
\index{wedge alternation}

The simple permutations that do not contain a simple permutation with one fewer entry are quite rare. We say that an \emph{alternation} is a permutation whose plot can be divided into two halves, by a single horizontal or vertical line, so that for every pair of entries from the same part there is an entry from the other part that \emph{separates} them, i.e., there is an entry from the other part that lies either horizontally or vertically between them.  A \emph{parallel alternation} is an alternation in which these two sets of entries form monotone subsequences, either both increasing or both decreasing (the permutation on the left of Figure~\ref{fig-simples-5ex} is a parallel alternation). The statement of Schmerl and Trotter's Theorem can be refined to say that every simple permutation that is not a parallel alternation contains a simple permutation with one fewer entry.

\begin{figure}
\begin{center}
\begin{footnotesize}
	\begin{tikzpicture}[scale=0.2222222222222, baseline=(current bounding box.center)]
		\draw [lightgray, ultra thick, line cap=round] (0,0) rectangle (9,9);
		\draw [lightgray, ultra thick, line cap=round] (0,4.5)--(9,4.5);
		\draw[fill=black] (1,5) circle (9pt);
		\draw[fill=black] (2,1) circle (9pt);
		\draw[fill=black] (3,6) circle (9pt);
		\draw[fill=black] (4,2) circle (9pt);
		\draw[fill=black] (5,7) circle (9pt);
		\draw[fill=black] (6,3) circle (9pt);
		\draw[fill=black] (7,8) circle (9pt);
		\draw[fill=black] (8,4) circle (9pt);
	\end{tikzpicture}
\quad\quad
	\begin{tikzpicture}[scale=0.2222222222222, baseline=(current bounding box.center)]
		\draw [lightgray, ultra thick, line cap=round] (0,0) rectangle (9,9);
		\draw [line cap=round] (5,1)--(5,8);
		\draw[fill=black] (1,2) circle (9pt);
		\draw[fill=black] (2,4) circle (9pt);
		\draw[fill=black] (3,6) circle (9pt);
		\draw[fill=black] (4,8) circle (9pt);
		\draw[fill=black] (5,1) circle (9pt);
		\draw[fill=black] (6,7) circle (9pt);
		\draw[fill=black] (7,5) circle (9pt);
		\draw[fill=black] (8,3) circle (9pt);
	\end{tikzpicture}
\quad\quad
	\begin{tikzpicture}[scale=0.2222222222222, baseline=(current bounding box.center)]
		\draw [lightgray, ultra thick, line cap=round] (0,0) rectangle (9,9);
		\draw [line cap=round] (2,8)--(2,1);
		\draw[fill=black] (1,4) circle (9pt);
		\draw[fill=black] (2,8) circle (9pt);
		\draw[fill=black] (3,3) circle (9pt);
		\draw[fill=black] (4,5) circle (9pt);
		\draw[fill=black] (5,2) circle (9pt);
		\draw[fill=black] (6,6) circle (9pt);
		\draw[fill=black] (7,1) circle (9pt);
		\draw[fill=black] (8,7) circle (9pt);
	\end{tikzpicture}
\quad\quad
	\begin{tikzpicture}[scale=0.2222222222222, baseline=(current bounding box.center)]
		\draw [lightgray, ultra thick, line cap=round] (0,0) rectangle (9,9);
		\draw [line cap=round] (1,2)--(6,2);
		\draw[fill=black] (1,2) circle (9pt) node[above] {$p_5$};
		\draw [line cap=round] (2,7)--(2,1);
		\draw[fill=black] (2,7) circle (9pt) node[above] {$p_6$};
		\draw[fill=black] (3,5) circle (9pt) node[right] {$p_1$};
		\draw[fill=black] (4,3) circle (9pt) node[left] {$p_2$};
		\draw [line cap=round] (5,1)--(5,5);
		\draw[fill=black] (5,1) circle (9pt) node[right] {$p_4$};
		\draw [line cap=round] (6,4)--(3,4);
		\draw[fill=black] (6,4) circle (9pt) node[right] {$p_3$};
		\draw [line cap=round] (7,8)--(7,5);
		\draw[fill=black] (7,8) circle (9pt) node[right] {$p_8$};
		\draw [line cap=round] (8,6)--(1,6);
		\draw[fill=black] (8,6) circle (9pt) node[below] {$p_7$};
	\end{tikzpicture}
\end{footnotesize}
\end{center}
\caption{From left to right, a parallel alternation, two wedge simple permutations, and a proper pin sequence.}
\label{fig-simples-5ex}
\end{figure}
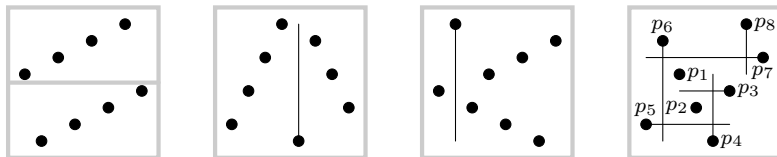

A \emph{wedge alternation} is an alternation in which the two halves of entries form monotone subsequences, one increasing and one decreasing. Wedge alternations can be made simple by the addition of a single entry in one of two positions, as shown in the second and third permutations of Figure~\ref{fig-simples-5ex}. We call the resulting permutations \emph{wedge simple permutations}. Brignall, Huczynska, and Vatter~\cite{brignall:decomposing-sim:} proved a Ramsey-like result for simple permutations, which states that every ``long enough'' simple permutation contains a large simple parallel alternation, a large wedge simple permutation, or a large member of a third family of simple permutations, called \emph{proper pin sequences}. This latter family (a member of which is shown on the far right of Figure~\ref{fig-simples-5ex}) is defined inductively. 

\index{proper pin sequence}
\index{axes-parallel rectangle}

An \emph{axes-parallel rectangle} is any rectangle of the form $X\times Y$ for intervals $X$ and $Y$. The \emph{rectangular hull} of a set of points in the plane is defined as the smallest axes-parallel rectangle containing them. Given independent points $\{p_1,\dots,p_i\}$ in the plane, a \emph{proper pin} for these points is a point $p$ that lies outside their rectangular hull and separates $p_i$ from $\{p_1,\dots,p_{i-1}\}$. A proper pin sequence is then constructed by starting with independent points $p_1$ and $p_2$, choosing $p_3$ to be a proper pin for $\{p_1,p_2\}$, then choosing $p_4$ to be a proper pin for $\{p_1,p_2,p_3\}$, and so on. It can be shown that every proper pin sequence is either simple or can be made simple by the removal of a single point. We can now state the theorem.

\begin{theorem}[Brignall, Huczynska, and Vatter~\cite{brignall:decomposing-sim:}]
\label{thm-bhv-ramsey}
There is a function $f(k)$ such that every simple permutation of length at least $f(k)$ contains a simple permutation of length at least $k$ that is either a parallel alternation, a wedge simple permutation, or a proper pin sequence.
\end{theorem}

By considering each of the three types of simple permutations guaranteed by Theorem~\ref{thm-bhv-ramsey}, this result implies that there is a function $g(k)$ such that every simple permutation of length at least $g(k)$ contains two simple permutations of length $k$ that are either disjoint or nearly so (they might share a single entry).

This corollary puts the results of B\'ona~\cite{bona:the-number-of-p:,bona:permutations-wi:} under the substitution decomposition umbrella. He showed that for any value of $r$, the number of permutations with at most $r$ copies of $132$ has an algebraic generating function. (Mansour and Vainshtein~\cite{mansour:counting-occurr:} later described a way to compute these generating functions.) Clearly the class of $132$-avoiding permutations contains only finitely many simple permutations (it is a subclass of the separable permutations). To put this another way, all simple permutations of length at least $4$ contain $132$. Therefore, by the above corollary to Theorem~\ref{thm-bhv-ramsey}, all simple permutations of length at least $g(4)$ contain $2$ copies of $132$. More generally, the class of permutations with at most $r$ copies of $132$ does not contain any simple permutations of length $g^r(4)$ or longer, and thus has an algebraic generating function by Theorem~\ref{thm-fin-simples-alg}. (A simpler derivation of this implication of Theorem~\ref{thm-bhv-ramsey} is given in Vatter~\cite{vatter:small-configura:}.)


As shown in Brignall, Ru\v{s}kuc, and Vatter~\cite{brignall:simple-permutat:b}, Theorem~\ref{thm-bhv-ramsey} can be used to devise a procedure to determine whether a class (specified by a finite basis) contains infinitely many simple permutations. Bassino, Bouvel, Pierrot, and Rossin~\cite{bassino:an-algorithm-fo:} have since given a much more practical algorithm for this decision problem.

We conclude this subsection by returning to the splittability question. The following is a special case of the results of Jel{\'{\i}}nek and Valtr~\cite{jelinek:splittings-and-:}.

\begin{proposition}
\label{prop-nonmerge-subst}
Substitution closed classes are not splittable.
\end{proposition}
\begin{proof}
Let $\C$ be a substitution closed class.  If $12\notin\C$ then $\C$ must be the class of all decreasing permutations, which is clearly not splittable, so we may assume that $12\in\C$.  Now take two proper subclasses $\D,\E\subsetneq\C$ and choose $\sigma\in\C\setminus\D$ and $\tau\in\C\setminus\E$.  Because $12\in\C$, the permutation $\rho=\sigma\oplus\tau=12[\sigma,\tau]$ lies in $\C$ but in neither $\D$ nor $\E$.

Now consider $\pi=\rho[\rho,\dots,\rho]$.  By definition, $\pi\in\C$, but we claim that $\pi\notin\D\odot\E$.  Indeed, if we tried to color the entries of $\pi$ red and blue so that the red subpermutation lied in $\D$ and the blue subpermutation lied in $\E$, we would have to use both colors in each interval order isomorphic to $\rho$ (because $\rho\notin\D\cup\E$), but then there would be monochromatic copy of $\rho$ containing one entry per interval.
\end{proof}

By Proposition~\ref{prop-nonmerge-subst} and our previous estimates on the number of simple permutations, we see that if $\beta$ is chosen uniformly at random from all permutations of length $k$, there is (asymptotically) a $1/e^2$ probability that it is simple, and thus that $\Av(\beta)$ is not splittable. As Jel{\'{\i}}nek and Valtr~\cite{jelinek:splittings-and-:} note, it would be interesting to obtain a better understanding of the probability that $\Av(\beta)$ is not splittable.

\subsection{Atomicity}
\label{subsec-atomicity}

\index{atomic permutation class}

To motivate the notion of atomicity, let us continue with the splittability question. If the class $\C$ can be expressed as the union of two of its proper subclasses, then it is trivially contained in the merge of these two classes. We say that $\C$ is \emph{atomic} if it cannot be expressed as the union of two of its proper subclasses. Thus the splittability question is only interesting for atomic classes.

Atomicity was first used in the context of permutations by Atkinson in his 1999 paper ``Restricted permutations''~\cite{atkinson:restricted-perm:} (which anticipated many of the structural notions discussed in this section), where he showed that
$$
\Av(321,2143)
=
\Av(321,2143,2413)
\cup
\Av(321,2143,3142),
$$
and used this to enumerate the former class. In the wider context of relational structures, the notion has a much longer history, dating back to a 1954 article of Fra{\"{\i}}ss{\'e}~\cite{fraisse:sur-lextension-:}. It is not difficult to show that the {\it joint embedding property\/} is a necessary and sufficient condition for the permutation class $\C$ to be atomic; this condition states that for all $\sigma,\tau\in\C$, there is a $\pi\in\C$ containing both $\sigma$ and $\tau$. Fra{\"{\i}}ss{\'e} established another necessary and sufficient condition for atomicity, which we describe only in the permutation context (Fra{\"{\i}}ss{\'e} proved his results in the context of arbitrary relational structures). Given two linearly ordered sets $A$ and $B$ and a bijection $f:A\rightarrow B$, every finite subset $\{a_1<\cdots<a_n\}\subseteq A$ maps to a finite sequence $f(a_1),\dots,f(a_n)\in B$ that is order isomorphic to a unique permutation. We say that this permutation is order isomorphic to $f(\{a_1,\dots,a_n\})$ and call the set of all permutations that are order isomorphic to $f(X)$ for finite subsets $X\subseteq A$ the {\it age of $f$\/}. While this class is typically denoted $\operatorname{Age}(f: A\rightarrow B)$, given our previous definitions we choose instead to refer to it as $\Sub(f: A\rightarrow B)$.

\begin{theorem}[Fra{\"{\i}}ss{\'e}~\cite{fraisse:sur-lextension-:}; see also Hodges~{\cite[Section 7.1]{hodges:model-theory:}}]\label{atomic-tfae}
The following are equivalent for a permutation class $\C$:
\begin{enumerate}
\item[(1)] $\C$ is atomic,
\item[(2)] $\C$ satisfies the joint embedding property, and
\item[(3)] $\C=\Sub(f: A\rightarrow B)$ for a bijection $f$ and countable linear orders $A$ and $B$.
\end{enumerate}
\end{theorem}
\begin{proof}
We first show that (1) and (2) are equivalent. Suppose to the contrary that $\C$ satisfies the joint embedding property but $\C\subseteq\D\cup\E$ for proper subclasses $\D,\E\subsetneq\C$. Clearly we may assume that neither $\D$ nor $\E$ is a subset of the other, so there are permutations $\sigma\in\D\setminus\E$ and $\tau\in\E\setminus\D$. By the joint embedding property there is some $\pi\in\C$ containing both of these permutations, but $\pi$ cannot lie in either $\D$ or $\E$, a contradiction. Next suppose that $\C$ does not satisfy the joint embedding property, so there are permutations $\sigma,\tau\in\C$ such that no permutation in $\C$ contains both. Therefore every permutation in $\C$ avoids either $\sigma$ or $\tau$, so we see that $\C$ is contained in the union of its proper subclasses $\C\cap\Av(\sigma)$ and $\C\cap\Av(\tau)$.

Next we show that (2) and (3) are equivalent. Suppose that $\C=\Sub(f: A\rightarrow B)$ and take $\sigma,\tau\in\C$. Thus $\sigma$ is order isomorphic to $f(A_\sigma)$ for a subset $A_\sigma\subseteq A$ and $\tau$ is order isomorphic to $f(A_\tau)$ for a subset $A_\tau\subseteq A$, so the permutation that is order isomorphic to $f(A_\sigma\cup A_\tau)$ contains both $\sigma$ and $\tau$. Next suppose that $\C$ satisfies the joint embedding property and list the elements of $\C$ as $\sigma_1,\sigma_2,\dots$. Define $\pi_0$ to be the empty permutation, and for $i\ge 1$, choose a permutation in $\C$ which contains both $\sigma_i$ and $\pi_{i-1}$ to be $\pi_i$. Thus every element of $\C$ is contained in some $\pi_i$ (and thus also all $\pi_j$ for $j\ge i$). Clearly we may choose $A_0$, $B_0$, and $f_0$ so that $\pi_0$ is order isomorphic to $f_0(A_0)$. Now for $i\ge 1$, we construct $A_{i}$, $B_{i}$, and $f_{i}$ so that $A_{i}\supseteq A_{i-1}$, $B_{i}\supseteq B_{i-1}$, $f_{i}(A_{i})$ is order isomorphic to $\pi_{i}$, and such that $f_{i}(a)=f_{i-1}(a)$ for all $a\in A_{i-1}$. The desired bijection is then $f=\lim_{i\rightarrow\infty} f_i$, with $A=\bigcup A_i$ and $B=\bigcup B_i$.
\end{proof}

Recall that every principal permutation class is either sum or skew closed (Observation~\ref{obs-principal-sum-closed}), so all principal classes satisfy the joint embedding property and are trivially atomic. Despite this, the notion of atomicity has proved to be quite useful in the study of permutation classes. In particular, we apply the following two propositions in our final subsection.

\begin{proposition}
\label{prop-wpo-atomic-union}
Every wpo permutation class can be expressed as a finite union of atomic classes.
\end{proposition}
\begin{proof}
Consider the binary tree whose root is the wpo class $\C$, all of whose leaves are atomic classes, and in which the children of the non-atomic class $\D$ are two proper subclasses $\D^1,\D^2\subsetneq\D$ such that $\D^1\cup\D^2=\D$.  Because $\C$ is wpo its subclasses satisfy the descending chain condition (Proposition~\ref{prop-wpo-subclasses-dcc}), so this tree contains no infinite paths and thus is finite by K\"onig's Lemma; its leaves are the desired atomic classes.
\end{proof}

Thanks to the following result, the problem of computing growth rates of wpo classes can be reduced to that of computing growth rates of atomic classes. We leave the proof as an easy exercise for the reader.

\begin{proposition}
\label{prop-wpo-atomic-gr}
The upper growth rate of a wpo permutation class is equal to the greatest upper growth rate of its atomic subclasses.
\end{proposition}

The notion of representing permutation classes as ages raises several interesting questions. For instance, define $\mathfrak{T}(A,B)$ as the set of all permutation classes that can be expressed as $\Sub(f: A\rightarrow B)$. Given two linear orders $A$ and $B$, we might ask if we can characterize $\mathfrak{T}(A,B)$. Atkinson, Murphy, and Ru\v{s}kuc~\cite{atkinson:pattern-avoidan:} were the first to investigate this question; they characterized the classes of the form $\Sub(f: \mathbb{N}\rightarrow\mathbb{N})$, which they called \emph{natural classes}. It is worth noting that every principal class is, up to symmetry, a natural class---sum closed classes are natural classes, while skew closed classes can be expressed as $\Sub(f: -\mathbb{N}\rightarrow\mathbb{N})$. Huczynska and Ru\v{s}kuc~\cite{huczynska:pattern-classes:} later studied classes of the form $\Sub(f: A\rightarrow\mathbb{N})$ for arbitrary linear orders $A$, which they called \emph{supernatural classes}. Of the many interesting questions that remain open, we quote their ``contiguity question'':

\begin{question}[Huczynska and Ru\v{s}kuc~\cite{huczynska:pattern-classes:}]
If $\C\in\mathfrak{T}(i\mathbb{N}, \mathbb{N})\cap \mathfrak{T}(k\mathbb{N},\mathbb{N})$, is $\C\in\mathfrak{T}(j\mathbb{N},\mathbb{N})$ for all $i\le j\le k$?
\end{question}

A stronger version of the joint embedding property, called amalgamation, has also been considered. Let $\C$ be a permutation class containing the permutation $\rho$. We say that $\C$ is \emph{${\rho}$-amalgamable} if given two permutations $\sigma,\tau\in\C$, each with a marked copy of $\rho$, we can find a permutation $\pi\in\C$ containing both $\sigma$ and $\tau$ such that the two marked copies of $\rho$ coincide. A class is called \emph{homogeneous} if it is $\rho$-amalgamable for every $\rho\in\C$. Cameron~\cite{cameron:homogeneous-per:} proved that there are precisely five infinite homogeneous permutation classes: $\Av(12)$, $\Av(21)$, the layered permutations, the reversed layered permutations, and the class of all permutations. As Cameron noted, the permutation case turned out to be far simpler than several noteworthy cases considered before, such as the undirected graph case (established in 1980 by Lachlan and Woodrow~\cite{lachlan:countable-ultra:}), the tournament case (due to Lachlan~\cite{lachlan:countable-homog:}), and the directed graph case (Cherlin~\cite{cherlin:the-classificat:}).

We conclude this subsection by relating amalgamation and splittability.

\begin{proposition}[Jel{\'{\i}}nek and Valtr~\cite{jelinek:splittings-and-:}]
Every permutation class that is not $1$-amalgamable is splittable.
\end{proposition}
\begin{proof}
Suppose that the class $\C$ is not $1$-amalgamable, so there are permutations $\sigma,\tau\in\C$ with marked entries $\sigma(i)$ and $\tau(j)$ respectively such that no permutation in $\C$ contains copies of $\sigma$ and $\tau$ in which the entries corresponding to $\sigma(i)$ and $\tau(j)$ coincide. We claim that $\C$ is equal to the merge of its proper subclasses $\C\cap\Av(\sigma)$ and $\C\cap\Av(\tau)$. Let $\pi\in\C$ be arbitrary. Color an entry of $\pi$ red if it can play the role of $\sigma(i)$ in a copy of $\sigma$ in $\pi$, and blue otherwise. None of the red entries can play the role of $\tau(j)$ and none of the blue entries can play the role of $\sigma(i)$, showing that the red entries avoid $\tau$ and the blue entries avoid $\sigma$, proving the result.
\end{proof}

\section{The set of all growth rates}
\label{sec-growth-rates}

The final topic of this survey, the set of growth rates of all permutation classes, relies on the rich toolbox of structural results collected in the previous section as well as the theory of the three different flavors of grid classes introduced in Sections~\ref{subsec-mono-grid}--\ref{subsec-gen-grid}. Like the ideas of the previous section, growth rates have been studied for a wide variety of combinatorial objects, for example, posets, set partitions, graphs, ordered graphs, tournaments, and ordered hypergraphs. For details of the work on growth rates in this more general context, we refer to Bollob\'as' survey for the 2007 \emph{British Combinatorial Conference}~\cite{bollobas:hereditary-and-}.

Characterizing the set of growth rates of permutation classes is easy at the very low end of the spectrum. The number of permutations of length $n$ in a class must be an integer, and if it is ever $0$ then the class is empty for all larger lengths. Thus there are no growth rates properly between $0$ and $1$. The next jump takes more work, and is referred to as the \emph{Fibonacci Dichotomy}.

\begin{theorem}[Kaiser and Klazar~\cite{kaiser:on-growth-rates:}]
\label{thm-fib-dichotomy}
Suppose that $\C$ is a permutation class. If $|\C_n|$ is ever less than the $n$th Fibonacci number, then $|\C_n|$ is eventually polynomial.
\end{theorem}

The Fibonacci numbers referred to above are the \emph{combinatorial} Fibonacci numbers, which begin $1,1,2,3,5,\dots$ for $n=0,1,2,3,4,\dots$. Classes with this enumeration certainly exist. For example, both $\bigoplus\{1,21\}$ and $\bigominus\{1,12\}$ are counted by the Fibonacci numbers. Huczynska and Vatter~\cite{huczynska:grid-classes-an:} gave a proof of Theorem~\ref{thm-fib-dichotomy} using monotone grid classes, which we sketch in Section~\ref{subsec-mono-grid}. We note in passing that while versions of Theorem~\ref{thm-fib-dichotomy} (i.e., jumps from polynomial to superpolynomial growth) are known to exist for a variety of different combinatorial structures (see Klazar~\cite{klazar:on-growth-rates:} for numerous examples), the most general possible form of such a result remains open, as discussed in Pouzet and Thi\'{e}ry~\cite{pouzet:some-relational:}.

It follows that the first three growth rates of permutation classes are $0$, $1$, and $\varphi$, where $\varphi$ is the golden ratio. Next consider the classes $\bigoplus\Sub(k\cdots 21)$. It is easy to see that these classes are counted by (one version of) the \emph{${k}$-generalized Fibonacci numbers}, i.e., that they have generating functions of the form
$$
\frac{1}{1-x-x^2-\cdots-x^k}.
$$
The growth rates of these classes are the largest roots of the polynomials $x^k-x^{k-1}-\cdots-1$, or equivalently (by multiplying by $x-1$), the largest roots of the polynomials $x^{k+1}-2x^k+1$. These classes converge to the class of layered permutations, and thus their growth rates converge to $2$. Kaiser and Klazar~\cite{kaiser:on-growth-rates:} showed that these are the only growth rates of permutation classes below $2$, making $2$ the least accumulation point of growth rates. Their work was later generalized by Balogh, Bollob\'as, and Morris~\cite{balogh:hereditary-prop:ordgraphs}, who showed that growth rates of ordered graph classes take on precisely the same values below $2$ (every permutation class can be viewed as a class of ordered graphs, but not vice versa).

This result has since been extended in Vatter~\cite{vatter:small-permutati:} where the growth rates up to
\[
\kappa=\mbox{the unique real root of $x^3-2x^2-1$}\approx 2.21 
\]
were characterized. One implication of this result is that growth rates of permutation classes and ordered graph classes diverge above $2$. Sections~\ref{subsec-geom-grid} and \ref{subsec-gen-grid} introduce the machinery (geometric and generalized grid classes) needed to study small permutation classes, while Section~\ref{subsec-spc} sketches the proof of the following result. 

\begin{figure}
\begin{footnotesize}
\begin{center}
	\begin{tikzpicture}[scale=6, baseline=(current bounding box.center)]
		\foreach \i in {0, 1, 1.618033989, 1.839286755, 1.927561975, 1.965948237,
1.983582843, 1.991964197, 1.99603118, 1.99802947, 1.999018633,
1.999510402, 1.999755501, 1.999877833, 1.999938939, 1.999969475,
1.999984739, 1.99999237, 1.999996185, 1.999998093, 1.999999046, 2,
2.065994892, 2.094309372, 2.10691934, 2.112673684, 2.115338658,
2.116583568, 2.117167971, 2.11744306, 2.117572742, 2.117633925,
2.117662803, 2.117676437, 2.117682874, 2.117685913, 2.117687349,
2.117688026, 2.117688346, 2.117688498, 2.117688569, 2.117688603,
2.117688633, 2.130395435, 2.147899036, 2.158663882, 2.16349256,
2.165685458, 2.16668849, 2.167149114, 2.16736111, 2.167458792,
2.167503829, 2.167524599, 2.167534181, 2.1675386, 2.16754064,
2.16754158, 2.167542014, 2.167542214, 2.167542307, 2.167542349,
2.167542369, 2.167542378, 2.167542386, 2.17455941, 2.180707499,
2.183384113, 2.185153079, 2.18715546, 2.188063256, 2.188476303,
2.188664608, 2.188750543, 2.188789783, 2.188807705, 2.188815892,
2.188819632, 2.18882134, 2.188822121, 2.188822477, 2.18882264,
2.188822715, 2.188822749, 2.188822764, 2.188822771, 2.188822775,
2.188822777, 2.194600051, 2.196189702, 2.196517555, 2.197383553,
2.197776008, 2.197954192, 2.198035171, 2.198071993, 2.19808874,
2.198096358, 2.198099823, 2.198101399, 2.198102117, 2.198102443,
2.198102591, 2.198102659, 2.19810269, 2.198102704, 2.19810271,
2.198102713, 2.198102714, 2.198102715, 2.200667783, 2.201514427,
2.201897435, 2.202071016, 2.202149758, 2.202185496, 2.20220172,
2.202209087, 2.202212431, 2.20221395, 2.20221464, 2.202214953,
2.202215095, 2.20221516, 2.202215189, 2.202215202, 2.202215208,
2.202215211, 2.202215212, 2.202215213, 2.202215213, 2.203364002,
2.203742556, 2.203913979, 2.203991679, 2.204026915, 2.204042897,
2.204050148, 2.204053437, 2.20405493, 2.204055607, 2.204055914,
2.204056054, 2.204056117, 2.204056145, 2.204056158, 2.204056164,
2.204056167, 2.204056168, 2.204056169, 2.204056169, 2.204056169,
2.204573466, 2.204743865, 2.204821072, 2.204856071, 2.20487194,
2.204879137, 2.204882401, 2.204883881, 2.204884552, 2.204884857,
2.204884995, 2.204885057, 2.204885086, 2.204885099, 2.204885104,
2.204885107, 2.204885108, 2.204885109, 2.204885109, 2.204885109,
2.204885109, 2.205118783, 2.205195756, 2.205230642, 2.205246458,
2.205253629, 2.205256881, 2.205258355, 2.205259024, 2.205259327,
2.205259465, 2.205259527, 2.205259555, 2.205259568, 2.205259574,
2.205259577, 2.205259578, 2.205259578, 2.205259579, 2.205259579,
2.205259579, 2.205259579, 2.205365319, 2.205400153, 2.205415943,
2.205423102, 2.205426348, 2.20542782, 2.205428487, 2.20542879,
2.205428927, 2.205428989, 2.205429017, 2.20542903, 2.205429036,
2.205429039, 2.20542904, 2.20542904, 2.205429041, 2.205429041,
2.205429041, 2.205429041, 2.205429041, 2.205476935, 2.205492713,
2.205499867, 2.20550311, 2.20550458, 2.205505247, 2.205505549,
2.205505686, 2.205505749, 2.205505777, 2.205505789, 2.205505795,
2.205505798, 2.205505799, 2.2055058, 2.2055058, 2.2055058, 2.2055058,
2.2055058, 2.2055058, 2.2055058, 2.205527504, 2.205534655, 2.205537896,
2.205539366, 2.205540033, 2.205540335, 2.205540472, 2.205540534,
2.205540562, 2.205540575, 2.205540581, 2.205540583, 2.205540585,
2.205540585, 2.205540585, 2.205540585, 2.205540586, 2.205540586,
2.205540586, 2.205540586, 2.205540586, 2.205550423, 2.205553665,
2.205555134, 2.205555801, 2.205556103, 2.20555624, 2.205556302,
2.20555633, 2.205556343, 2.205556349, 2.205556351, 2.205556352,
2.205556353, 2.205556353, 2.205556353, 2.205556353, 2.205556353,
2.205556353, 2.205556353, 2.205556353, 2.205556353, 2.205560813,
2.205562283, 2.205562949, 2.205563251, 2.205563388, 2.20556345,
2.205563478, 2.205563491, 2.205563497, 2.205563499, 2.205563501,
2.205563501, 2.205563501, 2.205563501, 2.205563501, 2.205563502,
2.205563502, 2.205563502, 2.205563502, 2.205563502, 2.205563502,
2.205565524, 2.20556619, 2.205566492, 2.205566629, 2.205566691,
2.205566719, 2.205566732, 2.205566738, 2.20556674, 2.205566741,
2.205566742, 2.205566742, 2.205566742, 2.205566742, 2.205566742,
2.205566742, 2.205566742, 2.205566742, 2.205566742, 2.205566742,
2.205566742, 2.205567659, 2.205567961, 2.205568098, 2.20556816,
2.205568188, 2.205568201, 2.205568207, 2.205568209, 2.205568211,
2.205568211, 2.205568211, 2.205568212, 2.205568212, 2.205568212,
2.205568212, 2.205568212, 2.205568212, 2.205568212, 2.205568212,
2.205568212, 2.205568212, 2.205568627, 2.205568764, 2.205568826,
2.205568854, 2.205568867, 2.205568873, 2.205568876, 2.205568877,
2.205568877, 2.205568878, 2.205568878, 2.205568878, 2.205568878,
2.205568878, 2.205568878, 2.205568878, 2.205568878, 2.205568878,
2.205568878, 2.205568878, 2.205568878, 2.205569066, 2.205569128,
2.205569157, 2.205569169, 2.205569175, 2.205569178, 2.205569179,
2.205569179, 2.20556918, 2.20556918, 2.20556918, 2.20556918, 2.20556918,
2.20556918, 2.20556918, 2.20556918, 2.20556918, 2.20556918, 2.20556918,
2.20556918, 2.20556918, 2.205569265, 2.205569293, 2.205569306,
2.205569312, 2.205569315, 2.205569316, 2.205569316, 2.205569317,
2.205569317, 2.205569317, 2.205569317, 2.205569317, 2.205569317,
2.205569317, 2.205569317, 2.205569317, 2.205569317, 2.205569317,
2.205569317, 2.205569317, 2.205569317, 2.205569356, 2.205569368,
2.205569374, 2.205569377, 2.205569378, 2.205569378, 2.205569379,
2.205569379, 2.205569379, 2.205569379, 2.205569379, 2.205569379,
2.205569379, 2.205569379, 2.205569379, 2.205569379, 2.205569379,
2.205569379, 2.205569379, 2.205569379, 2.205569379, 2.205569396,
2.205569402, 2.205569405, 2.205569406, 2.205569407, 2.205569407,
2.205569407, 2.205569407, 2.205569407, 2.205569407, 2.205569407,
2.205569407, 2.205569407, 2.205569407, 2.205569407, 2.205569407,
2.205569407, 2.205569407, 2.205569407, 2.205569407, 2.205569407}{
			\draw (\i,-0.1)--(\i,0.1);
		}
		\draw[very thick] (0,0)--(2.4,0);
		\draw[fill=black] (2.355257,-0.1) rectangle (2.4,0.1);
		\draw (0,-0.1) node[below] {$0$};
		\draw (1,-0.1) node[below] {$1$};
		\draw (1.618033989,-0.1) node[below] {$\varphi$};
		\draw (2,-0.1) node[below] {$2$};
		\draw (2.205569407,-0.1) node[below] {$\kappa$};
		\draw (2.2804132035,-0.1) node[below] {?};
	\end{tikzpicture}
\end{center}
\end{footnotesize}
\caption{The set of all growth rates of permutation classes, as presently known.}
\label{fig-set-of-growth-rates}
\end{figure}
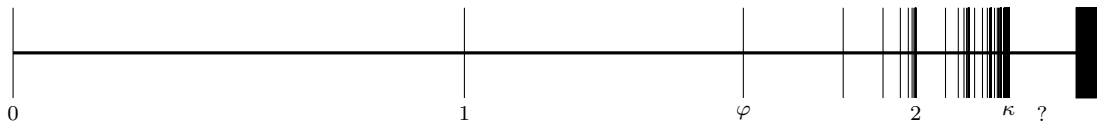

\begin{theorem}[Vatter~\cite{vatter:small-permutati:}]
\label{thm-spc-gr}
The sub-$\kappa$ growth rates of permutation classes consist precisely of $0$, $1$, $2$, and roots of the families of polynomials (for all nonnegative $k$ and $\ell$)
\begin{itemize}
\item $x^{k+1}-2x^k+1$ (the sub-$2$ growth rates, identified by Kaiser and Klazar~\cite{kaiser:on-growth-rates:}),
\item $(x^3-2x^2-1)x^{k+\ell}+x^\ell+1$,
\item $(x^3-2x^2-1)x^k+1$ (accumulation points of growth rates which themselves accumulate at $\kappa$),
\item $x^4-x^3-x^2-2x-3$, $x^5-x^4-x^3-2x^2-3x-1$, $x^3-x^2-x-3$, and $x^4-x^3-x^2-3x-1$ .
\end{itemize}
\end{theorem}

Theorem~\ref{thm-spc-gr} shows that $\kappa$ distinguishes itself on the number-line of growth rates by being the first accumulation point of accumulation points. It is also the least growth rate at which one finds uncountably many permutation classes; indeed, Klazar~\cite{klazar:on-the-least-ex:} originally defined $\kappa$ as
$$
\kappa=\inf \{\gamma \st \mbox{uncountably many classes $\C$ satisfy $\ugr(\C)<\gamma$}\}.
$$
One direction follows without too much work---the growth rate of the downward closure of the infinite antichain from Figure~\ref{fig-infinite-antichain} is equal to $\kappa$, so there are uncountably many classes of growth rate $\kappa$. The other direction will be proved in Section~\ref{subsec-spc}. The phase transition at $\kappa$ also has ramifications for exact enumeration:

\begin{theorem}[Albert, Ru\v{s}kuc, and Vatter~\cite{albert:inflations-of-g:}]
\label{thm-spc-strong-rat}
Every permutation class with growth rate less than $\kappa$ is strongly rational.
\end{theorem}

\index{small permutation class}

For this variety of reasons we call classes of growth rate less than $\kappa$ \emph{small}. At the other end of the spectrum, Albert and Linton~\cite{albert:growing-at-a-pe:} constructed an uncountable set of growth rates of permutation classes, and conjectured that at some point the set of growth rates of permutation classes contains all subsequent real numbers. Let us define
$$
\lambda=\inf\{\gamma\st\mbox{every real number $x>\gamma$ is the growth rate of a permutation class}\}.
$$
By refining their techniques, Vatter~\cite{vatter:permutation-cla} showed that $\lambda$ exists and is less than $2.49$. Bevan~\cite{bevan:intervals-of-pe:} has since lowered this bound, showing that $\lambda<2.36$. Together, this result and Theorem~\ref{thm-spc-gr} establish the number-line of growth rates shown in Figure~\ref{fig-set-of-growth-rates}.

\begin{figure}
\begin{footnotesize}
\begin{center}
	\begin{tikzpicture}[scale=0.2, baseline=(current bounding box.center)]
		\draw [lightgray, ultra thick, line cap=round] (0,0) rectangle (12,12);
		\draw[fill=black] (1,2) circle (10pt);
		\draw[fill=black] (2,3) circle (10pt);
		\draw[fill=black] (3,5) circle (10pt);
		\draw[fill=black] (4,1) circle (10pt);
		\draw[fill=black] (5,7) circle (10pt);
		\draw[fill=black] (6,4) circle (10pt);
		\draw[fill=black] (7,9) circle (10pt);
		\draw[fill=black] (8,6) circle (10pt);
		\draw[fill=black] (9,10) circle (10pt);
		\draw[fill=black] (10,11) circle (10pt);
		\draw[fill=black] (11,8) circle (10pt);
		\draw[rotate around={-45:(1.5,2.5)}] (1.5,2.5) ellipse (25pt and 50pt);
		\draw[rotate around={-45:(9.5,10.5)}] (9.5,10.5) ellipse (25pt and 50pt);
	\end{tikzpicture}
\quad\quad
	\begin{tikzpicture}[scale=0.2, baseline=(current bounding box.center)]
		\draw [lightgray, ultra thick, line cap=round] (0,0) rectangle (11,11);
		\draw[fill=black] (1,2) circle (10pt);
		\draw[fill=black] (2,3) circle (10pt);
		\draw[fill=black] (3,5) circle (10pt);
		\draw[fill=black] (4,1) circle (10pt);
		\draw[fill=black] (5,7) circle (10pt);
		\draw[fill=black] (6,4) circle (10pt);
		\draw[fill=black] (7,10) circle (10pt);
		\draw[fill=black] (8,6) circle (10pt);
		\draw[fill=black] (9,8) circle (10pt);
		\draw[fill=black] (10,9) circle (10pt);
		\draw[rotate around={-45:(1.5,2.5)}] (1.5,2.5) ellipse (25pt and 50pt);
		\draw[rotate around={-45:(9.5,8.5)}] (9.5,8.5) ellipse (25pt and 50pt);
	\end{tikzpicture}
\quad\quad
	\begin{tikzpicture}[scale=0.2, baseline=(current bounding box.center)]
		\draw [lightgray, ultra thick, line cap=round] (0,0) rectangle (12,12);
		\draw[fill=black] (1,3) circle (10pt);
		\draw[fill=black] (2,2) circle (10pt);
		\draw[fill=black] (3,5) circle (10pt);
		\draw[fill=black] (4,1) circle (10pt);
		\draw[fill=black] (5,7) circle (10pt);
		\draw[fill=black] (6,4) circle (10pt);
		\draw[fill=black] (7,9) circle (10pt);
		\draw[fill=black] (8,6) circle (10pt);
		\draw[fill=black] (9,10) circle (10pt);
		\draw[fill=black] (10,11) circle (10pt);
		\draw[fill=black] (11,8) circle (10pt);
		\draw[rotate around={45:(1.5,2.5)}] (1.5,2.5) ellipse (25pt and 50pt);
		\draw[rotate around={-45:(9.5,10.5)}] (9.5,10.5) ellipse (25pt and 50pt);
	\end{tikzpicture}
\quad\quad
	\begin{tikzpicture}[scale=0.2, baseline=(current bounding box.center)]
		\draw [lightgray, ultra thick, line cap=round] (0,0) rectangle (11,11);
		\draw[fill=black] (1,3) circle (10pt);
		\draw[fill=black] (2,2) circle (10pt);
		\draw[fill=black] (3,5) circle (10pt);
		\draw[fill=black] (4,1) circle (10pt);
		\draw[fill=black] (5,7) circle (10pt);
		\draw[fill=black] (6,4) circle (10pt);
		\draw[fill=black] (7,10) circle (10pt);
		\draw[fill=black] (8,6) circle (10pt);
		\draw[fill=black] (9,8) circle (10pt);
		\draw[fill=black] (10,9) circle (10pt);
		\draw[rotate around={45:(1.5,2.5)}] (1.5,2.5) ellipse (25pt and 50pt);
		\draw[rotate around={-45:(9.5,8.5)}] (9.5,8.5) ellipse (25pt and 50pt);
	\end{tikzpicture}
\end{center}
\end{footnotesize}
\caption{An infinite antichain used to make an interval of growth rates.}
\label{fig-gr-interval}
\end{figure}
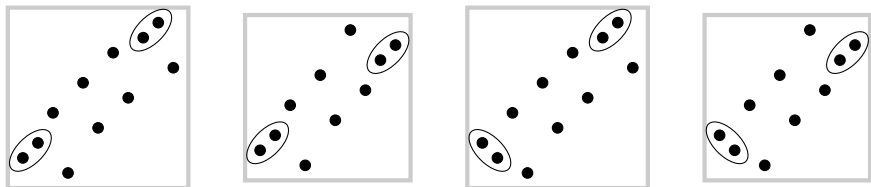

The results about $\lambda$ have been proved by constructing families of sum closed classes with a great deal of flexibility in their growth rates. For example, one of the constructions from \cite{vatter:permutation-cla} features the infinite antichain shown in Figure~\ref{fig-gr-interval}. Let $A$ denote the members of this antichain of length at least $5$. Clearly $A$ contains two permutations of each of these lengths, both of which are sum indecomposable. We can compute the generating function for $\bigoplus\Sub(A)$ if we know how many sum indecomposable permutations of each length are properly contained in a member of $A$. This sequence can be shown to be
$$
(r_n)=1,1,3,5,6,6,6,6,\dots
$$
(starting at $n=1$). Therefore the sequence counting sum indecomposable permutations in $\bigoplus\Sub(A)$ is
$$
(t_n)=1,1,3,5,8,8,8,8,\dots.
$$
If we want to construct a subclass of $\bigoplus\Sub(A)$ we may choose any subset of $A$ together with all of the sum indecomposable permutations properly contained the members of $A$. Thus for any sequence $(s_n)$ which satisfies $r_n\le s_n\le t_n$ for all $n$ we can construct a permutation class with generating function $1/(1-\sum s_nx^n)$. The growth rates of such classes can be shown to consist of the entire interval of real numbers between approximately $2.49$ and $2.51$.

\subsection{Monotone grid classes}
\label{subsec-mono-grid}

As mentioned in Section~\ref{subsec-merge}, Stankova~\cite{stankova:forbidden-subse:} computed the basis of the skew-merged permutations,
$$
\Av(21)\odot\Av(12)=\Av(2143, 3412).
$$
This class is the prototypical example of a monotone grid class, although we need some notation to introduce these in general.

The definition of monotone grid classes is quite similar to the definition of interval minors from Section~\ref{subsec-int-minor}, although here we are subdividing the plot of a permutation, rather than a matrix. Let $\pi$ be a permutation of length $n$ and choose intervals $X,Y\subseteq [1,n]$. We write $\pi(X\times Y)$ to denote the permutation that is order isomorphic to those entries with indices from $X$ and values from $Y$.

\begin{figure}
\begin{footnotesize}
\begin{center}
	\begin{tikzpicture}[scale=0.2, baseline=(current bounding box.center)]
		\draw[lightgray, ultra thick] (0,6.5)--(12,6.5);
		\draw[lightgray, ultra thick] (3.5,0)--(3.5,12);
		\draw[lightgray, ultra thick] (9.5,0)--(9.5,12);
		\draw[lightgray, ultra thick, line cap=round] (0,0) rectangle (12,12);
		\draw[fill=black] (1,2) circle (10pt);
		\draw[fill=black] (2,4) circle (10pt);
		\draw[fill=black] (3,5) circle (10pt);
		\draw[fill=black] (4,6) circle (10pt);
		\draw[fill=black] (5,11) circle (10pt);
		\draw[fill=black] (6,10) circle (10pt);
		\draw[fill=black] (7,3) circle (10pt);
		\draw[fill=black] (8,9) circle (10pt);
		\draw[fill=black] (9,1) circle (10pt);
		\draw[fill=black] (10,8) circle (10pt);
		\draw[fill=black] (11,7) circle (10pt);
	\end{tikzpicture}
\end{center}
\end{footnotesize}
\caption[A monotone gridding of a permutation.]{A $\fnmatrix{rrr}{0&-1&-1\\1&-1&0}$-gridding of a permutation. Here the column divisions are given by $c_1=1$, $c_2=4$, $c_3=10$, and $c_4=12$, while the row divisions are $r_1=1$, $r_2=7$, and $r_3=12$.}
\label{fig-mono-gridding}
\end{figure}
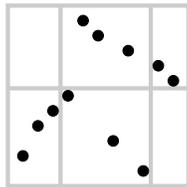

Given a $t\times u$ matrix $M$ consisting of $0$, $1$, and $-1$ entries, an $M$-gridding of the permutation $\pi$ of length $n$ is a choice of column divisions $1=c_1\le\cdots\le c_{t+1}=n+1$ and row divisions $1=r_1\le\cdots\le r_{u+1}=n+1$ such that for all $i$ and $j$, $\pi([c_i,c_{i+1})\times [r_j,r_{j+1}))$ is increasing if $M(i,j)=1$, decreasing if $M(i,j)=-1$, and empty if $M(i,j)=0$. Figure~\ref{fig-mono-gridding} shows an example. The class of all permutations possessing $M$-griddings is called the \emph{monotone grid class} of $M$, and denoted by $\Grid(M)$.

\index{monotone grid class}
\index{grid class}

Our first major result characterizes the classes contained in $\Grid(M)$ for some finite $\zpm$ matrix $M$. We call such classes \emph{monotonically griddable} (thus there are both monotone grid classes and monotonically griddable classes, and it is important to be recognizant of this distinction). First we need to establish an alternate characterization of monotone griddability. Recall from Section~\ref{subsec-subst-decomp} that an axes-parallel rectangle is any rectangle of the form $X\times Y$ for intervals $X$ and $Y$. If $R=X\times Y$ is an axes-parallel rectangle, we let $\pi(R)$ denote the permutation which is order isomorphic to the entries of $\pi$ lying in $R$, that is, with indices in $X$ and values in $Y$. The axes-parallel rectangle $R$ is then \emph{nonmonotone} for the permutation $\pi$ if $\pi(R)$ is nonmonotone. We further say that the line $L$ \emph{slices} the rectangle $R$ if $L$ intersects the interior of $R$. Finally, a collection $\mathfrak{L}$ of lines and a collection $\mathfrak{R}$ of rectangles, we say that $\mathfrak{L}$ slices $\mathfrak{R}$ if every rectangle in $\mathfrak{R}$ is sliced by some line in $\mathfrak{L}$.

\begin{proposition}
\label{prop-slice-nonmonotone-rects}
The class $\C$ is monotonically griddable if and only if there is a constant $\ell$ such that for every permutation $\pi\in\C$, the collection of axes-parallel nonmonotone rectangles of $\pi$ can be sliced by a collection of $\ell$ vertical or horizontal lines.
\end{proposition}
\begin{proof}
First, if $\C\subseteq\Grid(M)$ for a $\zpm$ matrix $M$ of size $t\times u$, then every $\pi\in\C$ has an $M$-gridding with at most $t+u-2$ vertical and horizontal lines and by definition these lines must slice every nonmonotone rectangle of $\pi$.

For the other direction, suppose that there is a constant $\ell$ such that for every $\pi\in\C$ there is a collection $\mathfrak{L}_\pi$ of vertical and horizontal lines that slice every nonmonotone rectangle of $\pi$. These lines define a monotone gridding of $\pi$ of size $t\times u$ with $t+u\le \ell+2$, i.e., they show that $\pi\in\Grid(M)$ for some $\zpm$ matrix of this size. There are only finitely many such matrices, so letting $M^{\oplus}$ denote the direct sum of all such matrices we see that $\C\subseteq\Grid(M^{\oplus})$.
\end{proof}

We can now state and prove the characterization of the monotonically griddable classes. One direction is clear. Using the notation
$$
\oplus^a 21=\underbrace{21\oplus\cdots\oplus 21}_{\mbox{\footnotesize $a$ copies of $21$}},
$$
it follows that if $\oplus^a 21\in\Grid(M)$ then $M$ must have at least $a$ rows and $a$ columns. Analogously, $\ominus^b 12$ cannot lie in the grid class of a matrix that is smaller than $b\times b$. In other words, if $\C$ is monotonically griddable, it must avoid $\oplus^a 21$ and $\ominus^b 12$ for some values of $a$ and $b$. As we prove below, this is also a sufficient condition for monotone griddability. The proof we give relies on Stankova's computation of the basis of the skew-merged permutations.

\begin{theorem}[Huczynska and Vatter~\cite{huczynska:grid-classes-an:}]
\label{thm-mono-griddable}
A permutation class is monotonically griddable if and only if it does not contain arbitrarily long sums of $21$ or skew sums of $12$.
\end{theorem}
\begin{proof}
As we have already observed the other direction of the proof, let $\C$ be a permutation class avoiding $\oplus^a 21$ and $\ominus^b 12$. We prove by induction on $a+b$ that there is a function $\ell(a,b)$ such that given any permutation $\pi\in\C$, the nonmonotone rectangles of $\pi$ can be sliced by a collection of $\ell(a,b)$ vertical or horizontal lines.

If either $a$ or $b$ equals $1$ then $\C$ consists solely of monotone permutations and thus we may set $\ell(1,b)=\ell(a,1)=0$. The next case to consider is $a=b=2$, meaning that $\C$ avoids both $2143$ and $3412$. Thus $\C$ is a subclass of the skew-merged permutations, and so we may set $\ell(2,2)=2$.

\begin{figure}
\begin{footnotesize}
\begin{center}
	\begin{tikzpicture}[scale=0.2, baseline=(current bounding box.center)]
		\filldraw[fill=lightgray, draw=white] (4,4) rectangle (10,10);
		\draw (0,4)--(10,4);
		\draw (4,0)--(4,10);
		\draw [lightgray, ultra thick, line cap=round] (0,0) rectangle (10,10);
		\draw[fill=black] (2,4) circle (10pt);
		\draw[fill=black] (4,2) circle (10pt);
		\draw[fill=black] (6,8) circle (10pt);
		\draw[fill=black] (8,6) circle (10pt);
		\node[below] at (5,0) {(I)};
	\end{tikzpicture}
\quad\quad
	\begin{tikzpicture}[scale=0.2, baseline=(current bounding box.center)]
		\filldraw[fill=lightgray, draw=white] (4,6) rectangle (0,10);
		\draw (0,6)--(10,6);
		\draw (4,0)--(4,10);
		\draw [lightgray, ultra thick, line cap=round] (0,0) rectangle (10,10);
		\draw[fill=black] (2,4) circle (10pt);
		\draw[fill=black] (4,2) circle (10pt);
		\draw[fill=black] (6,8) circle (10pt);
		\draw[fill=black] (8,6) circle (10pt);
		\node[below] at (5,0) {(II)};
	\end{tikzpicture}
\quad\quad
	\begin{tikzpicture}[scale=0.2, baseline=(current bounding box.center)]
		\filldraw[fill=lightgray, draw=white] (0,0) rectangle (6,6);
		\draw (0,6)--(10,6);
		\draw (6,0)--(6,10);
		\draw [lightgray, ultra thick, line cap=round] (0,0) rectangle (10,10);
		\draw[fill=black] (2,4) circle (10pt);
		\draw[fill=black] (4,2) circle (10pt);
		\draw[fill=black] (6,8) circle (10pt);
		\draw[fill=black] (8,6) circle (10pt);
		\node[below] at (5,0) {(III)};
	\end{tikzpicture}
\quad\quad
	\begin{tikzpicture}[scale=0.2, baseline=(current bounding box.center)]
		\filldraw[fill=lightgray, draw=white] (6,4) rectangle (10,0);
		\draw (0,4)--(10,4);
		\draw (6,0)--(6,10);
		\draw[lightgray, ultra thick, line cap=round] (0,0) rectangle (10,10);
		\draw[fill=black] (2,4) circle (10pt);
		\draw[fill=black] (4,2) circle (10pt);
		\draw[fill=black] (6,8) circle (10pt);
		\draw[fill=black] (8,6) circle (10pt);
		\node[below] at (5,0) {(IV)};
	\end{tikzpicture}
\quad\quad
\end{center}
\end{footnotesize}
\caption{The four regions in the proof of Theorem~\ref{thm-mono-griddable}.}
\label{fig-mono-griddable}
\end{figure}
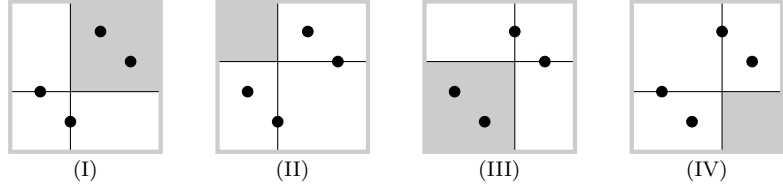

For the inductive step, we may assume by symmetry that $a\ge 3$ and $b\ge 2$. Take an arbitrary $\pi\in\C$. If $\pi$ avoids $\oplus^2 21=2143$ then its nonmonotone rectangles can be sliced by $\ell(2,b)$ vertical and horizontal lines by induction. Thus we may assume that $\pi$ contains $2143$. We fix a specific copy of $2143$ in $\pi$ and then partition the entries of $\pi$ into the four regions shown in Figure~\ref{fig-mono-griddable}. Clearly every nonmonotone rectangle either lies completely in one of these regions or is sliced by one of the four lines bordering these regions. Now notice that the entries in regions (I) and (III) avoid $\oplus^{a-1} 21$, while those in regions (II) and (IV) avoid $\ominus^{b-1} 12$. This shows that we can set
$$
\ell(a,b)=2\ell(a-1,b)+2\ell(a,b-1)+4,
$$
completing the proof.
\end{proof}

This result allows us to sketch the proof of the Fibonacci Dichotomy (Theorem~\ref{thm-fib-dichotomy}). Suppose that $\C$ is a permutation class and $|\C_n|$ is less than the $n$th Fibonacci number. The classes $\bigoplus\{1,21\}$ and $\bigominus\{1,12\}$ are both counted by the Fibonacci numbers. Therefore $\C$ cannot contain either of these classes, and thus is monotonically griddable by Theorem~\ref{thm-mono-griddable}.

Next we recall the definition of alternations, first encountered in Section~\ref{subsec-subst-decomp}. An alternation is a permutation whose plot can be divided into two parts, by a single horizontal or vertical line, so that for every pair of entries from the same part there is an entry from the other part that separates them. A parallel alternation is one in which the two halves of the alternation form monotone subsequences, either both increasing or both decreasing, while for a wedge alternation one of these is increasing and the other is decreasing. It follows from the Erd\H{o}s-Szekeres Theorem that every sufficiently long alternation contains a long parallel or wedge alternation. If $\C$ were to contain arbitrarily long alternations of either type, its growth rate would be at least $2$. Thus there is a bound on the length of alternations in $\C$. We now appeal to the following result.

\begin{proposition}[Huczynska and Vatter~\cite{huczynska:grid-classes-an:}]
\label{prop-chop-gridding}
Suppose that $\C$ is monotonically griddable and that the length of alternations in $\C$ is bounded. Then $\C\subseteq\Grid(M)$ for a $\zpm$ matrix $M$ in which no two nonzero entries share a row or column.
\end{proposition}

This result shows that if $\C$ has sub-Fibonacci enumeration, then it is contained in the monotone grid class of a ``signed permutation matrix''. All that remains is to show that such classes have eventually polynomial enumeration. Both Kaiser and Klazar~\cite{kaiser:on-growth-rates:} and Huczynska and Vatter~\cite{huczynska:grid-classes-an:} did so by bijectively associating such classes with downsets of vectors in $\mathbb{N}^t$ and then appealing to a 1976 \emph{Monthly} problem posed (and solved) by Stanley~\cite{stanley:solution-to-pro:}. Homberger and Vatter~\cite{homberger:on-the-effectiv:} present a constructive proof that gives an algorithm for computing these polynomials.

Monotone grid classes were first introduced (under a different name) in Murphy and Vatter~\cite{murphy:profile-classes:}, where the focus was on which monotone grid classes are wpo. Let $M$ be a $\zpm$ matrix of size $t\times u$. The \emph{cell graph} of $M$ is the graph on the vertices $\{(i,j) : M(i,j)\neq 0\}$ in which $(i,j)$ and $(k,\ell)$ are adjacent if the corresponding cells of $M$ share a row or a column and there are no nonzero entries between them in this row or column.  We say that the matrix $M$ is a \emph{forest} if its cell graph is a forest.  Viewing the absolute value of $M$ as the adjacency matrix of a bipartite graph, we obtain a different graph, its \emph{row-column graph}, which is the bipartite graph on the vertices $x_1,\dots,x_t,y_1,\dots,y_u$ where there is an edge between $x_i$ and $y_j$ if and only if $M(i,j)\neq 0$. It is not difficult to show that the cell graph of a matrix is a forest if and only if its row-column graph is also a forest (a formal proof is given in Vatter and Waton~\cite{vatter:on-partial-well:}).

\begin{theorem}[Murphy and Vatter~\cite{murphy:profile-classes:}]
\label{thm-mono-grid-wpo}
The grid class $\Grid(M)$ is wpo if and only if the cell graph of $M$ is a forest.
\end{theorem}

Theorem~\ref{thm-mono-grid-wpo} has since been generalized in several directions. In the next subsection we will see a generalization of the wpo part of this result (Theorem~\ref{thm-ggc-strong-rat}). Brignall~\cite{brignall:grid-classes-an:} has presented another generalization of Theorem~\ref{thm-mono-grid-wpo} in the context of the generalized grid classes of Section~\ref{subsec-gen-grid}.

Somewhat surprisingly, we know almost nothing about the bases of monotone grid classes. The skew-merged permutations have a finite basis, and Waton~\cite{waton:on-permutation-:} showed that the monotone grid class of $J_2$ (the $2\times 2$ all-one matrix) has a finite basis. Thus we remain far away from the following.

\begin{conjecture}
\label{conj-mono-grid-basis}
Every monotone grid class has a finite basis.
\end{conjecture}

We have used grid classes to establish the Fibonacci dichotomy, but what about growth rates of grid classes themselves? In exploring this question, Bevan~\cite{bevan:growth-rates-of:} uncovered a beautiful connection between permutation patterns and algebraic graph theory. Recall that the \emph{spectral radius} of a graph is the largest eigenvalue of its adjacency matrix.

\begin{theorem}[Bevan~\cite{bevan:growth-rates-of:}]
\label{thm-mono-grid-gr}
The growth rate of $\Grid(M)$ exists and is equal to the square of the spectral radius of the row-column graph of $M$.
\end{theorem}

It is well beyond the scope of this survey to attempt to sketch Bevan's proof, but we nonetheless attempt to convey some essence of his approach. Every connected component in the row-column graph of $M$ corresponds to some submatrix of $M$, and we call this submatrix a \emph{connected component} of $M$. It is not difficult to establish the following result, which was first used in Vatter~\cite{vatter:small-permutati:}.

\begin{proposition}
\label{prop-grid-component}
The upper growth rate of $\Grid(M)$ is equal to the greatest upper growth rate of the monotone grid class of a connected component of $M$.
\end{proposition}

Spectral radii follow the analogous pattern---the spectral radius of the graph $G$ is equal to the greatest spectral radius of a connected component of $G$. Therefore it suffices to consider grid classes with connected row-column graphs. A \emph{tour} on a graph is a walk (a sequence of not-necessarily distinct vertices, each connected by an edge) that ends where it began. Bevan showed that if the row-column graph of $M$ is connected, then the number of permutations of length $n$ in $\Grid(M)$ is, up to a polynomial factor, equal to the number of \emph{balanced} tours of length $2n$ in the row-column graph of $G$, where a balanced tour is a tour that traverses every edge the same number of times in each direction. He then showed that this balance condition does not affect the asymptotics of tours, establishing the theorem.

As a consequence of Bevan's Theorem~\ref{thm-mono-grid-gr}, we may appeal to the significant literature on algebraic graph theory to determine the possible growth rates of monotone grid classes. First, obviously every growth rate of a monotone grid class is an algebraic integer, because it is the square of an eigenvalue of a zero/one matrix. More strikingly, if the row-column graph of $M$ is a cycle, then the growth rate of $\Grid(M)$ is equal to $4$, no matter how long the cycle. Also, if the growth rate of a monotone grid class is less than this, then it is equal to $4\cos^2(\pi/k)$ for some integer $k\ge 3$. Finally, for every growth rate $\gamma\ge 2+\sqrt{5}\approx 4.24$, there is a monotone grid class with growth rate arbitrarily close to $\gamma$.

\begin{figure}
\begin{footnotesize}
$$
\begin{array}{ccc}
	\begin{tikzpicture}[scale=1, baseline=(current bounding box.center)]
		\foreach \i in {0,1,2}{
			\draw [lightgray, ultra thick, line cap=round] (0,\i)--(3,\i);
		}
		\foreach \i in {0,1,2,3}{
			\draw [lightgray, ultra thick, line cap=round] (\i,0)--(\i,2);
		}
		\draw [thick, line cap=round] (3,1) arc [radius=1, start angle=0, end angle=90];
		\draw [thick, line cap=round] (0,1)--(1,0);
		\draw [thick, line cap=round] (1,0)--(2,1);
		\draw [thick, line cap=round] (2,0)--(3,1);
		\draw [thick, line cap=round] (0,1)--(1,2);
		\draw [thick, line cap=round] (1,1)--(2,2);
	\end{tikzpicture}
&&
	\begin{tikzpicture}[scale=1, baseline=(current bounding box.center)]
		\foreach \i in {0,1,2}{
			\draw [lightgray, ultra thick, line cap=round] (0,\i)--(3,\i);
		}
		\foreach \i in {0,1,2,3}{
			\draw [lightgray, ultra thick, line cap=round] (\i,0)--(\i,2);
		}
		\draw [thick, line cap=round] (2,2) arc [radius=1, start angle=180, end angle=270];
		\draw [thick, line cap=round] (0,1)--(1,0);
		\draw [thick, line cap=round] (1,0)--(2,1);
		\draw [thick, line cap=round] (2,0)--(3,1);
		\draw [thick, line cap=round] (0,1)--(1,2);
		\draw [thick, line cap=round] (1,1)--(2,2);
	\end{tikzpicture}
\end{array}
$$
\end{footnotesize}
\caption[Two pictures which have no joint embedding in a particular monotone grid class.]{Waton proved that the monotone grid class of the matrix $\fnmatrix{rrr}{1&1&-1\\-1&1&1}$ is equal to the union of the two figure classes shown above.}
\label{fig-waton-nonatomic}
\end{figure}

The final question we address in this subsection is whether monotone grid classes are atomic. Many monotone grid classes (such as the class of skew-merged permutations) are atomic. However, monotone grid classes are not, in general, atomic. In his thesis~\cite{waton:on-permutation-:}, Waton established necessary and sufficient conditions for a grid class to be atomic. Let $M$ be a $\zpm$ matrix. Given a cycle in its row-column graph, we say that the \emph{sign} of the cycle is the product of the entries corresponding to its edges. So, for example, the matrix
$$
\fnmatrix{rrr}{1&1&-1\\-1&1&1}
$$
contains a positive cycle (corresponding to columns $1$ and $3$) and two negative cycles (corresponding to columns $1$ and $2$ and columns $2$ and $3$). Indeed, the grid class of this matrix is not atomic; Waton showed that it is the union of the two figure classes shown in Figure~\ref{fig-waton-nonatomic}. We conclude with his full characterization.

\begin{theorem}[Waton~\cite{waton:on-permutation-:}]
The monotone grid class $\Grid(M)$ is atomic if and only if the row-column graph of $M$ does not have a connected component containing a negative cycle together with any other cycle.
\end{theorem}

\subsection{Geometric grid classes}
\label{subsec-geom-grid}

\index{geometric grid class}

In Section~\ref{subsec-basics} we introduced the class of permutations that can be drawn on an $\textsf{X}$ and called it $\Sub(\textsf{X})$, though to avoid cumbersome notation we denote it by $\mathcal{X}$ here. We return to this class throughout this subsection because it is the prototypical example of a \emph{geometric} grid class. Clearly the class $\mathcal{X}$ is a subclass of the skew-merged permutations, but it is not all of the skew-merged permutations. For example, the permutation $3142$ cannot be drawn on an $\textsf{X}$ because once we place the $3$, $1$, and $4$ on the $\textsf{X}$, there is no place for the $2$ to lie simultaneously above the $1$ and to the right of the $4$:
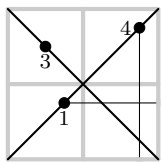
\begin{center}
\begin{footnotesize}
	\begin{tikzpicture}[scale=1, baseline=(current bounding box.center)]
		\draw[fill=black] (0.5,1.5) circle (2pt) node[below] {$3$};
		\draw[fill=black] (0.75,0.75) circle (2pt) node[below] {$1$};
		\draw[fill=black] (1.75,1.75) circle (2pt) node[left] {$4$};
		\foreach \i in {0,1,2}{
			\draw [lightgray, ultra thick, line cap=round] (0,\i)--(2,\i);
		}
		\foreach \i in {0,1,2}{
			\draw [lightgray, ultra thick, line cap=round] (\i,0)--(\i,2);
		}
		\draw [thick, line cap=round] (0,2)--(2,0);
		\draw [thick, line cap=round] (0,0)--(2,2);
		\draw (0.75,0.75)--(2,0.75);
		\draw (1.75,1.75)--(1.75,0);
		\draw [lightgray, ultra thick] (1.7,0)--(1.8,0);
		\draw [lightgray, ultra thick] (2,0.7)--(2,0.8);
	\end{tikzpicture}
\end{footnotesize}
\end{center}
By symmetry, $2413$ also cannot be drawn on an $\textsf{X}$, so $\mathcal{X}$ is also a subclass of the separable permutations. Indeed, it is not hard to see that $\mathcal{X}$ is the class of skew-merged separable permutations,
$$
\mathcal{X}=\Av(2143,2413,3142,3412),
$$
because every permutation drawn on an $\textsf{X}$ must have some point that is at least as far away from the center of the $\textsf{X}$ as every other point. Therefore every permutation in $\mathcal{X}$ is of the form $1\oplus\pi$, $\pi\oplus 1$, $1\ominus\pi$, or $\pi\ominus 1$ for some $\pi\in\mathcal{X}$. This leads quickly to the generating function of the class. From the above, we have
$$
\mathcal{X}
=
\{1\}
\cup
\left(1\oplus\mathcal{X}\right)
\cup
\left(\mathcal{X}\oplus 1\right)
\cup
\left(1\ominus\mathcal{X}\right)
\cup
\left(\mathcal{X}\ominus 1\right).
$$
Moreover, $\left(1\oplus\mathcal{X}\right)\cap \left(\mathcal{X}\oplus 1\right)=\left(1\oplus\mathcal{X}\oplus 1\right)$, so we see (by symmetry) that if $f$ is generating function for nonempty permutations in $\mathcal{X}$, $f=x+4xf-2x^2f$. Thus the generating function for $\mathcal{X}$ is
$$
\frac{x}{1-4x+2x^2}.
$$ 
It follows that the growth rate of $\mathcal{X}$ is $2+\sqrt{2}\approx 3.41$. This is significantly lower than the growth rate of the skew-merged permutations ($4$, which follows from Atkinson~\cite{atkinson:permutations-wh:} or Bevan's Theorem~\ref{thm-mono-grid-gr}) or that of the separable permutations (approximately $5.83$, as we saw in Section~\ref{subsec-basics}).

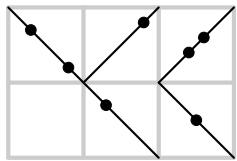
\begin{figure}
\begin{footnotesize}
\begin{center}
	\begin{tikzpicture}[scale=1, baseline=(current bounding box.center)]
		\foreach \i in {0,1,2}{
			\draw [lightgray, ultra thick, line cap=round] (0,\i)--(3,\i);
		}
		\foreach \i in {0,1,2,3}{
			\draw [lightgray, ultra thick, line cap=round] (\i,0)--(\i,2);
		}
		\draw [thick, line cap=round] (0,2)--(2,0);
		\draw [thick, line cap=round] (1,1)--(2,2);
		\draw [thick, line cap=round] (2,1)--(3,2);
		\draw [thick, line cap=round] (2,1)--(3,0);
		\draw[fill=black] (0.3,1.7) circle (2pt);
		\draw[fill=black] (0.8,1.2) circle (2pt);
		\draw[fill=black] (1.3,0.7) circle (2pt);
		\draw[fill=black] (1.8,1.8) circle (2pt);
		\draw[fill=black] (2.4,1.4) circle (2pt);
		\draw[fill=black] (2.5,0.5) circle (2pt);
		\draw[fill=black] (2.6,1.6) circle (2pt);
	\end{tikzpicture}
\end{center}
\end{footnotesize}
\caption[The permutation $6327415$ lies in a geometric grid class.]{The permutation $6327415$ lies in the geometric grid class of the matrix $\fnmatrix{rrr}{-1&1&1\\0&-1&-1}$.}
\label{fig-ggc-example}
\end{figure}

Suppose that we have a $\zpm$ matrix $M$, as in the definition of a monotone grid class. The \emph{standard figure} of $M$, denoted $\Lambda_M$, is the figure in $\mathbb{R}^2$ consisting of two types of line segments for every pair of indices $i,j$ such that $M(i,j)\neq 0$:
\begin{itemize}
\item the increasing open line segment from $(i-1,j-1)$ to $(i,j)$ if $M(i,j)=1$ or
\item the decreasing open line segment from $(i-1,j)$ to $(i,j-1)$ if $M(i,j)=-1$.
\end{itemize}
The \emph{geometric grid class} of $M$, denoted $\Geom(M)$, is the set of all permutations that can be drawn on $\Lambda_M$ (an example is shown in Figure~\ref{fig-ggc-example}). The inclusion $\Geom(M)\subseteq\Grid(M)$ always holds, but our example of $\mathcal{X}$ shows that the geometric grid class of $M$ may be a proper subclass of its monotone grid class. Indeed, this happens precisely when $M$ contains a cycle:

\begin{figure}
\begin{footnotesize}
\begin{center}
	\begin{tikzpicture}[scale=0.2222222222222, baseline={(0,0)}]
		\draw [lightgray, ultra thick, line cap=round] (0,0) rectangle (9,9);
		\draw [lightgray, ultra thick, line cap=round] (0,3.5)--(9,3.5);
		\draw [lightgray, ultra thick, line cap=round] (6.5,0)--(6.5,9);
		\draw (0,3.5)--(2,3)--(4,2)--(5,1)--(6.5,0);
		\draw (0,3.5)--(1,4)--(3,6)--(6,8)--(6.5,9);
		\draw (6.5,3.5)--(7,5)--(8,7)--(9,9);
		\draw[fill=black] (1,4) circle (9pt);
		\draw[fill=black] (2,3) circle (9pt);
		\draw[fill=black] (3,6) circle (9pt);
		\draw[fill=black] (4,2) circle (9pt);
		\draw[fill=black] (5,1) circle (9pt);
		\draw[fill=black] (6,8) circle (9pt);
		\draw[fill=black] (7,5) circle (9pt);
		\draw[fill=black] (8,7) circle (9pt);
		\draw[<->, line cap=round] (-0.5,3.5)--(-0.5,9);
	\end{tikzpicture}
\quad\quad
	\begin{tikzpicture}[scale=0.2222222222222, baseline={(0,0)}]
		\draw [lightgray, ultra thick, line cap=round] (0,0) rectangle (9,10);
		\draw [lightgray, ultra thick, line cap=round] (0,3.5)--(9,3.5);
		\draw [lightgray, ultra thick, line cap=round] (6.5,0)--(6.5,10);
		\draw (0,3.5)--(2,3)--(4,2)--(5,1)--(6.5,0);
		\draw [thick, line cap=round] (0,3.5)--(6.5,10);
		\draw (6.5,3.5)--(7,5.5)--(8,8)--(9,10);
		\draw[fill=black] (1,4.5) circle (9pt);
		\draw[fill=black] (2,3) circle (9pt);
		\draw[fill=black] (3,6.5) circle (9pt);
		\draw[fill=black] (4,2) circle (9pt);
		\draw[fill=black] (5,1) circle (9pt);
		\draw[fill=black] (6,9.5) circle (9pt);
		\draw[fill=black] (7,5.5) circle (9pt);
		\draw[fill=black] (8,8) circle (9pt);
		\draw[<->, line cap=round] (-0.5,0)--(-0.5,3.5);
	\end{tikzpicture}
\quad\quad
	\begin{tikzpicture}[scale=0.2222222222222, baseline={(0,0)}]
		\draw [lightgray, ultra thick, line cap=round] (0,-3) rectangle (9,10);
		\draw [lightgray, ultra thick, line cap=round] (0,3.5)--(9,3.5);
		\draw [lightgray, ultra thick, line cap=round] (6.5,-3)--(6.5,10);
		\draw [thick, line cap=round] (0,3.5)--(6.5,-3);
		\draw [thick, line cap=round] (0,3.5)--(6.5,10);
		\draw (6.5,3.5)--(7,5.5)--(8,8)--(9,10);
		\draw[fill=black] (1,4.5) circle (9pt);
		\draw[fill=black] (2,1.5) circle (9pt);
		\draw[fill=black] (3,6.5) circle (9pt);
		\draw[fill=black] (4,-0.5) circle (9pt);
		\draw[fill=black] (5,-1.5) circle (9pt);
		\draw[fill=black] (6,9.5) circle (9pt);
		\draw[fill=black] (7,5.5) circle (9pt);
		\draw[fill=black] (8,8) circle (9pt);
		\draw[<->, line cap=round] (6.5,-3.5)--(9,-3.5);
	\end{tikzpicture}
\quad\quad
	\begin{tikzpicture}[scale=0.2222222222222, baseline={(0,0)}]
		\draw [lightgray, ultra thick, line cap=round] (0,-3) rectangle (13,10);
		\draw [lightgray, ultra thick, line cap=round] (0,3.5)--(13,3.5);
		\draw [lightgray, ultra thick, line cap=round] (6.5,-3)--(6.5,10);
		\draw [thick, line cap=round] (0,3.5)--(6.5,-3);
		\draw [thick, line cap=round] (0,3.5)--(6.5,10);
		\draw [thick, line cap=round] (6.5,3.5)--(13,10);
		\draw[fill=black] (1,4.5) circle (9pt);
		\draw[fill=black] (2,1.5) circle (9pt);
		\draw[fill=black] (3,6.5) circle (9pt);
		\draw[fill=black] (4,-0.5) circle (9pt);
		\draw[fill=black] (5,-1.5) circle (9pt);
		\draw[fill=black] (6,9.5) circle (9pt);
		\draw[fill=black] (8.5,5.5) circle (9pt);
		\draw[fill=black] (11,8) circle (9pt);
	\end{tikzpicture}
\end{center}
\end{footnotesize}
\caption{``Straightening'' a member of a monotone grid class to show that it also lies in the corresponding geometric grid class.}
\label{fig-straighten-forest}
\end{figure}
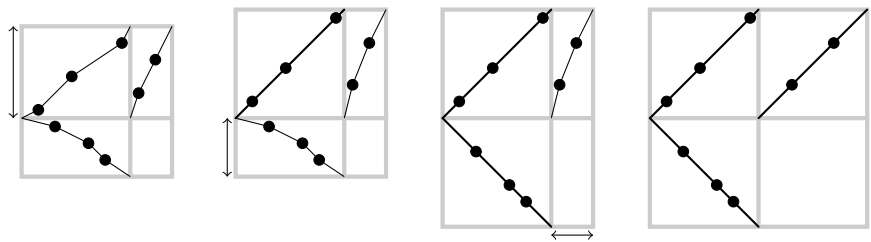

\begin{proposition}
\label{prop-forests-are-geoms}
The classes $\Grid(M)$ and $\Geom(M)$ are equal if and only if the cell graph of $M$ is a forest.
\end{proposition}

Figure~\ref{fig-straighten-forest} gives a sense of how to prove Proposition~\ref{prop-forests-are-geoms}. Suppose that the cell graph of $M$ is a tree (the forest case follows easily from the tree case) and choose one cell to be the root. Now take the plot of any gridded permutation $\pi\in\Grid(M)$ and stretch (or shrink) this row vertically so that the points lie on a line of slope $\pm 1$. Next, for every neighbor cell in the same row (there can be at most two), we stretch the $x$-axis, while for every neighbor cell in the same column we stretch the $y$-axis. Because $M$ is a tree, we can continue this process throughout all cells without having to revisit a cell.

Next we turn to the enumeration of geometric grid classes, where it turns out we can say a lot more than we could for monotone grid classes. First, though, we need a bit more precision. There are infinitely many ways to draw every permutation $\pi\in\Geom(M)$ on the standard figure $\Lambda_M$ because we can move the points by tiny amounts without changing the underlying permutation, but clearly we want to consider minute changes such as this as equivalent. Therefore we say that a \emph{gridded permutation} is a (drawing of a) permutation together with grid lines corresponding to $M$, and we denote the set of all gridded permutations in $\Geom(M)$ by $\Geom^\gridded(M)$. Every permutation in $\Geom(M)$ then corresponds to at least one gridded permutation in $\Geom^\gridded(M)$, and not more than ${n\choose t-1}{n\choose u-1}$, because there are only so many places we can insert the grid lines. Thus we obtain the following.

\begin{observation}
\label{obs-gr-gridded}
The (upper, lower, and proper) growth rates of $\Geom(M)$ and $\Geom^\gridded(M)$ are identical.
\end{observation}

We would like to describe an encoding of the gridded permutations in $\Geom^\gridded(M)$ by words over a finite alphabet. First let us return to the permutations drawn on an $\textsf{X}$ to see an easy example before presenting the general construction. Below is our drawing from Figure~\ref{fig-drawnon} with two changes: first, the points are labeled by what quadrant they lie in (we consider the center of the $\textsf{X}$ to be the origin), and second, the line segments have been assigned orientations.
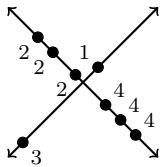
\begin{center}
\begin{footnotesize}
	\begin{tikzpicture}[scale=0.1, baseline=(current bounding box.center)]
		\draw [<->, thick, line cap=round] (-10,10)--(0,0)--(10,-10);
		\draw [<->, thick, line cap=round] (-10,-10)--(0,0)--(10,10);
		\useasboundingbox (current bounding box.south west) rectangle (current bounding box.north east);
		\draw[fill=black] (-8,-8) circle (20pt) node [below right] {$3$};
		\draw[fill=black] (-6,6) circle (20pt) node [below left] {$2$};
		\draw[fill=black] (-4,4) circle (20pt) node [below left] {$2$};
		\draw[fill=black] (-1,1) circle (20pt) node [below left] {$2$};
		\draw[fill=black] (2,2) circle (20pt) node [above left] {$1$};
		\draw[fill=black] (3,-3) circle (20pt) node [above right] {$4$};
		\draw[fill=black] (5,-5) circle (20pt) node [above right] {$4$};
		\draw[fill=black] (7,-7) circle (20pt) node [above right] {$4$};
	\end{tikzpicture}
\end{footnotesize}
\end{center}
To encode this gridded permutation, we order these points according to their distance from the beginning of their line segment, which in this case is also their distance from the center of the figure, and record the labels of the points in this order. While it may take a ruler to verify it, in our example above the encoding is $21242443$.

Suppose that we have an arbitrary $\zpm$ matrix $M$. The important property of the orientation above is that it is \emph{consistent}---in each column all lines are oriented either left or right, and in each row all lines are oriented either up or down. This is not possible for all matrices, for example, the standard figure of the matrix
$$
\fnmatrix{rr}{1&1\\1&-1}
$$
cannot be consistently oriented. But there is a way to remedy this. The grid lines of the standard figure consist of $x=i$ and $y=i$ for all (relevant) integers $i$. If we also add grid lines at half-integer values of $x$ and $y$, each cell is chopped into four but the geometric grid class itself is not changed.

In terms of the matrix we started with, this operation is equivalent to performing the substitutions
$$
0\leftarrow\fnmatrix{rr}{0&0\\0&0},
\quad\quad
1\leftarrow\fnmatrix{rr}{0&1\\1&0},
\quad\quad
-1\leftarrow\fnmatrix{rr}{-1&0\\0&-1},
$$
and we call the resulting matrix, written $M^{\times 2}$, the \emph{double refinement} of $M$. Thus the standard figure of $M^{\times 2}$ consists of $2\times 2$ blocks that are all subfigures of the $\textsf{X}$, and thus we can use its consistent orientation to define a consistent orientation of $M^{\times 2}$.

For the rest of this discussion let us suppose that we have a consistent orientation for the standard figure of $M$ (replacing $M$ by $M^{\times 2}$ if necessary). Next we fix a \emph{cell alphabet}, $\Sigma_M$, containing one letter corresponding to each nonempty cell of $M$. Finally, we encode a gridded drawing (of a permutation) by ordering its points by their distance from the beginning of their line segment (we can assume there are no ties by possibly moving points by a minuscule amount) and then recording the corresponding ``cell letter'' of each point in this order.

Formally, we have just defined a map
$$
\varphi^\gridded\st \Sigma_M^\ast\longrightarrow\Geom^\gridded(M).
$$
This is in general a many-to-one map, because the letters of $\Sigma_M$ often ``commute''; for example, in our encoding of permutations drawn on an $\textsf{X}$, interchanging adjacent occurrences of $1$ and $3$ does not change the image (the gridded permutation), and the same holds for $2$ and $4$. This is because the corresponding cells of the standard figure are \emph{independent}, meaning that they share neither a row nor a column. For general matrices $M$, we write $v\equiv_M w$ if $w$ can be obtained from $v$ via a sequence of interchanges of adjacent occurrences of letters corresponding to independent cells. It is not difficult to see that $\varphi^\gridded$ is actually a bijection when restricted to equivalence classes of words modulo $\equiv_M$, i.e., that the map
$$
\varphi^\gridded\st \Sigma_M^\ast/\equiv_M\longrightarrow \Geom^\gridded(M)
$$
is a bijection.

Objects such as $\Sigma^\ast/\equiv$ are known as \emph{trace (or partially commutative) monoids}. They were first introduced by Cartier and Foata~\cite{cartier:problemes-combi:} in 1969 who used them to give a combinatorial proof of MacMahon's Master Theorem. Cartier and Foata showed via M\"obius inversion that the generating function for equivalence classes of $\Sigma^\ast/\equiv$ (by length) is
$$
\frac{1}{1-(c_1x-c_2x^2+c_3x^3-c_4x^4+\cdots)},
$$
where $c_k$ is the number of $k$-elements subsets of $\Sigma$ that pairwise commute (see also Flajolet and Sedgewick~\cite[V.3.3]{flajolet:analytic-combin:}). In our example of permutations drawn on an $\textsf{X}$, we have that $c_0=1$, $c_1=4$, $c_2=2$, and $c_3=c_4=0$. Thus the generating function for the corresponding trace monoid is
$$
\frac{1}{1-4x+2x^2}.
$$
As Observation~\ref{obs-gr-gridded} tells us should be the case, this generating function has the same growth rate as the class $\mathcal{X}$.

Growth rates of trace monoids were studied by Goldwurm and Santini~\cite{goldwurm:clique-polynomi:}. Bevan~\cite{bevan:growth-rates-of:geom} used their work to express the growth rate of $\Geom(M)$ in terms of the row-column graph of $M$, giving a geometric analogue of his Theorem~\ref{thm-mono-grid-gr}. A \emph{${k}$-matching} of a graph is a set of $k$ edges, no two incident with the same vertex. Letting $m_k(G)$ denote the number of $k$-matchings of a graph on $n$ vertices, its \emph{matching polynomial} is defined by
$$
\mu_G(x)=\sum_{k\ge 0} (-1)^k m_k(G) x^{n-2k}.
$$
(No matching can contain over $\lfloor n/2\rfloor$ edges, so this is indeed a polynomial.) We state Bevan's result below in terms of the double refinement of $M$ because this is necessary in order to guarantee that we have a consistent orientation.

\begin{theorem}[Bevan~\cite{bevan:growth-rates-of:geom}]
The growth rate of $\Geom(M)$ exists and is equal to the square of the largest root of the matching polynomial of the row-column graph of $M^{\times 2}$.
\end{theorem}

In contrast to Theorem~\ref{thm-mono-grid-gr}, this result shows that changing the sign of an entry of $M$ can change the growth rate of the corresponding geometric grid class. For example, Bevan~\cite{bevan:growth-rates-of:geom} notes that
$$
\gr\left(\Geom\fnmatrix{rrr}{-1&0&-1\\1&-1&1}\right)
=
4
<
3+\sqrt{2}
=
\gr\left(\Geom\fnmatrix{rrr}{1&0&-1\\1&-1&1}\right).
$$


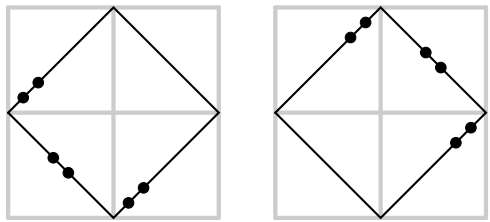
\begin{figure}
\begin{center}
\begin{footnotesize}
	\begin{tikzpicture}[scale=0.2, baseline=(current bounding box.center)]
		\foreach \i in {0,7,14}{
			\draw [lightgray, ultra thick, line cap=round] (0,\i)--(14,\i);
		}
		\foreach \i in {0,7,14}{
			\draw [lightgray, ultra thick, line cap=round] (\i,0)--(\i,14);
		}
		\draw [thick, line cap=round] (0,7)--(7,14)--(14,7)--(7,0)--(0,7);
		\draw[fill=black] (1,8) circle (10pt); 
		\draw[fill=black] (2,9) circle (10pt); 
		\draw[fill=black] (3,4) circle (10pt); 
		\draw[fill=black] (4,3) circle (10pt); 
		\draw[fill=black] (8,1) circle (10pt); 
		\draw[fill=black] (9,2) circle (10pt); 
	\end{tikzpicture}
\quad\quad
	\begin{tikzpicture}[scale=0.2, baseline=(current bounding box.center)]
		\foreach \i in {0,7,14}{
			\draw [lightgray, ultra thick, line cap=round] (0,\i)--(14,\i);
		}
		\foreach \i in {0,7,14}{
			\draw [lightgray, ultra thick, line cap=round] (\i,0)--(\i,14);
		}
		\draw [thick, line cap=round] (0,7)--(7,14)--(14,7)--(7,0)--(0,7);
		\draw[fill=black] (5,12) circle (10pt); 
		\draw[fill=black] (6,13) circle (10pt); 
		\draw[fill=black] (10,11) circle (10pt); 
		\draw[fill=black] (11,10) circle (10pt); 
		\draw[fill=black] (12,5) circle (10pt); 
		\draw[fill=black] (13,6) circle (10pt); 
	\end{tikzpicture}
\end{footnotesize}
\end{center}
\caption{Two very different griddings of the permutation $564312$.}
\label{fig-epicene}
\end{figure}

Having addressed the asymptotic enumeration of geometric grid classes, we move on to their exact enumeration. As in our study of monotonically griddable classes, we are interested not only in geometric grid classes themselves, but also in their subclasses. Thus we say that the class $\C$ is \emph{geometrically griddable} if $\C\subseteq\Geom(M)$ for some finite $\zpm$ matrix $M$. Unlike the case with monotonically griddable classes, there is no known characterization of geometrically griddable classes.

With exact enumeration we have a new type of problem. While our map $\varphi^\gridded$ restricts to a bijection between $\Sigma_M^\ast/\equiv_M$ and $\Geom^\gridded(M)$, a given permutation may have many different griddings, as demonstrated in Figure~\ref{fig-epicene}. In \cite{albert:geometric-grid-:}, this issue is addressed by imposing an order on all $\Geom^\gridded(M)$ griddings of a given permutation $\pi$. Among all of these griddings, we would like to select the minimal one, called the \emph{preferred gridding}. By using a trick involving ``marking entries'' it can be shown that the set of all preferred griddings of permutations in a given geometrically griddable class is in bijection with a regular language, establishing the following result.

\begin{theorem}[Albert, Atkinson, Bouvel, Ru\v{s}kuc, and Vatter~\cite{albert:geometric-grid-:}]
\label{thm-ggc-strong-rat}
Every geometrically griddable class is strongly rational and finitely based.
\end{theorem}

In particular, every geometrically griddable class is wpo (we have rewritten history a bit here; the wpo property plays an important role in the proof of Theorem~\ref{thm-ggc-strong-rat}, and thus must be established first). By Proposition~\ref{prop-forests-are-geoms}, Theorem~\ref{thm-ggc-strong-rat} thereby generalizes one direction of Theorem~\ref{thm-mono-grid-wpo}.

Continuing in this line of research, Albert, Ru\v{s}kuc, and Vatter studied the inflations and substitution closures of geometrically griddable classes. Their main result is the following generalization of Albert and Atkinson's Theorem~\ref{thm-fin-simples-alg}.

\begin{theorem}[Albert, Ru\v{s}kuc, and Vatter~\cite{albert:inflations-of-g:}]
\label{thm-ggc-inflations}
Let $\C$ be a geometrically griddable class. Its substitution closure, $\langle\C\rangle$, is strongly algebraic. Moreover, for every strongly rational class $\U$, the inflation $\C[\U]$ is strongly rational.
\end{theorem}

It follows that all such classes are wpo. The proof of this result is well beyond the scope of this survey. Both parts of the theorem rely on the notion of ``query-complete sets of properties'' introduced by Brignall, Huczynska, and Vatter~\cite{brignall:simple-permutat:} in their generalization of Albert and Atkinson's Theorem~\ref{thm-fin-simples-alg}. In addition, the proof of the second part of the theorem relies on the work of Brignall~\cite{brignall:wreath-products:} on decompositions of permutations contained in inflations of the form $\C[\U]$ and the work of Albert, Atkinson, and Vatter~\cite{albert:subclasses-of-t:} on strongly rational classes (in particular, that the sum indecomposable permutations in a strongly rational class themselves have a rational generating function).

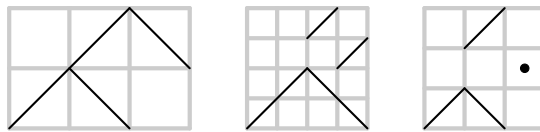
\begin{figure}
\begin{footnotesize}
$$
\begin{array}{ccccc}
	\begin{tikzpicture}[scale=0.8, baseline=(current bounding box.center)]
		\foreach \i in {0,1,2}{
			\draw [lightgray, ultra thick, line cap=round] (0,\i)--(3,\i);
		}
		\foreach \i in {0,1,2,3}{
			\draw [lightgray, ultra thick, line cap=round] (\i,0)--(\i,2);
		}
		\draw [thick, line cap=round] (0,0)--(2,2);
		\draw [thick, line cap=round] (1,1)--(2,0);
		\draw [thick, line cap=round] (2,2)--(3,1);
	\end{tikzpicture}
&&
	\begin{tikzpicture}[scale=0.4, baseline=(current bounding box.center)]
		\foreach \i in {0,1,2,3,4}{
			\draw [lightgray, ultra thick, line cap=round] (0,\i)--(4,\i);
		}
		\foreach \i in {0,1,2,3,4}{
			\draw [lightgray, ultra thick, line cap=round] (\i,0)--(\i,4);
		}
		\draw [thick, line cap=round] (0,0)--(2,2);
		\draw [thick, line cap=round] (2,2)--(4,0);
		\draw [thick, line cap=round] (2,3)--(3,4);
		\draw [thick, line cap=round] (3,2)--(4,3);
	\end{tikzpicture}
&&
	\begin{tikzpicture}[scale=0.533333333, baseline=(current bounding box.center)]
		\foreach \i in {0,1,2,3}{
			\draw [lightgray, ultra thick, line cap=round] (0,\i)--(3,\i);
		}
		\foreach \i in {0,1,2,3}{
			\draw [lightgray, ultra thick, line cap=round] (\i,0)--(\i,3);
		}
		\draw [thick, line cap=round] (0,0)--(1,1);
		\draw [thick, line cap=round] (1,1)--(2,0);
		\draw [thick, line cap=round] (1,2)--(2,3);
		\draw[fill=black] (2.5,1.5) circle (3pt); 
	\end{tikzpicture}
\end{array}
$$
\end{footnotesize}
\caption{The simple permutations in $\Av(3124, 4312)$ lie in the union of the two geometric grid classes shown on the left; the intersection of these two classes is the geometric grid class shown on the right.}
\label{fig-jay-class}
\end{figure}

Theorem~\ref{thm-ggc-inflations} greatly expands the applicability of the substitution decomposition, and has found numerous applications. One example is the work of Pantone~\cite{pantone:the-enumeration:}, who used this approach to enumerate the class $\Av(3124, 4312)$. Pantone showed that the simple permutations in this class lie in the union of the first and second geometric grid classes shown in Figure~\ref{fig-jay-class}. He then counted the class of interest by computing the intersection of these two geometric grid classes (shown on the right of this figure; here the dot is a single point), constructing regular languages in bijection with the simple permutations in all three geometric grid classes, and then applying the machinery of the substitution decomposition. Three simpler examples in the same spirit are described in Albert, Atkinson, and Vatter~\cite{albert:inflations-of-g:2x4}.

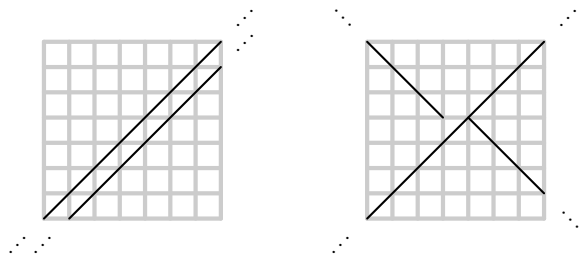
\begin{figure}
\begin{center}
\begin{footnotesize}
	\begin{tikzpicture}[scale=0.336, baseline=(current bounding box.center)]
		\foreach \i in {2,3,4,5,6,7,8,9}{
			\draw [lightgray, ultra thick, line cap=round] (2,\i)--(9,\i);
			\draw [lightgray, ultra thick, line cap=round] (\i,2)--(\i,9);
		}
		\draw [thick, line cap=round] (2,2)--(9,9);
		\draw [thick, line cap=round] (3,2)--(9,8);
		\node[rotate=45] at (2,1) {$\dots$};
		\node[rotate=45] at (1,1) {$\dots$};
		\node[rotate=45] at (10,10) {$\dots$};
		\node[rotate=45] at (10,9) {$\dots$};
	\end{tikzpicture}
\quad\quad
	\begin{tikzpicture}[scale=0.336, baseline=(current bounding box.center)]
		\foreach \i in {2,3,4,5,6,7,8,9}{
			\draw [lightgray, ultra thick, line cap=round] (2,\i)--(9,\i);
			\draw [lightgray, ultra thick, line cap=round] (\i,2)--(\i,9);
		}
		\draw [thick, line cap=round] (2,2)--(9,9);
		\draw [thick, line cap=round] (2,9)--(5,6);
		\draw [thick, line cap=round] (6,6)--(9,3);
		\node[rotate=45] at (1,1) {$\dots$};
		\node[rotate=135] at (1,10) {$\dots$};
		\node[rotate=135] at (10,2) {$\dots$};
		\node[rotate=45] at (10,10) {$\dots$};
	\end{tikzpicture}
\end{footnotesize}
\end{center}
\caption{Geometric descriptions of the $321$-avoiding permutations and the simple skew-merged permutations (up to symmetry).}
\label{fig-infinite-geom}
\end{figure}

Lately, though no general theory has been established, there has been some promising work done using geometric grid classes of \emph{infinite} matrices. Recall Waton's Proposition~\ref{prop-321-lines}, which shows that $\Av(321)$ is the class of permutations that can be drawn on two parallel lines. By taking these lines to have slope $1$ and adding the grid lines $x=i$ and $y=i$ for all integers $i$, we see that $\Av(321)$ is the infinite geometric grid class shown on the left of Figure~\ref{fig-infinite-geom}. This viewpoint, known as the \emph{staircase decomposition} has been used by Guillemot and Vialette~\cite{guillemot:pattern-matchin:} in their study of the complexity of the permutation containment problem, by Albert, Atkinson, Brignall, Ru\v{s}kuc, Smith, and West~\cite{albert:growth-rates-fo:} in their study of growth rates of classes of the form $\Av(321,\beta)$, and by Albert and Vatter~\cite{albert:generating-and-:} in their study of $321$-avoiding simple permutations. This latter work has since been refined by B\'ona, Homberger, Pantone, and Vatter~\cite{bona:pattern-avoidin:} to enumerate the involutions in the classes $\Av(1342)$ and $\Av(2341)$.

Albert and Vatter~\cite{albert:generating-and-:} also gave a geometric interpretation of the skew-merged permutations. It can be shown that every simple skew-merged permutation is, up to symmetry, an element of the infinite geometric grid class shown on the right of Figure~\ref{fig-infinite-geom}, and they used this fact to give a more structural derivation of the generating function for skew-merged permutations, originally due to Atkinson~\cite{atkinson:permutations-wh:}. Albert and Brignall~\cite{albert:enumerating-ind:} studied the class of permutations whose corresponding Schubert varieties are defined by inclusions. They showed that the simple permutations of this class lie in an infinite sequence of geometric grid classes that they called \emph{crenellations} (see Figure~\ref{fig-crenellation}), and used this insight to enumerate the class.

\index{crenellation}

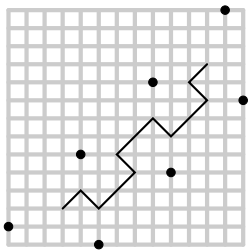
\begin{figure}
\begin{center}
	\begin{tikzpicture}[scale=0.24, baseline=(current bounding box.center)]
		\foreach \i in {0,1,2,3,4,5,6,7,8,9,10,11,12,13}{
			\draw [lightgray, ultra thick, line cap=round] (0,\i)--(13,\i);
			\draw [lightgray, ultra thick, line cap=round] (\i,0)--(\i,13);
		}
		\draw [thick, line cap=round] (3,2)--(4,3)--(5,2)--(7,4)--(6,5)--(8,7)--(9,6)--(11,8)--(10,9)--(11,10);
		\draw[fill=black] (0,1) circle (6.66666pt); 
		\draw[fill=black] (5,0) circle (6.66666pt);
		\draw[fill=black] (4,5) circle (6.66666pt);
		\draw[fill=black] (9,4) circle (6.66666pt);
		\draw[fill=black] (8,9) circle (6.66666pt);
		\draw[fill=black] (13,8) circle (6.66666pt);
		\draw[fill=black] (12,13) circle (6.66666pt);
	\end{tikzpicture}
\end{center}
\caption{A crenellation.}
\label{fig-crenellation}
\end{figure}

\subsection{Generalized grid classes}
\label{subsec-gen-grid}

\index{generalized grid class}

In a \emph{generalized grid class} the cells are allowed to contain nonmonotone permutations. Formally, let $\mathcal{M}$ be a $t\times u$ matrix of permutation classes. An $\mathcal{M}$-gridding of the permutation $\pi$ of length $n$ is a choice of column divisions $1=c_1\le\cdots\le c_{t+1}=n+1$ and row divisions $1=r_1\le\cdots\le r_{u+1}=n+1$ such that for all $i$ and $j$, $\pi([c_i,c_{i+1})\times [r_j,r_{j+1}))$ lies in the class $\mathcal{M}_{i,j}$.

The class of all permutations with $\mathcal{M}$-griddings is the grid class of $\mathcal{M}$, denoted $\Grid(\mathcal{M})$. In the context of monotone grid classes, we defined the class $\C$ to be monotonically griddable if $\C\subseteq\Grid(M)$ for some finite $\zpm$ matrix $M$. Now we define the class $\C$ to be \emph{${\G}$-griddable} if $\C\subseteq\Grid(\mathcal{M})$ for some finite matrix $\mathcal{M}$ whose entries are all subclasses of $\G$ (for the purposes of this definition, we may take them all to be equal to $\G$). From this perspective, a class is monotonically griddable if and only if it is $\left(\Av(21)\cup\Av(12)\right)$-griddable (though this fact takes a bit of thought).

As Theorem~\ref{thm-mono-griddable} characterizes monotonically griddable classes, our next result characterizes generalized grid classes.

\begin{theorem}[Vatter~\cite{vatter:small-permutati:}]
\label{thm-gridding-characterization}
The permutation class $\C$ is $\G$-griddable if and only if it does not contain arbitrarily long sums or skew sums of basis elements of $\G$, that is, if there exists a constant $m$ so that $\C$ does not contain $\beta_1\oplus\cdots\oplus\beta_m$ or $\beta_1\ominus\cdots\ominus\beta_m$ for any basis elements $\beta_1,\dots,\beta_m$ of $\G$.
\end{theorem}

We prove this result by appealing to a result of Gy{\'a}rf{\'a}s and Lehel from 1970, but a proof from first principles can be found in Vatter~\cite{vatter:small-permutati:}. One direction is again clear---if $\C$ contains arbitrarily long sums or skew sums of basis elements of $\G$, then it is not $\G$-griddable.

Now suppose that $\C$ does not contain arbitrarily long sums or skew sums of basis elements of $\G$ and consider the set
$$
\mathfrak{R}_\pi=\{\mbox{axes-parallel rectangles $R$}\st \pi(R)\notin\G\}.
$$
Thus for every $R\in\mathfrak{R}_\pi$, $\pi(R)$ contains a basis element of $\G$, so in any $\G$-gridding of $\pi$ every rectangle in $\mathfrak{R}_\pi$ must be sliced by a grid line. Clearly the converse is also true (if every rectangle in $\mathfrak{R}_\pi$ is sliced by a grid line, these grid lines define a $\G$-gridding of $\pi$), so $\C$ is $\G$-griddable if and only if there is a constant $\ell$ such that, for every $\pi\in\C$, the set $\mathfrak{R}_\pi$ can be sliced by $\ell$ horizontal or vertical lines. (This is the analogue of Proposition~\ref{prop-slice-nonmonotone-rects} for generalized grid classes.)

We say that two rectangles are {\it independent\/} if both their $x$- and $y$-axis projections are disjoint, and a set of rectangles is said to be independent if they are pairwise independent. An {\it increasing sequence\/} of rectangles is a sequence $R_1,\dots,R_m$ of independent rectangles such that $R_{i+1}$ lies above and to the right of $R_i$ for all $i\ge 0$. {\it Decreasing sequences\/} of rectangles are defined analogously. By our assumptions on $\C$, there is some constant $m$ such that for every $\pi\in\C$ the set $\mathfrak{R}_\pi$ contains neither an increasing nor a decreasing sequence of size $m$.

An independent set of rectangles corresponds to a permutation. Therefore, by the Erd\H{o}s-Szekeres Theorem, $\mathfrak{R}_\pi$ cannot contain an independent set of size greater than $(m-1)^2+1$. The proof of Theorem~\ref{thm-gridding-characterization} is therefore completed by the following result of Gy{\'a}rf{\'a}s and Lehel, which was later strengthened by K{\'a}rolyi and Tardos~\cite{karolyi:on-point-covers:}.

\begin{theorem}[Gy{\'a}rf{\'a}s and Lehel~\cite{gyarfas:a-helly-type-pr:}]
\label{rectangles-lemma}
There is a function $f(m)$ such that for any collection $\mathfrak{R}$ of axes-parallel rectangles in the plane that has no independent set of size greater than $m$, there exists a set of $f(m)$ horizontal and vertical lines that slice every rectangle in $\mathfrak{R}$.
\end{theorem}

Our next question is which permutation classes can be gridded by one of their proper subclasses. We call the class $\C$ \emph{grid irreducible} if it is not $\G$-griddable for any proper subclass $\G\subsetneq\C$. Somewhat surprisingly, this question has a concrete answer.


\begin{proposition}
\label{prop-grid-irreduce}
The class $\C$ is grid irreducible if and only if $\C=\{1\}$ or if for every $\pi\in\C$ either $\bigoplus\Sub(\pi)\subseteq\C$ or $\bigominus\Sub(\pi)\subseteq\C$.
\end{proposition}
\begin{proof}
One direction is immediate from Theorem~\ref{thm-gridding-characterization}: if $\C$ contains arbitrarily long sums or skew sums of all of its elements, then it is grid irreducible. Moreover, the only finite grid irreducible class is $\{1\}$, which takes care of that part of the other direction.

Now suppose that $\C$ is an infinite permutation class. If there is some $\pi\in\C$ such that neither $\bigoplus\Sub(\pi)$ nor $\bigominus\Sub(\pi)$ is contained in $\C$ then $\C$ is $\left(\C\cap\Av(\pi)\right)$-griddable by Theorem~\ref{thm-gridding-characterization}, and thus is not grid irreducible.
\end{proof}

As shown in the proof of this result, if neither $\bigoplus\Sub(\pi_1)$ nor $\bigominus\Sub(\pi_1)$ is contained in $\C$ then it is $\left(\C\cap\Av(\pi_1)\right)$-griddable. We may then apply this to $\C\cap\Av(\pi_1)$ to see that if neither $\bigoplus\Sub(\pi_2)$ nor $\bigominus\Sub(\pi_2)$ is contained in $\C\cap\Av(\pi_1)$, it is $\left(\C\cap\Av(\pi_1,\pi_2)\right)$-griddable. Continuing in this manner, we can construct a descending chain of classes
$$
\begin{array}{ccccccc}
\C
&\supsetneq&
\C\cap\Av(\pi_1)
&\supsetneq&
\C\cap\Av(\pi_1,\pi_2)
&\supsetneq&
\cdots
\\
\verteq&&\verteq&&\verteq
\\
\C^{(0)}
&\supsetneq&
\C^{(1)}
&\supsetneq&
\C^{(2)}
&\supsetneq&
\cdots
\end{array}
$$
such that $\C$ is $\C^{(i)}$-griddable for all $i$. Because there are infinite strictly descending chains of permutation classes, there is no guarantee that this process will terminate (see Vatter~\cite{vatter:small-permutati:} for such a construction). However, if $\C$ is wpo then it satisfies the descending chain condition by Proposition~\ref{prop-wpo-subclasses-dcc} and thus this process must stop. This proves the following.

\begin{proposition}
\label{prop-grid-irreduce-wpo}
If the class $\C$ is wpo then it is $\G$-griddable for the grid irreducible class
$$
\G=\{\pi\st \mbox{either } \bigoplus\Sub(\pi)\subseteq\C \mbox{ or } \bigominus\Sub(\pi)\subseteq\C\}.
$$
\end{proposition}

We need a final result about generalized grid classes, which shows that atomic grid irreducible classes are very constrained.


\begin{proposition}
\label{prop-atomic-grid-irreduce}
Suppose that the class $\C$ is atomic. Then it is grid irreducible if and only if it is sum or skew closed.
\end{proposition}
\begin{proof}
Theorem~\ref{thm-gridding-characterization} shows that every sum or skew closed class is grid irreducible, so it suffices to prove the reverse direction. Suppose that $\C$ is atomic, but neither sum nor skew closed. Thus there are permutations $\pi,\sigma\in\C$ such that $\bigoplus\Sub(\pi),\bigominus\Sub(\tau)\not\subseteq\C$. Because $\C$ is atomic, we can find a permutation $\tau\in\C$ containing both $\pi$ and $\sigma$. Thus $\C$ does not contain arbitrarily long sums or skew sums of $\tau$, so it is $\left(\C\cap\Av(\tau)\right)$-griddable (again by Theorem~\ref{thm-gridding-characterization}). This shows that $\C$ is not grid irreducible, completing the proof.
\end{proof}

We conclude by investigating a special case of generalized grid classes and their relation to the splittability question. We call the $2\times 1$ generalized grid class $\Grid(\D\ \E)$ the \emph{horizontal juxtaposition} of the classes $\D$ and $\E$ (the obvious symmetry of this construction is called a \emph{vertical juxtaposition}). Atkinson~\cite{atkinson:restricted-perm:} introduced juxtapositions and established the structure of their basis elements.

\index{juxtaposition}

Using our results on generalized grid classes we are able to completely characterize the classes which can be expressed as nontrivial juxtapositions. If $\C$ is not atomic then $\C\subseteq\D\cup\E$ for proper subclasses $\D,\E\subsetneq\C$, so $\C$ is also contained in the nontrivial juxtaposition $\Grid(\D\ \E)$. Suppose instead that $\C$ is atomic. If $\C$ is either sum or skew closed then it is grid irreducible by Proposition~\ref{prop-grid-irreduce} and thus not contained in a juxtaposition of proper subclasses. Otherwise, Proposition~\ref{prop-atomic-grid-irreduce} shows that $\C$ is $\D$-griddable for a proper subclass $\D\subsetneq\C$. Choose $\mathcal{M}$ to be a matrix of minimal possible size such that all of its entries are equal to $\D$ and $\C\subseteq\Grid(\mathcal{M})$. We may suppose by symmetry that $\mathcal{M}$ has at least two columns, and thus by slicing it by a vertical line we see that $\C$ is contained in the horizontal juxtaposition of two proper subclasses, establishing the following.


\begin{proposition}
\label{prop-grid-splittable}
A permutation class is contained in the juxtaposition of two proper subclasses if and only if it is neither sum nor skew closed.
\end{proposition}

In particular, if a class is contained in the juxtaposition of two proper subclasses then it is splittable. Thus the splittability question is only interesting for sum or skew closed classes that are not substitution closed.

\subsection{Small permutation classes}
\label{subsec-spc}

Our goal in this subsection is to outline the proofs of Theorem~\ref{thm-spc-gr}, which specifies all possible growth rates of small permutations classes, and \ref{thm-spc-strong-rat}, which states that all small permutation classes are strongly rational, while glossing over some of the more technical details. Adopting a perspective that is slightly historically backward, we prove Theorem~\ref{thm-spc-strong-rat} first, by showing that every small permutation class is the inflation of a geometric grid class by a strongly rational class, and is thereby strongly rational itself. We then show that if $\C$ is a small permutation class, its growth rate is either equal to $0$, $1$, or $2$, or is equal to the growth rate of the largest sum or skew closed class contained in it (thereby establishing that these classes have proper growth rates). The actual list of possible growth rates provided in Theorem~\ref{thm-spc-gr} can then be produced by characterizing all possible enumerations of sum indecomposable permutations in classes with growth rates less than $\kappa$, though we do not go into that rather lengthy undertaking here.

For the moment, let $\gamma$ be an arbitrary real number and suppose we would like to grid all permutation classes of growth rate less than $\gamma$.  That is, we would like to find a single \emph{cell class} $\G$ such that every permutation of (upper) growth rate less than $\gamma$ is $\G$-griddable.

To this end, define
$$
\G_\gamma=\{\pi\st \mbox{either $\gr\left(\bigoplus\Sub(\pi)\right)<\gamma$ or $\gr\left(\bigominus\Sub(\pi)\right)<\gamma$}\}.
$$
If $\pi\in\G_\gamma$, then at least one of $\bigoplus\Sub(\pi)$ or $\bigominus\Sub(\pi)$ has growth rate less than $\gamma$, so in order to grid all permutation classes of growth rate less than $\gamma$ we must have $\pi$ in our cell class. Our next result shows that this cell class can indeed be used to grid all classes of growth rate less than $\gamma$.

\begin{proposition}
\label{prop-G-gamma}
For every real number $\gamma$, if the permutation class $\C$ satisfies $\ugr(\C)<\gamma$ then it is $\G_\gamma$-griddable.
\end{proposition}
\begin{proof}
Suppose to the contrary that $\C$ satisfies $\ugr(\C)<\gamma$ but is not $\G_\gamma$-griddable, and let $f$ denote the generating function of $\C$. Now fix an arbitrary integer $m$.  We seek to derive a contradiction by showing that $f(1/\gamma)>m$, which, because $m$ is arbitrary, will imply that $\ugr(\C)\ge\gamma$. By Theorem~\ref{thm-gridding-characterization} and symmetry, we may assume that there is a permutation of the form $\beta_1\oplus\cdots\oplus\beta_m$ contained in $\C$ where each $\beta_i$ is a basis element of $\G_\gamma$. If $\gr(\bigoplus\beta_i)$ were less than $\gamma$ then $\beta_i$ would lie in $\G_\gamma$, so we know that $\gr(\bigoplus\beta_i)\ge\gamma$ for every index $i$.

Let $s_i$ denote the generating function for the nonempty sum indecomposable permutations contained in (or equal to) $\beta_i$, so that the generating function for $\bigoplus\beta_i$ is given by $1/(1-s_i)$.  These generating functions have positive singularities that are less than or equal to $1/\gamma$, and because each $s_i$ is a polynomial this singularity must be a pole.  Moreover, each $s_i$ has positive coefficients, which implies that the unique positive solution to $s_i(x)=1$ occurs for $x\le1/\gamma$.  In particular, $s_i(1/\gamma)\ge 1$ for all indices $i$.

Now consider the set of permutations of the form $\alpha_1\oplus\cdots\oplus\alpha_k$ for some $k\le m$, where for each $i$, $\alpha_i$ is a nonempty sum indecomposable permutation contained in $\beta_i$.  Clearly this is a subset (but likely not a subclass) of $\C$.  Moreover, the generating function for this set of permutations is
$$
s_1+s_1s_2+\cdots+s_1\cdots s_m,
$$
which is at least $m$ when evaluated at $1/\gamma$. This implies that the generating function for $\C$ is also at least $m$ when evaluated at $1/\gamma$, completing the contradiction.
\end{proof}

Proposition~\ref{prop-G-gamma}, which was not used in the original proof of Theorem~\ref{thm-spc-gr}, gives us an easily computable membership test to determine whether $\pi\in\G_\gamma$. The small permutation classes are all $\G_{\kappa-\epsilon}$-griddable for some $\epsilon>0$ so we begin by establishing some restrictions on $\G_\kappa$.

The reader may have noticed that both infinite antichains we have encountered in this survey (in Figures~\ref{fig-inc-inc-basis} and \ref{fig-infinite-antichain}) consist of variations on a single theme. While that particular coincidence should not be taken too far (there are many more infinite antichains based on different motifs, for example, all those lying in monotone grid classes with cycles), the theme of those two antichains plays a significant role in the characterization of small permutation classes. We define the \emph{increasing oscillating sequence} as
$$
4,1,6,3,8,5,\dots,2k+2,2k-1,\dots.
$$
(This sequence, itself listed as A065164 in the OEIS~\cite{sloane:the-on-line-enc:}, also arises in the study of juggling and genomics, see Buhler, Eisenbud, Graham, and Wright~\cite{buhler:juggling-drops-:} and Pevzner~\cite{pevzner:computational-m:}, respectively.)

\index{increasing oscillating sequence}

An  {\it increasing oscillation\/} is defined to be any sum indecomposable permutation that is order isomorphic to a subsequence of the increasing oscillating sequence. It is easily seen that there are only two increasing oscillations of each length, which are inverses of each other. Finally, a \emph{decreasing oscillation} is the reverse of an increasing oscillation, and an \emph{oscillation} is either an increasing or a decreasing oscillation.

Let $\mathcal{O}$ denote the downward closure of the set of all oscillations,
$$
\mathcal{O}=\Sub(\text{increasing and decreasing oscillations of all lengths}).
$$
It can be shown that their substitution closure $\langle\mathcal{O}\rangle$ has a finite basis (Vatter~\cite[Proposition A.2]{vatter:small-permutati:}). From this, it can then be computed that for every basis element $\beta$ of $\langle\mathcal{O}\rangle$, both $\bigoplus\Sub(\beta)$ and $\bigominus\Sub(\beta)$ have growth rates greater than $2.24$. Therefore Proposition~\ref{prop-G-gamma} immediately gives us the following.

\begin{proposition}
\label{prop-G-gamma-osc}
The cell class $\G_\kappa$ is contained in $\langle\mathcal{O}\rangle$.
\end{proposition}

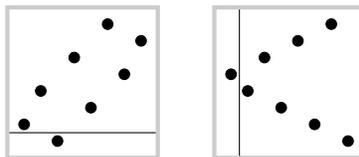
\begin{figure}
\begin{center}
\begin{footnotesize}
	\begin{tikzpicture}[scale=0.2222222222222, baseline=(current bounding box.center)]
		\draw (0,1.5)--(9,1.5);
		\draw [lightgray, ultra thick, line cap=round] (0,0) rectangle (9,9);
		\draw[fill=black] (1,2) circle (9pt);
		\draw[fill=black] (2,4) circle (9pt);
		\draw[fill=black] (3,1) circle (9pt);
		\draw[fill=black] (4,6) circle (9pt);
		\draw[fill=black] (5,3) circle (9pt);
		\draw[fill=black] (6,8) circle (9pt);
		\draw[fill=black] (7,5) circle (9pt);
		\draw[fill=black] (8,7) circle (9pt);
	\end{tikzpicture}
\quad\quad
	\begin{tikzpicture}[scale=0.2222222222222, baseline=(current bounding box.center)]
		\draw (1.5,0)--(1.5,9);
		\draw [lightgray, ultra thick, line cap=round] (0,0) rectangle (9,9);
		\draw[fill=black] (1,5) circle (9pt);
		\draw[fill=black] (2,4) circle (9pt);
		\draw[fill=black] (3,6) circle (9pt);
		\draw[fill=black] (4,3) circle (9pt);
		\draw[fill=black] (5,7) circle (9pt);
		\draw[fill=black] (6,2) circle (9pt);
		\draw[fill=black] (7,8) circle (9pt);
		\draw[fill=black] (8,1) circle (9pt);
	\end{tikzpicture}
\end{footnotesize}
\end{center}
\caption{Two permutations that show why the hypotheses of Theorem~\ref{thm-gen-grid-bdd-alts} are necessary On the left is a long increasing oscillation, while on the right is a permutation with large substitution depth, each with an extra grid line. Both situations require a large number of additional grid lines in order to grid them into independent rectangles.}
\label{fig-chop-problems}
\end{figure}

This gives us a place to start, but we need to get some control on the structure of small permutation classes. In particular, we need an analogue of Proposition~\ref{prop-chop-gridding} that will allow us to ``chop'' the griddings of these classes. The direct analogue of Proposition~\ref{prop-chop-gridding} is below, while Figure~\ref{fig-chop-problems} shows why we must impose these hypotheses.

\begin{theorem}[Vatter~{\cite[Theorem 5.4]{vatter:small-permutati:}}]
\label{thm-gen-grid-bdd-alts}
Suppose that the cell class $\G$ contains only finitely many simple permutations and has bounded substitution depth. If the class $\C$ is $\G$-griddable and the length of alternations in $\C$ is bounded, then $\C\subseteq\Grid(\mathcal{M})$ for a matrix $\mathcal{M}$ in which every nonempty entry is a subclass of $\G$ and no two nonempty entries share a row or a column.
\end{theorem}

Clearly Theorem~\ref{thm-gen-grid-bdd-alts} is not enough for our purposes, since small permutation classes may contain arbitrarily long alternations. Still, the examples of Figure~\ref{fig-chop-problems} demonstrate obstructions to chopping any sort of gridding, and the actual chopping result used for small permutation classes (Theorem~5.4 of \cite{vatter:small-permutati:}) has precisely the same conditions on $\G$. Thus we need to verify that our cell classes $\G_{\kappa-\epsilon}$ have finitely many simple permutations and bounded substitution depth.

We tackle substitution depth first. It can be shown (Proposition~4.2 of \cite{vatter:small-permutati:}) that every permutation of substitution depth at least $8n$ contains a wedge alternation of length at least $n$. Furthermore, using the membership test provided by Proposition~\ref{prop-G-gamma} one can prove that for every $\gamma<1+\varphi\approx 2.62$, the cell class $\G_\gamma$ does not contain arbitrarily long wedge alternations. Thus there is some constant $d$ such that every permutation in $\G_\kappa$ has substitution depth at most $d$.

It is not difficult to establish that the growth rate of $\mathcal{O}$ is precisely $\kappa$. Moreover, using the fact that $\mathcal{O}$ is the union of two posets (corresponding to the increasing and decreasing oscillations) neither of which has an antichain containing more than three elements, it follows that for every $\epsilon>0$ we have $\G_{\kappa-\epsilon}\subseteq \langle\mathcal{O}^{k}\rangle$ for some $k$, where
$$
\mathcal{O}^{k}=\Sub(\mbox{oscillations of length at most $k$}).
$$

With these two computations handled, we can then chop the griddings of small permutation classes to obtain the following result.

\begin{theorem}[Vatter~{\cite[Theorem 5.4]{vatter:small-permutati:}}]
\label{thm-spc-gridding}
Suppose that $\ugr(\C)<\kappa-\epsilon$ for some $\epsilon>0$. Then $\C\subseteq\Grid(\mathcal{M})$ for a finite matrix $\mathcal{M}$ such that
\begin{enumerate}
\item[(1)] every entry is equal to $\G_{\kappa-\epsilon}$, the class of monotone permutations, or is empty,
\item[(2)] every nonempty entry that shares a row or column with another nonempty entry is equal to the class of monotone permutations, and
\item[(3)] no nonempty entry shares a row or column with more than one other nonempty entry.
\end{enumerate}
\end{theorem}

Moreover, we have seen above that $\G_{\kappa-\epsilon}$ is a subclass of $\langle\mathcal{O}^{k}\rangle$ and contains only permutations of substitution depth at most $d$. This gives us the following.

\begin{proposition}
\label{prop-cell-class-strong-rat}
For every $\epsilon>0$, the cell class $\G_{\kappa-\epsilon}$ is strongly rational.
\end{proposition}
\begin{proof}
Recall that by the second part of Theorem~\ref{thm-ggc-strong-rat}, the inflation of a geometrically griddable class by a strongly rational class is itself strongly rational. Choose $d$ so that every permutation in $\G_{\kappa-\epsilon}$ has substitution depth at most $d$. Thus each of these permutations is either the inflation of a simple permutation by permutations of substitution depth at most $d-1$ or the sum or skew sum of such permutations. Thus if we define
$$
\tilde{\mathcal{O}}^{k}=\mathcal{O}^{k}\cup\Av(21)\cup\Av(12),
$$
we see that
$$
\G_{\kappa-\epsilon}
\subseteq
\underbrace{\tilde{\mathcal{O}}^{k}[\tilde{\mathcal{O}}^{k}[\cdots]]}_{\mbox{$d$ copies of $\tilde{\mathcal{O}}^{k}$}}.
$$
Because $\tilde{\mathcal{O}}^{k}$ is the union of a finite class with the class of all monotone permutations, it is geometrically griddable, and thus the strong rationality of $\G_{\kappa-\epsilon}$ follows by iteratively applying the second part of Theorem~\ref{thm-ggc-strong-rat}.
\end{proof}

It is possible to show the stronger result that $\G_\kappa$ is strongly rational, but this requires more work, and Proposition~\ref{prop-cell-class-strong-rat} is enough for our purposes. The first of these purposes is to prove Theorem~\ref{thm-spc-strong-rat}, which states that small permutation classes are strongly rational. By Theorem~\ref{thm-spc-gridding}, every small permutation class is contained in $\Grid(M)[\G_{\kappa-\epsilon}]$ for a matrix $M$ in which no nonzero entry shares a row or column with more than one other nonzero entry. It follows that $M$ is a forest so $\Grid(M)=\Geom(M)$ by Proposition~\ref{prop-forests-are-geoms}. The result now follows by the second part of Theorem~\ref{thm-ggc-inflations}.

Our last goal is the characterization of growth rates of small permutation classes (Theorem~\ref{thm-spc-gr}). We observed in Proposition~\ref{prop-grid-component} that the upper growth rate of $\Grid(M)$ is equal to the greatest upper growth rate of the monotone grid class of a connected component of $M$. A similar argument holds for generalized grid classes and their subclasses. Thus given a class $\C\subseteq\Grid(\mathcal{M})$, its upper growth rate is equal to the upper growth rate of the restriction of $\C$ to a connected component of $\mathcal{M}$. We can now sketch the reduction to sum closed classes.

\begin{theorem}\label{gr-atomic-grid-irreducible}
If the permutation class $\C$ satisfies $\ugr(\C)<\kappa$ then $\gr(\C)$ exists and is equal to $0$, $1$, $2$, or the growth rate of a subclass that is either sum or skew closed.
\end{theorem}
\begin{proof}
Suppose that $\ugr(\C)<\kappa$ and set $\C^0=\C$. By our previous work, we can find a cell class $\G^0\subseteq\C\cap\G_{\kappa-\epsilon}$ for some $\epsilon>0$ such that $\C$ is $\G^0$-griddable. Next we can find a matrix $\mathcal{M}^0$ with all entries equal to subclasses of $\G^0$ and satisfying the conclusions of Theorem~\ref{thm-spc-gridding} such that $\C^0\subseteq\Grid(\mathcal{M}^0)$. By our remarks above, the upper growth rate of $\C^0$ is the upper growth rate of the restriction of $\C^0$ to a connected component of $\mathcal{M}^0$. This connected component is either a single cell (in which case we get a subclass of $\G^0$) or a pair of monotone cells. It can be shown that all subclasses of a $1\times 2$ monotone grid class have growth rates $0$, $1$, or $2$, so we are done if the latter situation holds. Thus we may assume the former situation holds and choose $\C^1\subseteq\G^0$ such that $\ugr(\C^0)=\ugr(\C^1)$.

Importantly, Proposition~\ref{prop-cell-class-strong-rat} shows that $\C^1$ is strongly rational, and thus wpo. By Proposition~\ref{prop-wpo-atomic-gr}, this implies that the upper growth rate of $\C^1$ is equal to that of one of its atomic subclasses, say $\mathcal{A}^1\subseteq\C^1$. By Proposition~\ref{prop-grid-irreduce-wpo}, we may now choose a grid irreducible class $\G^1$ such that $\mathcal{A}^1$ is $\G^1$-griddable. Then we choose a matrix $\mathcal{M}^1$ with all entries equal to subclasses of $\G^1$ which satisfies the conclusions of Theorem~\ref{thm-spc-gridding} and with $\mathcal{A}^1\subseteq\Grid(\mathcal{M}^1)$. The upper growth rate of $\mathcal{A}^1$ is then equal to that of a restriction of $\mathcal{A}^1$ to a connected component of $\mathcal{M}^1$. We are done as before if this component consists of two cells, so we may assume that $\ugr(\mathcal{A}^1)=\ugr(\C^2)$ for a subclass $\C^2\subseteq\G^1$.

By repeating this process indefinitely, we either find that the upper growth rate of $\C$ is equal to $0$, $1$, or $2$, or we construct an infinite descending chain
$$
\C=\C^0\supseteq\G^0\supseteq\C^1\supseteq\mathcal{A}^1\supseteq\G^1\supseteq\C^2\supseteq\mathcal{A}^2\supseteq\G^2\supseteq\C^3\supseteq\cdots,
$$
all with identical upper growth rates. Moreover, because $\G^0$ is wpo, Proposition~\ref{prop-wpo-subclasses-dcc} shows that this chain must terminate. Thus there is some $i$ such that $\C^i=\mathcal{A}^i=\G^i$. This implies that the class $\C^i$ is both atomic and grid irreducible, and thus it is either sum or skew closed by Proposition~\ref{prop-atomic-grid-irreduce}, proving the theorem.
\end{proof}

Thus we have reduced the characterization of growth rates of small permutation classes to the characterization of growth rates of small \emph{sum closed} permutation classes, as promised. From this point Theorem~\ref{thm-spc-gr} follows after a fairly technical analysis of sum indecomposable permutations.

\bigskip
\textbf{Acknowledgements.}
This chapter has greatly benefited by the comments, corrections, and suggestions of the referee as well as those of
Michael Albert,
David Bevan,
Jonathan Bloom,
Robert Brignall,
Cheyne Homberger,
V\'{\i}t Jel{\'{\i}}nek,
and
Jay Pantone.

\bibliographystyle{acm}
\bibliography{../../../refs}

\end{document}